\documentclass{amsart}%
\usepackage{amssymb,amsmath,xcolor,graphicx,xspace,ragged2e}
\usepackage{amsfonts}
\usepackage{amsmath}
\usepackage{amssymb}
\usepackage[unicode=true]{hyperref}
\usepackage{graphicx}%
\providecommand{\U}[1]{\protect\rule{.1in}{.1in}}
\graphicspath{{2019mmpafa4_graphics/}{2019mmpafa4_tcache/}{2019mmpafa4_gcache/}}
\DeclareGraphicsExtensions{.pdf,.eps,.ps,.png,.jpg,.jpeg}
\graphicspath{{2019mmpafa2_graphics/}{2019mmpafa2_tcache/}{2019mmpafa2_gcache/}}
\DeclareGraphicsExtensions{.pdf,.eps,.ps,.png,.jpg,.jpeg}
\providecommand{\U}[1]{\protect\rule{.1in}{.1in}}
\newtheorem{theorem}{Theorem}
\theoremstyle{plain}

\newtheorem{definition}{Definition}
\newtheorem{example}{Example}

\newtheorem{lemma}{Lemma}

\newtheorem{problem}{Problem}
\newtheorem{proposition}{Proposition}
\newtheorem{remark}{Remark}

\numberwithin{equation}{section}
\begin{document}
\title[Extrapolation]{Extrapolation: Stories and Problems}
\author{Sergey Astashkin}
\address{Samara National Research University, Moskovskoye shosse 34, 443086, Samara, Russia}
\email{astash56@mail.ru}
\author{Mario Milman}
\address{Instituto Argentino de Matematica \\
Buenos Aires, Argentina}
\email{mario.milman@icloud.com}
\urladdr{https://sites.google.com/site/mariomilman/}

\begin{abstract}
We discuss some aspects of Extrapolation Theory. The presentation includes
many examples and open problems.

\end{abstract}
\maketitle
\tableofcontents

\section{Introduction}

Apparently the first result concerning extrapolation of operator inequalities
was formulated by Yano \cite{yano}, although special cases had been considered
earlier in papers by Marcinkiewicz and Titchmarsh (cf. \cite{zyg} for proofs
and precise references). Yano's result concerns with operators acting on
$L^{p}$ spaces defined on finite measure spaces\footnote{Although it has a
straightforward extension to infinite measure spaces (cf. (\ref{infmeasure})
below, for example).}, where $p$ belongs to a fixed open interval
$(p_{0},p_{1}),$ and the rate of blow up of the norm inequalities, as$\ p$
approaches the end points of the interval $(p_{0},p_{1})$, is prescribed.
Yano's theorem was stated before the foundations of the general theory of
interpolation were developed in the sixties and seventies. Moreover, during
this so called \textquotedblleft golden era of interpolation\textquotedblright%
\ most researchers were busy..interpolating, and Yano's theorem remained an
isolated result for a long time. As a matter of fact, despite its obvious
elementary nature, and the many concrete questions that it left open, further
developments had to await for the general theory of Extrapolation that was
only initiated in the 1980's (cf. \cite{jm1}, \cite{10}, \cite{11},
\cite{M94}). The abstract theory of extrapolation not only substantially
extends and improves Yano's original result but it provides a general
framework, and a powerful machinery, to extract information from inequalities
that decay in a specified way in scales of Banach spaces.

On the occasion of this special issue of Pure and Applied Functional Analysis
devoted to Extrapolation Theory\footnote{We refer to the Preface of this
Special Issue for a brief discussion of Extrapolation and its connection with
Interpolation.} it seemed to us that it would be important to include a paper
collecting in some organized fashion open problems, in order to promote more
activity in this area. Indeed, it is probably fair to say that the vitality of
a field at any given time can be gauged in terms of the quality, and the
number, of its open problems.

Let us now say a few words about our intended audience, and how it has
influenced the choices, as well as the presentation, of topics we shall
discuss. We generally expect that our reader will have some familiarity with
the rudiments of interpolation theory as developed, for example, in the first
few chapters of \cite{bs} or \cite{BL}, to which we refer for background
information and notation. Moreover, since there is a close connection between
Extrapolation and Interpolation, we decided to organize the exposition
exploiting the familiarity of the reader with Interpolation. In short, we have
tried to frame our selection of problems by means of comparing familiar
results of interpolation theory with their (possible) counterparts in
extrapolation, often providing informal background explanations and, as much
as possible, including explicit examples and calculations. These discussions
correspond to what we refer to as *stories* in the title of our paper, which
we have supplemented with Appendices that contain supporting material in order
to facilitate the reading. We hope that these choices will make it easier for
newcomers to profitably read the paper, while at the same time, it is also our
expectation that experts could also find something of interest in the material
as well.

In our effort to streamline the presentation we were led to introduce new
concepts and notation that hopefully will clarify the connections and
facilitate the formulation of new problems. In particular, there are two
natural ways to look at extrapolation methods: either as mechanisms that
reverse the process of interpolation or, alternatively, as interpolation
processes that involve scales of interpolation spaces, rather than pairs of
spaces. The latter is the point of view that we have largely adopted for our
presentation in this paper. The unifying concepts that we introduce here are
those of \textquotedblleft$K$ and $J-$functionals for a scale of
spaces\textquotedblright\ (cf. Section \ref{sec:J&M}, Definition
\ref{def:scaleKfunctional} and Definition \ref{def:scaleJfunctional}).

The penalty that we have paid is that we produced a narrower set of problems
and far more effort was spent on the explanations and motivations than we had
originally envisioned. To somewhat mitigate these concerns we now provide a
super brief set of references to some of the topics that we have not
considered in this paper but have recently received extensive treatments in
the literature, and refer the reader to these works for further references. In
particular, we mention a number of recent papers devoted to abstract
extrapolation methods (cf. \cite{AL2017}, \cite{AL}); the connections between
extrapolation and the so called limiting interpolation spaces was recently
explored in detail in \cite{ALM}, \cite{asly}. Applications of extrapolation
to embeddings of function spaces and other topics in Harmonic Analysis and
Approximation are well known, for a recent account see \cite{dom}. Some
aspects of the theory of extrapolation, as it applies to non-commutative
$L^{p}$ spaces, has been treated in a paper appearing in this issue (cf.
\cite{suk}), where more references in this direction can be found. We also
refer the reader\footnote{Further sources include the dedicated web page
\par
https://sites.google.com/site/mariomilman/home/extrapolation-links \newline
that contains a listing of many relevant publications on Extrapolation
Theory.} to the other articles in this issue for further information on
extrapolation, other potential sets of problems, and inspiration.

Finally, producing a list of problems creates, well, problems of its own. For
example, one would probably need to exclude some topics and thus some
important problems could end up not being mentioned, one could also create,
unintentionally, the impression that some questions are more important than
others. It is also in the nature of developing such lists that some of the
problems will turn out to be easier to solve than others...Indeed, it will be
immediately clear that our list covers only a small sample of possible topics
and that many important issues are not even mentioned. For these, and other
possible shortcomings, we must apologize in advance and stress that, in
particular, any omissions (or inclusions) do not represent necessarily a value
judgement on our part\footnote{Solve at your own peril!} but simply reflect
our own taste and limitations. Moreover, although we have systematically
documented all the topics discussed therein, it was not our intention to
present a comprehensive bibliography. To supplement our references we refer
the reader to the bibliographies of the papers that we do reference, as well
as the collection of papers included in this volume, and, once again,
apologize in advance\footnote{To add one more story, it is perhaps appropriate
to quote Barry Simon here. In his talk \textquotedblleft Tales of our
forefathers\textquotedblright\ http://www.math.caltech.edu/simon/biblio.html
he quotes H. Steinhaus (via a story by Steinhaus' student Mark Kac):
\textquotedblleft\textit{The acceptance of your work by the mathematical
public goes through three phases: First, they say it's wrong. Then, they say
it's trivial. Finally, they say I did it first}.\textquotedblright} if your
favorite papers are not included in our somewhat limited bibliography.

We should issue one more warning concerning the presentation. In order to try
to make the paper more user friendly we tried\footnote{Tried is the operative
word here.} to avoid getting bogged down with another long formal paper and
decided to adopt the somewhat more informal style of a conference or lecture
presentation. This has resulted in a paper that is not linearly ordered,
contains some repetitions, and may require the reader to jump around
topics\footnote{Here we were inspired by Cort\'azar's novel Hopscotch
\cite{cortazar}, which can be described as an open-ended novel, or antinovel;
where the reader is invited to rearrange the different parts of the novel
according to a plan prescribed by the author. For the Cort\'azar fans we add
that some version of \textquotedblleft extrapolaci\'on\textquotedblright\ plays
also a r\^{o}le in the novel.}, and even though we provide some road
maps\footnote{It is possible to read the text without jumping around and in
such case read no further. Otherwise we suggest: Section 2$\rightarrow
3\rightarrow4\rightarrow6\rightarrow8....$}, the reader is invited to
rearrange the order and do her/his own jumping decisions.

\section{Improving Yano's theorem\label{sec:improve}}

In order to illustrate the general methods of extrapolation we start by
showing how the modern ideas led to a substantial improvement of Yano's
theorem. Since the classical result deals with $L^{p}$ spaces on finite
measure spaces\footnote{The main effect of this assumption for us is that in
this case we work with an *ordered* pair (i.e. $L^{\infty}\subset L^{1}$) and
therefore a special reiteration formula is available that makes some
calculations somewhat easier.}, \textbf{in this section} we shall work with
spaces of functions defined on fixed finite measure space that, without loss
of generality, we shall take to be the interval $[0,1]$ with Lebesgue
measure$.$ In particular, we shall let\textbf{ }$L^{p}:=L^{p}[0,1],$ etc.

\begin{theorem}
\label{teo:primero}(Yano \cite{yano}, \cite[Ch.~12, Theorem~4.41]{zyg}) Let
$\alpha>0.$ Then,

(i) Suppose that $T$ is a linear operator with values in the set of measurable
functions on $[0,1],$ such that $T$ is bounded on $L^{p},$ for all
$p\in(1,p_{0}),$ $p_{0}>1$, and $\left\Vert T\right\Vert _{L^{p}\rightarrow
L^{p}}\leq\frac{c}{(p-1)^{\alpha}},$ as $p\rightarrow1,$ with a constant $c>0$
independent of $p$. Then, $T$ can be extended to be a bounded operator
\begin{equation}
T:L(LogL)^{\alpha}\rightarrow L^{1}, \label{y1}%
\end{equation}
where the Zygmund space $L(LogL)^{\alpha}:=L(LogL)^{\alpha}[0,1]$ consists of
all functions $f$ that are measurable on $[0,1]$ and, moreover, such that
\[
\Vert f\Vert_{L(LogL)^{\alpha}}:=\int_{0}^{1}f^{\ast}(t)\log^{\alpha
}(e/t)\,dt<\infty
\]
($f^{\ast}$ is the decreasing rearrangement of the function $|f|$, see e.g.
\cite[Ch.~II]{KPS}).

(ii) Suppose that $T$ is a bounded linear operator on $L^{p},$ for all $p>1,$
and such that for some constant $c>0,$ $\left\Vert T\right\Vert _{L^{p}%
\rightarrow L^{p}}\leq cp^{\alpha},$ as $p\rightarrow\infty$. Then, $T$ is a
bounded operator,
\begin{equation}
T:\,L^{\infty}\rightarrow e^{L^{1/\alpha}}, \label{y2}%
\end{equation}
where the Zygmund space $e^{L^{1/\alpha}}:$ =$e^{L^{1/\alpha}}[0,1]$ consists
of all measurable functions $f(t)$ on $[0,1]$ such that
\[
\Vert f\Vert_{e^{L^{1/\alpha}}}:=\sup_{0<t\leq1}(f^{\ast}(t)\log^{-\alpha
}(e/t))<\infty.
\]

\end{theorem}

Yano's result is a simple consequence of the modern theory of extrapolation.
Indeed, to see (i) we apply the $\sum$ extrapolation functor\footnote{We refer
to Appendix \ref{apendixA:2}\ \ for a discussion on the $\sum$ and $\Delta$
methods of extrapolation.} of \cite{jm} to get%
\begin{equation}
T:\sum_{p>1}\frac{L^{p}}{(p-1)^{\alpha}}\rightarrow\sum_{p>1}L^{p},
\label{yano2}%
\end{equation}
and (\ref{y1}) follows from the known calculations (cf. \cite{jm} and the
discussions below)
\[
\sum_{p>1}\left(  \frac{L^{p}}{(p-1)^{\alpha}}\right)  =L(LogL)^{\alpha
},\text{ \ \ \ \ \ \ }\sum_{p>1}L^{p}=L^{1}.
\]
Likewise, if (ii) holds, then applying the ${\large \Delta}-$functor yields%
\[
T:{\large \Delta}_{p>1}\left(  L^{p}\right)  \rightarrow{\large \Delta}%
_{p>1}(p^{-\alpha}L^{p}).
\]
Consequently, (\ref{y2}) follows since (cf. \cite{jm})%
\[
\Delta_{p>1}\left(  L^{p}\right)  =L^{\infty};\text{ \ \ \ \ \ \ \ \ }%
\Delta_{p>1}(p^{-\alpha}L^{p})=e^{L^{1/\alpha}}.
\]

\begin{example}
The following generalization of Theorem \ref{teo:primero} holds (cf.
\cite{jm}). Let $\vec{A}=(A_{0},A_{1}),$ $\vec{B}=(B_{0},B_{1})$ be Banach
pairs, and let $\alpha>0.$ Then,

(i) Suppose that $T$ is a bounded linear operator\footnote{We refer to
Appendix \ref{apendixA:1}\ \ for notation and background on the real method of
interpolation.}, $T:$ $(A_{0},A_{1})_{\theta,1;J}^{\blacktriangleleft
}\rightarrow(B_{0},B_{1})_{\theta,\infty;K}^{\blacktriangleleft},$ for all
$\theta\in(0,1),$ and such that there exists a constant $c>0,$ so that
\linebreak$\left\Vert T\right\Vert _{(A_{0},A_{1})_{\theta,1;J}%
^{\blacktriangleleft}\rightarrow(B_{0},B_{1})_{\theta,\infty;K}%
^{\blacktriangleleft}}\leq{c}{\theta^{\alpha}}.$ Then, $T$ can be extended to
be a bounded operator%
\[
T:\sum\limits_{\theta}\frac{1}{\theta^{\alpha}}\vec{A}_{\theta,1;J}%
^{\blacktriangleleft}\rightarrow\sum\limits_{\theta}\vec{B}_{\theta,\infty
;K}^{\blacktriangleleft}.
\]
Moreover,
\[
\sum\limits_{\theta}\frac{1}{\theta^{\alpha}}\vec{A}_{\theta,q;J}%
^{\blacktriangleleft}=\{f:\left\Vert f\right\Vert _{\sum\limits_{\theta}%
\frac{1}{\theta}\vec{A}_{\theta,q;J}^{\blacktriangleleft}}=\int_{0}%
^{1}K(s,f;\vec{A})(\log\frac{1}{s})^{\alpha-1}\frac{ds}{s}<\infty\},
\]%
\[
\sum\limits_{\theta}\vec{B}_{\theta,\infty;K}^{\blacktriangleleft}=B_{0}%
+B_{1}.
\]

(ii) Suppose that $T$ is a bounded linear operator, $T:$ $(A_{0}%
,A_{1})_{\theta,q(\theta);K}^{\blacktriangleleft}\rightarrow(B_{0}%
,B_{1})_{\theta,\infty;K}^{\blacktriangleleft},$ for all $\theta\in(0,1),$ and
such that there exists a constant $c>0,$ so that \linebreak$\left\Vert
T\right\Vert _{(A_{0},A_{1})_{\theta,1;K}^{\blacktriangleleft}\rightarrow
(B_{0},B_{1})_{\theta,\infty;K}^{\blacktriangleleft}}\leq\frac{c}%
{(1-\theta)^{\alpha}}.$ Then, $T$ is a bounded operator%
\[
T:{\large \Delta}((A_{0},A_{1})_{\theta,q(\theta);K}^{\blacktriangleleft
})\rightarrow{\large \Delta}((1-\theta)^{\alpha}\vec{B}_{\theta,\infty
;K}^{\blacktriangleleft}).
\]
Moreover,%
\[
{\large \Delta}((1-\theta)^{\alpha}\vec{B}_{\theta,\infty;K}%
^{\blacktriangleleft})=\{f:\left\Vert f\right\Vert _{\Delta((1-\theta
)^{\alpha}\vec{B}_{\theta,\infty;K}^{\blacktriangleleft})}=\sup_{0<t<1}%
(1+\log\frac{1}{t})^{-\alpha}\frac{K(t,f;\vec{B})}{t}<\infty\}
\]%
\[
A_{0}\cap A_{1}\subset{\large \Delta}((A_{0},A_{1})_{\theta,q(\theta
);K}^{\blacktriangleleft}).
\]

\end{example}

\begin{remark}
In particular, the previous Example shows that the conclusions of Yano's
Theorem hold under weaker assumptions: We can replace \textquotedblleft strong
type $(p,p)"$ by \textquotedblleft weak type $(p,p)$" (cf. Section
\ref{sec:discutido} below for full details).
\end{remark}

\begin{remark}
\label{rem:reitera}It follows from the previous example that Yano's theorem
holds for $L^{p}$-spaces based on infinite measure, if we give a proper
interpretation to (\ref{y1}) and (\ref{y2}). For example, suppose that the
underlying measure space is $(0,\infty)$ with Lebesgue measure. Then
(\ref{y1}) should now read%
\[
T:L(LogL)^{\alpha}(0,\infty)+L^{\infty}(0,\infty)\rightarrow L^{1}%
(0,\infty)+L^{\infty}(0,\infty).
\]
The proof is the same: we apply the $\sum-$functor but now we need to recall
that (cf. \cite{jm}, see also \cite{ALM})%
\begin{align}
\sum_{p>1}\left(  \frac{p^{\alpha}}{(p-1)^{\alpha}}L^{p}(0,\infty)\right)   &
=L(LogL)^{\alpha}(0,\infty)+L^{\infty}(0,\infty)\label{infmeasure}\\
\sum_{p>1}L^{p}(0,\infty)  &  =L^{1}(0,\infty)+L^{\infty}(0,\infty).
\end{align}

\end{remark}

\begin{example}
Consider the Sobolev pair $\vec{A}=(W_{L^{1}}^{k}(\mathbb{R}^{n}\mathbb{)}%
$,$W_{L^{\infty}}^{k}(\mathbb{R}^{n}\mathbb{))}$, where $k\in\mathbb{N},$ and
for a function space $X(\mathbb{R}^{n}\mathbb{)}$, the corresponding Sobolev
space is defined using the norm%
\[
\left\Vert f\right\Vert _{W_{X}^{k}}=\sum\limits_{\left\vert j\right\vert \leq
k}\left\Vert D^{j}f\right\Vert _{X}.
\]
Then a classical computation due to DeVore and Scherer yields (cf.
\cite{bs}),
\begin{equation}
K(t,f;\vec{A})\approx\sum\limits_{\left\vert j\right\vert \leq k}%
K(t,D^{j}f,L^{1},L^{\infty}). \label{formula}%
\end{equation}
It follows that $\vec{A}_{1/p^{\prime},p;K}^{\blacktriangleleft}=W_{L^{p}}%
^{k}(\mathbb{R}^{n}\mathbb{)}$, and therefore,%
\[
\sum\limits_{p>1}\frac{1}{(1/p^{\prime})^{\alpha}}\vec{A}_{1/p^{\prime}%
,p;K}^{\blacktriangleleft}=W_{L(LogL)^{\alpha}}^{k}(\mathbb{R}^{n}%
\mathbb{)+}W_{L^{\infty}}^{k}(\mathbb{R}^{n}\mathbb{)}\text{.}%
\]
Let $\Omega$ be a Lipschitz domain of finite measure, then the pair $\vec
{A}=(W_{L^{1}}^{k}(\Omega)$,$\mathring{W}_{L^{\infty}}^{k}(\Omega))$ is
ordered, and the corresponding version of (\ref{formula}) holds (cf.
\cite{bs}, \cite{cpc}), yielding (cf. \cite{jm})%
\[
\sum\limits_{p>1}\frac{1}{(1/p^{\prime})^{\alpha}}\vec{A}_{1/p^{\prime}%
,p;K}^{\blacktriangleleft}=W_{L(LogL)^{\alpha}}^{k}(\Omega)\text{.}%
\]

\end{example}

\begin{example}
We refer to \cite[Chapter 5, Section 6]{bs} for details and notation. Let $H$
be the Hilbert transform and let $N$ be the nontangential maximal function
(cf. \cite[(6.2), page 363]{bs}). Let $H(L^{1})$ be the Hardy space for
$L^{1}(\mathbb{R}),$ $H(L^{1})=\{f:f\in L^{1}(\mathbb{R})$ and $Hf\in
L^{1}(\mathbb{R})\},$ with
\[
\left\Vert f\right\Vert _{\operatorname{Re}H^{1}(\mathbb{R})}=\left\Vert
f\right\Vert _{L^{1}(\mathbb{R})}+\left\Vert Hf\right\Vert _{L^{1}%
(\mathbb{R})}.
\]
Let $\vec{A}$ be the pair $(H(L^{1})(\mathbb{R}),L^{\infty}(\mathbb{R}))$.
Then%
\[
K(t,f;\vec{A})\approx\int_{0}^{t}N(f)^{\ast}(s)ds.
\]
Therefore, for $\alpha>0,$%
\begin{align*}
\left\Vert f\right\Vert _{\sum\limits_{p>1}\frac{1}{(1/p^{\prime})^{\alpha}%
}\vec{A}_{1/p^{\prime},p;K}^{\blacktriangleleft}}  &  \approx\int_{0}%
^{1}N(f)^{\ast\ast}(s)(\log\frac{1}{s})^{\alpha-1}ds\\
&  \approx\left\Vert N(f)\right\Vert _{L(LogL)^{\alpha}+L^{\infty}}.
\end{align*}

\end{example}

\begin{example}
Many other computations of extrapolation spaces and the corresponding
extrapolation theorems can be read off from the known classical computations
of $K$-functionals (cf. \cite{bs}, \cite{nil}, \cite{ov}, \cite{tor} and the
references therein). For example, if we consider $\vec{A}=(L^{p}%
(\mathbb{R}^{n}\mathbb{)}$,$\mathring{W}_{L^{p}}^{k}(\mathbb{R}^{n}%
\mathbb{))}$, $1\leq p\leq\infty,$ then (cf. \cite[page 341]{bs}, and also
\cite{karamixi})
\begin{align*}
K(t^{k},f;\vec{A})  &  \approx\omega_{p}^{k}(t,f):=\sup_{|h|\leq t}\Vert
\Delta_{h}^{k}f\Vert_{L^{p}}\\
&  \approx\Big\{t^{-n}\int_{|h|\leq t}\Vert\Delta_{h}^{k}f\Vert_{p}%
^{p}dh\Big\}^{1/p},\;p\in\lbrack1,\infty]
\end{align*}
where $\Delta_{h}^{k}$ denotes the $k$-th difference operator defined
recursively by
\[
\Delta_{h}f(x)=\Delta_{h}^{1}f(x)=f(x+h)-f(x),\text{ }\Delta_{h}^{k}%
=\Delta_{h}^{1}\Delta_{h}^{k-1}.
\]
Therefore, for $0<s<1,$ $k\in\mathbb{N}$, the Besov space $\mathcal{B}%
_{p,q}^{sk}:=(L^{p}(\mathbb{R}^{n}\mathbb{)}$,$\mathring{W}_{Lp}%
^{k}(\mathbb{R}^{n}\mathbb{))}_{s,q;K}^{\blacktriangleleft}$ is defined by the
condition%
\begin{equation}
\Vert f\Vert_{\mathcal{B}_{p,q}^{sk}}\approx\lbrack(1-s)sq]^{1/q}k\left(
\int_{0}^{\infty}t^{-skq}\{t^{-n}\int_{|h|\leq t}\Vert\Delta_{h}^{k}f\Vert
_{p}^{p}dh\}^{q/p}\frac{dt}{t}\right)  ^{1/q}<\infty, \label{def1}%
\end{equation}
with the usual conventions if $q=\infty.$

It \ follows that for $\alpha>0,$
\[
\left\Vert f\right\Vert _{\sum\limits_{s}\frac{\mathcal{B}_{p,q}^{sk}%
}{s^{\alpha}}}\approx\int_{0}^{1}\omega_{p}^{k}(t,f)(\log\frac{1}{t}%
)^{\alpha-1}\frac{dt}{t}.
\]
In particular, for $\alpha=1,$ one obtains the Dini type of spaces
$\mathring{B}^{0,1,1}$ that were examined and applied to study the mixing
properties of vector fields by Bianchini \cite[see (1.16) and (4.6)]{bi}. In
fact, these spaces are also useful in Harmonic Analysis (cf. \cite{fe},
\cite{GM} and the references therein).
\end{example}

\begin{example}
\label{non-commutative} Suppose that $\mathcal{H}$ is a separable complex
Hilbert space. Recall that the \textit{Schatten-von Neumann class}
${{\mathfrak{S}}}^{p}$ consists of all compact operators $T:\;\mathcal{H}%
\rightarrow\mathcal{H}$ such that
\[
\Vert T\Vert_{{{\mathfrak{S}}}^{p}}:=\left(  \sum_{j=1}^{\infty}s_{j}%
(T)^{p}\right)  ^{\frac{1}{p}}<\infty,
\]
where $\{s_{j}(T)\}_{j=1}^{\infty}$ is the non-increasing sequence of
$s$-numbers of $T$ determined by the Schmidt expansion (cf. \cite{GK}). The
Schatten-von Neumann classes belong to the larger family of symmetrically
normed ideals. In particular, the so-called Matsaev operator ideals and their
duals, have been singled out and studied as suitable end point ideals for the
scale of Schatten-von Neumann classes ${{\mathfrak{S}}}^{p}$.

Let $\alpha>0$. The \textit{Matsaev ideal} $\mathcal{M}^{\alpha}$ is the ideal
of compact operators in a Hilbert space $\mathcal{H}$, endowed with the norm
\[
\Vert T\Vert _{\mathcal{M}^{\alpha }}:=\sum_{j=1}^{n}\frac{\log ^{\alpha
-1}(ej)s_{j}(T)}{j}\text{.}
\]
Similarly, the dual ideal to $\mathcal{M}^{\alpha},$ $\mathcal{M}^{\alpha
,\ast},$ consists of all compact operators $T$ such that
\[
\Vert T\Vert_{\mathcal{M}^{\alpha,\ast}}:=\underset{n\in\mathbb{N}}{\sup}%
\frac{\sum_{j=1}^{n}s_{j}(T)}{\log^{\alpha}(en)}<\infty\text{.}%
\]
By the well-known equivalence (see e.g. \cite{Konig})
\[
K(t,T;{{\mathfrak{S}}}^{1},{{\mathfrak{S}}}^{\infty})\approx\sum_{j=1}%
^{[t]}s_{j}(T),\text{ \ (where }[t]\text{ is the integer part of }t\text{),}%
\]
and straightforward direct calculations, we have
\[
({{\mathfrak{S}}}^{\infty},{{\mathfrak{S}}}^{1})_{1/p,p;K}^{\blacktriangleleft
}=({{\mathfrak{S}}}^{1},{{\mathfrak{S}}}^{\infty})_{1/p^{\prime}%
,p;K}^{\blacktriangleleft}={{\mathfrak{S}}}^{p}.
\]
Hence, as above, for each $p_{0}>1$
\[
\mathcal{M}^{\alpha}=\sum_{p>p_{0}}p({{\mathfrak{S}}}^{\infty},{{\mathfrak{S}%
}}^{1})_{1/p,p;K}^{\blacktriangleleft}=\sum_{p>p_{0}}p{{\mathfrak{S}}}^{p}%
\]
and
\[
\mathcal{M}^{\alpha,\ast}=\Delta_{1<p<p_{0}}((p-1)^{\alpha}({{\mathfrak{S}}%
}^{\infty},{{\mathfrak{S}}}^{1})_{1/p,p;K}^{\blacktriangleleft})=\Delta
_{1<p<p_{0}}((p-1)^{\alpha}{{\mathfrak{S}}}^{p}),
\]
(see \cite[5.2]{M94}). These relations give a version of Yano's theorem in the
non-commutative setting, where the ideals $\mathcal{M}^{\alpha}$ and
$\mathcal{M}^{\alpha,\ast}$ play the role of the $L(LogL)^{\alpha}$ and
$e^{L^{1/\alpha}}$ spaces (cf. \cite[Theorem~42]{M94}).
\end{example}

The key to prove these results is that we can reduce the computation of the
$\sum$- and $\Delta$-functors as follows. Let $M$ and $N$ be
tempered\footnote{We say that $M$ is tempered if $M(2\theta)\approx
M(\theta),$ when $\theta$ is close to zero, and $M(1-2(1-\theta))\approx
M(\theta),$ when $\theta$ is close to $1.$} weights (cf. \cite{jm}), then
\begin{equation}
\sum\limits_{\theta}{\small M(\theta)\vec{A}}_{\theta,q(\theta);K}%
^{\blacktriangleleft}{\small =}\sum\limits_{\theta}{\small M(\theta)\vec{A}%
}_{\theta,1;J}^{\blacktriangleleft}{\small ,}\text{ }{\large \Delta}_{\theta
}{\small (N(\theta)\vec{A}}_{\theta,q(\theta);K}^{\blacktriangleleft
}{\small )=}{\large \Delta}_{\theta}{\small (N(\theta)\vec{A}}_{\theta
,\infty;K}^{\blacktriangleleft}{\small ).} \label{validity}%
\end{equation}

With this reduction at hand we can compute the $\sum$- and ${\large \Delta}%
$-functors via their Fubini type properties described in the next two
theorems. Before we go to the statement and proof of these results we recall
the concept of characteristic function of an interpolation functor \cite{jan},
which we shall use freely in what follows.

\begin{definition}
\label{def:cara}Let $F$ be an interpolation functor\footnote{cf. Appendix
\ref{apendixA:2}.}, then the characteristic function $\rho$ of $F$ satisfies
(cf. \cite{jan}, \cite[(2.7), page 11]{jm})%
\[
F(\mathbb{C},\frac{1}{t}\mathbb{C})=\frac{1}{\rho(t)}\mathbb{C},\;\;t>0.
\]

\end{definition}

\begin{theorem}
\label{sumformula}(Jawerth-Milman \cite[Theorem 3.1 (i), page 20]{jm}) Let
$\vec{A}$ be a Banach pair, and let $\{\rho_{\theta}\}_{\theta\in(0,1)}$ be a
family of quasi-concave functions. Suppose that $\rho(t)=\sup_{\theta\in
(0,1)}\rho_{\theta}(t)$ is finite at one point (and hence at all points). Then%
\[
\sum\limits_{\theta}\left(  \vec{A}_{\rho_{\theta},1;J}\right)  =\vec{A}%
_{\rho,1;J,}\text{ with }\left\Vert f\right\Vert _{\sum\limits_{\theta}\left(
\vec{A}_{\rho_{\theta},1;J}\right)  }=\left\Vert f\right\Vert _{\vec{A}%
_{\rho,1;J}},
\]
where for a quasi-concave function $\tau,$ we let $\vec{A}_{\tau,1;J}$ be the
space of elements that can be represented by%
\[
f=\int_{0}^{\infty}u(s)\frac{ds}{s},
\]
where $u:(0,\infty)\rightarrow\Delta(\vec{A})$ is strongly measurable and such
that
\[
\int_{0}^{\infty}\frac{J(s,u(s);\vec{A})}{\tau(s)}\frac{ds}{s}<\infty,
\]
with%
\[
\left\Vert f\right\Vert _{\vec{A}_{\tau,1;J}}=\inf_{x=\int_{0}^{\infty
}u(s)\frac{ds}{s}}\{\int_{0}^{\infty}\frac{J(s,u(s);\vec{A})}{\tau(s)}%
\frac{ds}{s}\}.
\]

\end{theorem}

\begin{proof}
We repeat with full details\footnote{In their younger days, before changing
their ways, Jawerth-Milman had adopted the terse writing style of \cite{BL}.}
the elementary argument given in \cite{jm} since it illustrates why the
$(1,J)$ functor \textquotedblleft commutes\textquotedblright\ with $\sum$ (Fubini).

By definition $\rho(t)\geq\rho_{\theta},$ for all $\theta\in(0,1),$ therefore,
it is easy to see that $\left\Vert f\right\Vert _{\vec{A}_{\rho,1;J}}%
\leq\left\Vert f\right\Vert _{\vec{A}_{\rho_{\theta},1;J}}.$ Indeed, if
$f\in\vec{A}_{\rho_{\theta},1;J},$ then given $\varepsilon>0,$ we can select a
representation $f=\int_{0}^{\infty}u(s)\frac{ds}{s}$ such that,
\[
\int_{0}^{\infty}\frac{J(s,u(s);\vec{A})}{\rho_{\theta}(s)}\frac{ds}{s}%
\leq(1+\varepsilon)\left\Vert f\right\Vert _{\vec{A}_{\rho_{\theta},1;J}}.
\]
Therefore,%
\[
\left\Vert f\right\Vert _{\vec{A}_{\rho,1;J}}\leq\int_{0}^{\infty}%
\frac{J(s,u(s);\vec{A})}{\rho(s)}\frac{ds}{s}\leq\int_{0}^{\infty}%
\frac{J(s,u(s);\vec{A})}{\rho_{\theta}(s)}\frac{ds}{s}\leq(1+\varepsilon
)\left\Vert f\right\Vert _{\vec{A}_{\rho_{\theta},1;J}}.
\]
Now, letting $\varepsilon\rightarrow0,$ we find that for all $\theta\in(0,1),$%
\[
\left\Vert f\right\Vert _{\vec{A}_{\rho,1;J}}\leq\left\Vert f\right\Vert
_{\vec{A}_{\rho_{\theta},1;J}}.
\]
The previous inequality can be now extended to all of $\sum\limits_{\theta
}\left(  \vec{A}_{\rho_{\theta},1;J}\right)  .$ Indeed, let $f\in
\sum\limits_{\theta}\left(  \vec{A}_{\rho_{\theta},1;J}\right)  .$ Select a
decomposition $f=\sum\limits_{\theta}f_{\theta}$ such
that\footnote{Informally, we need to use an epsilon argument, etc.}
$\left\Vert f\right\Vert _{\sum\limits_{\theta}\left(  \vec{A}_{\rho_{\theta
},1;J}\right)  }\approx\sum\limits_{\theta}\left\Vert f_{\theta}\right\Vert
_{\left\Vert f\right\Vert _{\vec{A}_{\rho_{\theta},1;J}}}.$ Then,%
\begin{align*}
\left\Vert f\right\Vert _{\vec{A}_{\rho,1;J}}  &  =\left\Vert \sum
\limits_{\theta}f_{\theta}\right\Vert _{\vec{A}_{\rho,1;J}}\\
&  \leq\sum\limits_{\theta}\left\Vert f_{\theta}\right\Vert _{\vec{A}%
_{\rho,1;J}}\\
&  \leq\sum\limits_{\theta}\left\Vert f_{\theta}\right\Vert _{\vec{A}%
_{\rho_{\theta},1;J}}\\
&  \preceq\left\Vert f\right\Vert _{\sum\limits_{\theta}\left(  \vec{A}%
_{\rho_{\theta},1;J}\right)  }.
\end{align*}
To show the converse inequality let us first observe that for each $f\in
\Delta(\vec{A}),\theta\in(0,1),t>0,$%
\begin{align*}
\left\Vert f\right\Vert _{\sum\limits_{\theta}\left(  \vec{A}_{\rho_{\theta
},1;J}\right)  }  &  \leq\left\Vert f\right\Vert _{\vec{A}_{\rho_{\theta}%
,1;J}}\\
&  \leq\frac{J(t,f;\vec{A})}{\rho_{\theta}(t)}%
\end{align*}
(see \cite[Theorem 3.11.2 (4), page 64]{BL}). Therefore,
\begin{align*}
\left\Vert f\right\Vert _{\sum\limits_{\theta}\left(  \vec{A}_{\rho_{\theta
},1;J}\right)  }  &  \leq\inf_{\theta}\{\frac{1}{\rho_{\theta}(t)}%
\}J(t,f;\vec{A})\\
&  \leq\frac{1}{\rho(t)}J(t,f;\vec{A}).
\end{align*}
Let $f\in\vec{A}_{\rho,1;J},\varepsilon>0,$ and select a decomposition
$f=\int_{0}^{\infty}u(s)\frac{ds}{s}$ such that
\[
\int_{0}^{\infty}\frac{J(s,u(s);\vec{A})}{\rho(s)}\frac{ds}{s}\leq
(1+\varepsilon)\left\Vert f\right\Vert _{\vec{A}_{\rho,1;J}}.
\]
We have
\begin{align*}
\left\Vert f\right\Vert _{\sum\limits_{\theta}\left(  \vec{A}_{\rho_{\theta
},1;J}\right)  }  &  \leq\left\Vert \int_{0}^{\infty}u(s)\frac{ds}%
{s}\right\Vert _{\sum\limits_{\theta}\left(  \vec{A}_{\rho_{\theta}%
,1;J}\right)  }\\
&  \leq\int_{0}^{\infty}\left\Vert u(s)\right\Vert _{\sum\limits_{\theta
}\left(  \vec{A}_{\rho_{\theta},1;J}\right)  }\frac{ds}{s}\\
&  \leq\int_{0}^{\infty}\frac{1}{\rho(s)}J(s,u(s);\vec{A})\frac{ds}{s}\\
&  \leq(1+\varepsilon)\left\Vert f\right\Vert _{\vec{A}_{\rho,1;J}}.
\end{align*}
Now we can safely let $\varepsilon\rightarrow0$ to conclude the proof.
\end{proof}

Likewise, and even easier \textquotedblleft Fubini argument\textquotedblright%
\ but this time using the $L^{\infty}$-norm (informally: \textquotedblleft sup
commutes with sup\textquotedblright) shows that (cf. \cite[Theorem 3.1 (ii),
page 20-21]{jm})

\begin{theorem}
\label{intersectformula}Suppose that $\rho^{\ast}(t)=\inf_{\alpha\in
(0,1)}\{\rho_{\alpha}(t)\}$ is non-zero at a point, then%
\[
\Delta_{\alpha\in(0,1)}\left(  \vec{A}_{\rho_{\alpha},\infty;K}\right)
=\vec{A}_{\rho^{\ast},\infty;K},
\]
where for a quasi-concave function $\tau,$%
\[
\vec{A}_{\tau,\infty;K}=\{f:\left\Vert x\right\Vert _{\vec{A}_{\tau,\infty;K}%
}:=\sup_{s>0}\{\frac{K(s,f;\vec{A})}{\tau(s)}\}<\infty\}.
\]

\end{theorem}

\begin{proof}
Since for all indices $\alpha\in(0,1),$ and for all $t>0,$ $\rho^{\ast}%
(t)\leq\rho_{\alpha}(t),$ we readily see that%
\[
\left\Vert f\right\Vert _{\Delta\left(  \vec{A}_{\rho_{\alpha},\infty
;K}\right)  }=\sup_{\alpha\in(0,1)}\sup_{s>0}\frac{K(s,f;\vec{A})}%
{\rho_{\alpha}(s)}\leq\sup_{s>0}\frac{K(s,f;\vec{A})}{\rho^{\ast}%
(s)}=\left\Vert f\right\Vert _{\vec{A}_{\rho^{\ast},\infty;K}}.
\]
Now, for all $t>0,$ we have
\begin{align*}
K(t,f;\vec{A})  &  \leq\rho_{\alpha}(t)\left\Vert f\right\Vert _{\vec{A}%
_{\rho_{\alpha},\infty;K}}\\
&  \leq\rho_{\alpha}(t)\sup_{\alpha\in(0,1)}\left\Vert f\right\Vert _{\vec
{A}_{\rho_{\alpha},\infty;K}}\\
&  =\rho_{\alpha}(t)\left\Vert f\right\Vert _{\Delta\left(  \vec{A}%
_{\rho_{\alpha},\infty;K}\right)  }.
\end{align*}
Thus, for all indices $\alpha,$ and for all $t>0,$%
\[
\frac{K(t,f;\vec{A})}{\rho^{\ast}(t)}\leq\frac{\rho_{\alpha}(t)}{\rho^{\ast
}(t)}\left\Vert f\right\Vert _{\Delta\left(  \vec{A}_{\rho_{\alpha},\infty
;K}\right)  }.
\]
Taking the infimum over all indices $\alpha,$ we get%
\begin{align*}
\frac{K(t,f;\vec{A})}{\rho^{\ast}(t)}  &  \leq\left\Vert f\right\Vert
_{\Delta\left(  \vec{A}_{\rho_{\alpha},\infty;K}\right)  }\inf_{0<\alpha
<1}\{\frac{\rho_{\alpha}(t)}{\rho^{\ast}(t)}\}\\
&  =\left\Vert f\right\Vert _{\Delta\left(  \vec{A}_{\rho_{\alpha},\infty
;K}\right)  },
\end{align*}
and the result follows taking the supremum over all $t>0.$
\end{proof}

\begin{problem}
\label{prob:multiplier}(\textbf{The multiplier problem I}) Give a complete
characterization of the weights $M(\theta)$ for which the formulae
(\ref{validity}) holds. More generally, let $\{\rho_{\theta}\}$ be a family of
concave functions and consider interpolation functors $\{F_{\rho_{\theta}}\}$
with characteristic functions $\rho_{\theta}.$ We ask to characterize the
weights $M$ and $N$ such that, for all Banach pairs $\vec{A},$ we have
\begin{equation}
\sum_{\theta}{\small M(\theta)\vec{A}}_{\rho_{\theta},1;J}^{\blacktriangleleft
}{\small =}\sum_{\theta}{\small M(\theta)F}_{\rho_{\theta}}{\small (\vec
{A}),\Delta(N(\theta)\vec{A}}_{\rho_{\theta},\infty;K}^{\blacktriangleleft
}{\small )=\Delta(N(\theta)F}_{\rho_{\theta}}{\small (\vec{A})).} \label{foto}%
\end{equation}
The same question for more general extrapolation functors (cf. Section
\ref{apendixA:2}), e.g. the $\Sigma_{p}$-methods (cf. \cite[(2.6), page
10]{jm} and \cite{karami}), the $\Delta_{p}$-methods (cf. \cite{karami}), the
extrapolation methods of Astashkin-Lykov (cf. \cite{AL}, \cite{AL2009}, etc.).
\end{problem}

\begin{remark}
In connection to Problem \ref{prob:multiplier} we should like to mention some
cases where progress has been made. Consider the pair $(L^{1}(0,1),L^{\infty
}(0,1))$, then it is shown (cf. \cite[Theorem 3.5]{AL2017}) that%
\begin{equation}
\Delta_{1<p<\infty}(\omega(p)(L^{1},L^{\infty})_{1/p^{\prime},\infty
;K}^{\blacktriangleleft})=\Delta_{1<p<\infty}(\omega(p)(L^{1},L^{\infty
})_{1/p^{\prime},p;K}^{\blacktriangleleft}), \label{eq1}%
\end{equation}
holds for weights of the form $\omega(p)=\psi(e^{-p})$, where $\psi$ is an
increasing positive function on $[0,1]$ such that for some $C>0,$
$\psi(et)\leq C\psi(t)$, $0<t\leq1/e$. It is worth to note that, in general,
these weights fail to be tempered. Indeed, it is easy to see that the function
$\omega(1/(1-\theta))$ is tempered at $0$ if and only if $\omega
(2p)\approx\omega(p)$, when $p$ is sufficiently large. Therefore, this is
equivalent to the condition $\psi(t)\leq C\psi(t^{2})$ for $0<t\leq1$.
Moreover, formula \eqref{eq1} holds even for weights decreasing at a much
faster rate at infinity, e.g. weights of the form $\omega(p)=\psi(e^{-e^{p}}%
)$. Once again we refer to \cite{AL2017}. Likewise, one can ask for a
characterization of the weights $\omega(p)$ such that we have
\begin{equation}
\sum_{1<p<\infty}(\omega(p)(L^{1},L^{\infty})_{1/p^{\prime},1;J}%
^{\blacktriangleleft})=\sum_{1<p<\infty}(\omega(p)(L^{1},L^{\infty
})_{1/p^{\prime},p;J}^{\blacktriangleleft}). \label{eq2}%
\end{equation}
It is known (cf. \cite[Theorem 3]{LykovMathSb}) that formula \eqref{eq2} holds
for weights $\omega$ of the form $\omega(p)=\psi(p/(p-1))$, where
$\psi:\,[1,\infty)\rightarrow\lbrack1,\infty)$ is an increasing positive
function such that for some $C>0,$ $\psi(x+e^{-x})\leq C\psi(x)$, $x\geq1$.
Note that in this example $\omega(p)$ is tempered at $1$ if and only if there
exists $C>0,$ such that $\psi(2x)\leq C\psi(x)$ for sufficiently large $x$.
Therefore, \eqref{eq2} can be valid for a rather wide class of non-tempered weights.
\end{remark}

\begin{problem}
Formulate suitable versions of Theorems \ref{sumformula} and
\ref{intersectformula} for other extrapolation functors (cf. the previous Problem).
\end{problem}

\begin{problem}
(Open Ended) There is a natural duality associated to Theorems
\ref{sumformula} and \ref{intersectformula} (cf. \cite{jm}) but more generally
the r\^{o}le of duality in extrapolation theory has not been studied systematically.
\end{problem}

\section{Yano's theorem for weak type operators\label{sec:discutido}}

The general method to prove Yano's theorem indicated in Section
\ref{sec:improve} shows, as a bonus, that we can extrapolate replacing strong
type by weak type. Indeed, this is a direct consequence of (\ref{validity}).
We develop this point in detail. Once again the underlying measure space will
be $(0,1)$ with Lebesgue measure, but we now assume that $T$ satisfies
\begin{equation}
T:L(p,1)\rightarrow L(p,\infty),\text{ with }\left\Vert T\right\Vert
_{L(p,1)\rightarrow L(p,\infty)}\leq\frac{1}{(p-1)^{\alpha}},\text{ for
}1<p<p_{0}. \label{yano6}%
\end{equation}
Before we go on we remark that we need to be fastidious about how we define
the norms of the Lorentz spaces. We shall let\footnote{It is easy to give
$L(p,1)$ a more familiar norm, as we now indicate. For this purpose we we may
restrict ourselves, without loss of generality, to functions $f$ that are
integrable with compact support. Then, integrating by parts we find%
\begin{align*}
\frac{1}{pp^{\prime}}\int_{0}^{\infty}f^{\ast\ast}(s)s^{1/p}\frac{ds}{s}  &
=\frac{1}{p}\int_{0}^{\infty}f^{\ast}(s)s^{1/p}\frac{ds}{s}\\
&  =\int_{0}^{\infty}f^{\ast}(s)ds^{1/p}.
\end{align*}
}%
\begin{equation}
L(p,1)=(L^{1},L^{\infty})_{1/p^{\prime},1;K}^{\blacktriangleleft
}=\{f:\left\Vert f\right\Vert _{L(p,1)}=\frac{1}{pp^{\prime}}\int_{0}^{\infty
}f^{\ast\ast}(s)s^{1/p}\frac{ds}{s}<\infty\}, \label{yano7}%
\end{equation}%
\begin{equation}
L(p,\infty)=\{f:\left\Vert f\right\Vert _{L(p,\infty)}=\sup_{t>0}\{f^{\ast
\ast}(s)s^{1/p}\}=(L^{1},L^{\infty})_{1/p^{\prime},\infty;K}%
^{\blacktriangleleft}. \label{yano8}%
\end{equation}
The reason we use (\ref{yano7}) to define the $L(p,1)$-spaces is that we can
then apply directly the strong form of the fundamental Lemma (cf. \cite{CJM90}
and Section \ref{sec:sfl} below) to obtain\footnote{A more general result is
discussed in detail in Example \ref{exa:sfl} below.} that, with constants of
equivalence independent of $p,$%
\begin{equation}
L(p,1)=(L^{1},L^{\infty})_{1/p^{\prime},1;K}^{\blacktriangleleft}%
=(L^{1},L^{\infty})_{1/p^{\prime},1;J}^{\blacktriangleleft}. \label{yano5}%
\end{equation}
At this point we see that%
\[
T:\sum_{1<p<p_{0}}\frac{(L^{1},L^{\infty})_{1/p^{\prime},1;J}%
^{\blacktriangleleft}}{(p-1)^{\alpha}}\rightarrow\sum_{1<p<p_{0}}%
(L^{1},L^{\infty})_{1/p^{\prime},\infty;K}^{\blacktriangleleft}=\sum
_{1<p<p_{0}}L(p,\infty).
\]
Now using Theorem \ref{sumformula} and the reiteration formula (cf. \cite{jm}
and the recent extensive discussion in \cite{ALM})%
\[
\sum_{1<p<p_{0}}\frac{(L^{1},L^{\infty})_{1/p^{\prime},1;J}%
^{\blacktriangleleft}}{(p-1)^{\alpha}}=\sum_{1<p<\infty}\frac{(L^{1}%
,L^{\infty})_{1/p^{\prime},1;J}^{\blacktriangleleft}}{(p-1)^{\alpha}},
\]
we have%
\[
\sum_{1<p<p_{0}}\frac{(L^{1},L^{\infty})_{1/p^{\prime},1;J}%
^{\blacktriangleleft}}{(p-1)^{\alpha}}=L(LogL)^{\alpha},\text{ }%
\sum_{1<p<p_{0}}L(p,\infty)=L^{1}.
\]

We can deal in a similar fashion with the second part of Yano's theorem.

\subsection{Limiting Spaces}

To describe extrapolation spaces like the $L(LogL)^{\alpha}(0,1)$ spaces it is
useful to introduce variants of the usual interpolation constructions. Here we
consider the simplest such constructions, the $\langle\vec{X}\rangle_{w,q;K}$
spaces (cf. Appendix~\ref{apendixA:1} below, moreover, we refer to \cite{ALM}
for a comprehensive discussion of limiting spaces that can be described as
extrapolation spaces). Let $w:(0,1)\rightarrow(0,\infty)$ and define
\[
\langle\vec{X}\rangle_{w,q;K}=\{f\in X_{0}+X_{1}:\Vert x\Vert_{\langle\vec
{X}\rangle_{w,q;K}}:=\{\int_{0}^{1}(w(s)K(s,f;\vec{X}))^{q}\frac{ds}%
{s}\}^{1/q}<\infty\}.
\]
For example, if $w_{\theta}(s)=s^{-\theta},0\leq\theta\leq1,$ (cf.
\eqref{dela})
\[
\langle\vec{X}\rangle_{w_{\theta},q;K}:=\langle\vec{X}\rangle_{\theta
,q;K}=\{f\in X_{0}+X_{1}:\Vert x\Vert_{\langle\vec{X}\rangle_{\theta,q;K}%
}:=\Phi_{\theta,q}(\chi_{(0,1)}K(s,f;\vec{X}))<\infty\},
\]
Then, for every $p_{0}>1$ and $\alpha>0,$%
\begin{align}
\sum_{1<p<p_{0}}\frac{p^{\alpha}(L^{1},L^{\infty})_{1/p^{\prime}%
,1;J}^{\blacktriangleleft}}{(p-1)^{\alpha}}  &  =\{f:\int_{0}^{1}f^{\ast\ast
}(s)\left(  \log\frac{1}{s}\right)  ^{\alpha-1}ds<\infty\}\nonumber\\
&  =L(LogL)^{\alpha}. \label{llog}%
\end{align}
Thus,%
\[
\sum_{1<p<p_{0}}\frac{p^{\alpha}(L^{1},L^{\infty})_{1/p^{\prime}%
,1;J}^{\blacktriangleleft}}{(p-1)^{\alpha}}=\langle L^{1},L^{\infty}%
\rangle_{\left(  \log\frac{1}{s}\right)  ^{\alpha-1},1:K},
\]
and in the special case $\alpha=1,$%
\[
\sum_{1<p<p_{0}}\frac{p(L^{1},L^{\infty})_{1/p^{\prime},1;J}%
^{\blacktriangleleft}}{p-1}=\langle L^{1},L^{\infty}\rangle_{w_{0}%
,1:K}:=\langle L^{1},L^{\infty}\rangle_{0,1:K}.
\]

\begin{problem}
The characterization of $L(LogL)^{\alpha}$ as an extrapolation space for the
$\sum$-method given by (\ref{llog}) leads to the following reiteration
formulae, which for simplicity we formulate for $\alpha=1:$%
\[
\langle L^{1},L^{\infty}\rangle_{0,1;K}=\langle L^{1},L^{p}\rangle
_{0,1;K}=\langle L^{1},\langle L^{1},L^{\infty}\rangle_{1/p^{\prime}%
,1;K}\rangle_{0,1;K}.
\]
More generally, the formulae can be stated for ordered Banach pairs. We ask,
for conditions on $\kappa,\mu,$ quasi-concave functions for the validity of%
\[
\langle L^{1},L^{\infty}\rangle_{\mu,1;K}=\langle L^{1},\langle L^{1}%
,L^{\infty}\rangle_{\kappa,1;K}\rangle_{\mu,1;K}.
\]

\end{problem}

\section{Jawerth-Milman meet Calder\'{o}n: K and J functionals for scales of
spaces and their r\^{o}le in Extrapolation\label{sec:J&M}}

The results of the previous sections point to a connection between
Extrapolation and Calder\'{o}n's theory of weak type interpolation. This will
be the topic of our discussion in this section. We have reorganized the
results of \cite{jm} by means of introducing the concept of $K-$functional for
scales of interpolation spaces. This approach to the results of \cite{jm}
shows more clearly the connection with Calder\'{o}n's theory, leads to a
cleaner presentation and makes it easier to formulate Problems.

Let us recall that Calder\'{o}n \cite{cal} developed methods to characterize
weak type interpolation inequalities via rearrangement inequalities and he
used it with great success to characterize the corresponding interpolation
spaces (cf. \cite{bs}, \cite{BK}). It is somewhat less well known that
Calder\'{o}n also understood that one could formulate the results through the
use of $K-$functionals, as was explicitly displayed in the thesis of his Ph.D.
student E. Oklander (\cite{Oklander}, \cite{oklan}). In a nutshell, the idea,
expressed in modern terminology, is simply a consequence of
\begin{equation}
T:\vec{X}\rightarrow\vec{Y}\text{ with norm }M\Leftrightarrow\text{ for all
}t>0,K(t,Tf;\vec{Y})\leq MK(t,f;\vec{X}). \label{char0}%
\end{equation}
In particular, when dealing with weak type interpolation the abstract
formulation can be made very explicit. Indeed, if we let $\vec{X}%
=(L(p_{0},1),L(p_{1},1))$ and $\vec{Y}=(L(q_{0},\infty),L(q_{1},\infty)),$ the
corresponding $K-$functionals are known. In fact, this is result was already
contained in Oklander \cite[Theorem 3, page 51; and Theorem 4, page
52]{Oklander}, \cite{oklan}, where these $K-$functionals are computed exactly,
and was also done by Sharpley \cite[Lemma 6.8, page 504]{shar}, in the
slightly more general context of general Lorentz spaces\footnote{For a concave
function $\phi$ let%
\[
\Lambda_{\phi}=\{f:\left\Vert f\right\Vert _{\Lambda_{\phi}}=\int_{0}^{\infty
}f^{\ast}(s)d\phi(s)<\infty\},
\]
then if $\phi_{1},\phi_{2}$ are concave functions we have
\[
K(t,f;\Lambda_{\phi_{1}},\Lambda_{\phi_{2}})=\int_{0}^{\infty}f^{\ast}%
(s)d\min(\phi_{1}(s),t\phi_{2}(s)).
\]
In particular,%
\[
\langle\Lambda_{\phi_{1}},\Lambda_{\phi_{2}}\rangle_{0,1;K}=\{f:\int_{0}%
^{1}\int_{0}^{\infty}f^{\ast}(s)(\frac{d}{ds}\min(\phi_{1}(s),t\phi
_{2}(s)))\frac{dt}{t}.
\]
} and Marcinkiewicz spaces. As a consequence, in all these cases one can
obtain an explicit characterization which, in turn, can be reformulated in
terms of what nowadays is called the Calder\'{o}n operator. We refer to
Sharpley \cite[Theorem 6.9, page 511]{shar}, Bennett-Sharpley \cite{bs} and
the references therein.

Returning to Yano's theorem, as we have seen in the previous Section, the
natural formulation of the result is in terms of weak type operators. More
generally, we are led to consider the following problem in the setting of real
interpolation scales. Suppose that $\vec{X}$ and $\vec{Y}$ are Banach pairs
and suppose that$\ T$ satisfies
\begin{equation}
T:\vec{X}_{\theta,1;J}^{\blacktriangleleft}\rightarrow\vec{Y}_{\theta
,\infty;K}^{\blacktriangleleft},\text{ with }\left\Vert T\right\Vert _{\vec
{X}_{\theta,1;J}^{\blacktriangleleft}\rightarrow\vec{Y}_{\theta,\infty
;K}^{\blacktriangleleft}}\leq M(\theta),\text{ for all }\theta\in(0,1).
\label{char1}%
\end{equation}

\begin{problem}
Provide an intrinsic characterization of (\ref{char1}).
\end{problem}

Jawerth-Milman \cite{jm} took up this problem and formulated the solution as
an extension of Calder\'{o}n's characterization of weak type interpolation
(\ref{char0}). We now reformulate their solution introducing the notion of
$K-$functional for a scale of real interpolation spaces\footnote{$K-$%
functionals for many spaces or for families of spaces have been defined and
studied before but there were hardly ever explicitly computed (cf.
\cite{cwja}). The special setting of extrapolation allow us to do explicit
computations}. In fact, formally we can consider the $K-$functional for any
scale of interpolation spaces. Maybe, it is worth to stress here that the more
important issue here is the dependence of the $K-$functional on the weight
$M(\theta)$. Indeed, if a scale of interpolation functors is complete, then
when $M(\theta)=1$ we recover the usual $K-$functional (cf. Remark
\ref{rem:complete} below).

\begin{definition}
\label{def:scaleKfunctional}Let $\{\rho_{\theta}\}_{\theta\in(0,1)}$ be a
family of (quasi)-concave functions, and let $\{F_{\rho_{\theta}}\}_{\theta
\in(0,1)}$ be a family of interpolation functors such that the characteristic
function (cf. Definition \ref{def:cara}) of each $F_{\rho_{\theta}}$ is
$\rho_{\theta}$\footnote{For example the functors $\vec{X}\rightarrow\vec
{X}_{\rho_{\theta},q;J}^{\blacktriangleleft}$ or $\vec{X}\rightarrow\vec
{X}_{\rho_{\theta},q;K}^{\blacktriangleleft}$.}, $\theta\in(0,1)$. Let $M$ be
a weight, that is $M:(0,1)\rightarrow(0,\infty).$ Then, for any Banach pair
$\vec{X}$ we let%
\[
K(t,f;\{M(\theta)F_{\rho_{\theta}}(\vec{X})\}):=\left\Vert f\right\Vert
_{\sum\limits_{\theta}\rho_{\theta}(t)M(\theta)F_{\rho_{\theta}}(\vec{X})}.
\]

\end{definition}

\begin{remark}
In particular, when dealing with a family of interpolation methods
$\{F_{t^{\theta}}\}_{\theta\in(0,1)}$ of exact type $\theta,$ then we shall
usually write $F_{\theta}$ instead of $F_{t^{\theta}},$ and therefore in this
case we have%
\[
K(t,f;\{M(\theta)F_{_{\theta}}(\vec{X})\})=\left\Vert f\right\Vert
_{\sum\limits_{\theta}t^{\theta}M(\theta)F_{_{\theta}}(\vec{X})}.
\]

\end{remark}

\begin{remark}
\label{rem:complete}Note that if $M(\theta)\equiv1,$ then $K(t,f;F_{_{\theta}%
}(\vec{X}))$ is essentially $K(t,f;\vec{X})$ (cf. \cite[page 25, formula
(3.9)]{jm}).
\end{remark}

Likewise, we can introduce the concept of $J-$functional for a scale of spaces.

\begin{definition}
\label{def:scaleJfunctional}Let $\{\rho_{\theta}\}_{\theta\in(0,1)}$ be a
family of (quasi)-concave functions, and let $\{F_{\rho_{\theta}}\}_{\theta
\in(0,1)}$ be a family of interpolation functors such that the characteristic
function of each $F_{\rho_{\theta}}$ is $\rho_{\theta}$\footnote{For example
the functors $\vec{X}\rightarrow\vec{X}_{\rho_{\theta},q;J}%
^{\blacktriangleleft}$ or $\vec{X}\rightarrow\vec{X}_{\rho_{\theta}%
,q;K}^{\blacktriangleleft}$.}, $\theta\in(0,1)$. Then, for any Banach pair
$\vec{X}$ and a given weight $M$, we let%
\[
J(t,f;\{M(\theta)F_{\rho_{\theta}}(\vec{X})\}):=\left\Vert f\right\Vert
_{{\large \Delta(}\rho_{\theta}(t)M(\theta)F_{\rho_{\theta}}(\vec{X}))}.
\]

\end{definition}

\begin{theorem}
\label{teo:ca2}(cf. \cite{jm}) Let $\vec{X}$ and $\vec{Y}$ be mutually closed
Banach pairs and let $M(\theta)$ be a tempered weight. Then (\ref{char1})
holds if and only if there exists a constant $c>0$ such that
\begin{equation}
K(t,Tf;\{\vec{Y}_{\theta,\infty;K}^{\blacktriangleleft}\})\leq
cK(t,f;\{M(\theta)\vec{X}_{\theta,1;J}^{\blacktriangleleft}\}),\;\;t>0.
\label{char2}%
\end{equation}

\end{theorem}

In fact, we can also write down a \textquotedblleft Calder\'{o}n operator type
formulation\textquotedblright\ of this result by making explicit the
$K-$functionals for scales that are involved. The extrapolation version of
Calder\'{o}n's result then reads: $(\ref{char2})$ holds if and only
if\footnote{For a Banach pair $\vec{Y}$ we let $Y_{i}^{\circ}$ to be the
closure of $Y_{0}\cap Y_{1}$ in $Y_{i}.$}%
\begin{equation}
K(t,Tf,\vec{Y}^{\circ})\leq c\int_{0}^{\infty}K(\frac{t}{r},f;\vec{X})d\mu(r),
\label{char2'}%
\end{equation}
where
\begin{equation}
\tau(r)=\inf_{\theta}\{r^{\theta}M(\theta)\},\text{ \ \ }r>0, \label{tau}%
\end{equation}
and $\mu$ is the representing measure of $\tau:$%
\begin{equation}
\tau(x)=\int_{0}^{\infty}\min\{1,\frac{x}{r}\}d\mu(r),\text{ \ \ }x>0.
\label{taurepresented}%
\end{equation}
Let us recall the details. Suppose that (\ref{char2}) holds. Recall that with
absolute constants we can write (cf. \cite{jm})%
\begin{align*}
K(t,Tf;\{\vec{Y}_{\theta,\infty;K}^{\blacktriangleleft}\})  &  \approx
K(t,Tf;\{\vec{Y}_{\theta,1;J}^{\blacktriangleleft}\})\\
&  =\left\Vert Tf\right\Vert _{\sum\limits_{\theta}t^{\theta}\vec{Y}%
_{\theta,1;J}^{\blacktriangleleft}}\\
&  \approx K(t,Tf,\vec{Y}^{\circ}).
\end{align*}
It remains to establish the equivalence between the right hand side of
(\ref{char2'}) and%
\[
K(t,f;\{M(\theta)\vec{X}_{\theta,1;J}^{\blacktriangleleft}\}):=\left\Vert
f\right\Vert _{\sum\limits_{\theta}t^{\theta}M(\theta)\vec{X}_{\theta
,1;J}^{\blacktriangleleft}}.
\]
To compute the indicated norm on the right hand side observe that the
characteristic function, $C_{\theta}(s)$ say, of $t^{\theta}M(\theta)\vec
{X}_{\theta,1;J}^{\blacktriangleleft}$, is given by $C_{\theta}(s)=\frac
{s^{\theta}}{t^{\theta}M(\theta)},$ and therefore by Theorem \ref{sumformula},
for each $t>0,$%
\begin{equation}
\sum\limits_{\theta}t^{\theta}M(\theta)\vec{X}_{\theta,1;J}%
^{\blacktriangleleft}=\vec{X}_{\rho_{t},1;J} \label{reveeer1}%
\end{equation}
where $\rho_{t}(s)=\sup_{0<\theta<1}\{\frac{s^{\theta}}{t^{\theta}M(\theta
)}\}.$ Consequently, rewriting $\rho_{t}(s)$ in terms of $\tau$ (cf.
(\ref{tau})), we have%
\begin{align}
\left\Vert f\right\Vert _{\sum\limits_{\theta}t^{\theta}M(\theta)\vec
{X}_{\theta,1;J}^{\blacktriangleleft}}  &  =\left\Vert f\right\Vert _{\vec
{X}_{\rho_{t},1;J}}=\inf_{f=\int_{0}^{\infty}u(s)\frac{ds}{s}}\{\int%
_{0}^{\infty}\frac{J(s,u(s);\vec{X})}{\rho_{t}(s)}\frac{ds}{s}\}\nonumber\\
&  =\inf_{f=\int_{0}^{\infty}u(s)\frac{ds}{s}}\{\int_{0}^{\infty}%
J(s,u(s);\vec{X})\tau(\frac{t}{s})\frac{ds}{s}\}\nonumber\\
&  =\inf_{f=\int_{0}^{\infty}u(s)\frac{ds}{s}}\{\int_{0}^{\infty}\int%
_{0}^{\infty}J(s,u(s);\vec{X})\min\{1,\frac{t}{rs}\}d\mu(r)\frac{ds}{s}\}.
\label{need}%
\end{align}
Using the strong form of the fundamental Lemma, we can find a special
decomposition $f=\int_{0}^{\infty}u_{f}(s)\frac{ds}{s},$ such that, with
constants independent of $f,$ we have
\[
\text{ }K(t,f;\vec{X})\approx\int_{0}^{\infty}J(s,u_{f}(s);\vec{X}%
)\min\{1,\frac{t}{s}\}\frac{ds}{s}.
\]
Therefore, combining with (\ref{need}) we obtain
\begin{align*}
\left\Vert f\right\Vert _{\sum\limits_{\theta}t^{\theta}M(\theta)\vec
{X}_{\theta,1;J}^{\blacktriangleleft}}  &  \leq\int_{0}^{\infty}\int%
_{0}^{\infty}J(s,u_{f}(s);\vec{X})\min\{1,\frac{t}{rs}\}d\mu(r)\frac{ds}{s}\\
&  =\int_{0}^{\infty}\int_{0}^{\infty}J(s,u_{f}(s);\vec{X})\min\{1,\frac
{t}{rs}\}\frac{ds}{s}d\mu(r)\\
&  \preceq\int_{0}^{\infty}K(\frac{t}{r},f;\vec{X})d\mu(r).
\end{align*}
The last inequality can be reversed. In preparation to prove this claim we
let
\[
N(\theta):=\int_{0}^{\infty}r^{-\theta}d\mu(r),\theta\in(0,1)\text{, where
}\mu\text{ is the representing measure of }\tau\text{ (cf.
(\ref{taurepresented})).}%
\]
It is shown in \cite[Lemma 3.9]{jm} that if $M(\theta)$ is tempered then with
constants independent of $\theta,$%
\[
N(\theta)\preceq M(\theta).
\]
We will also use the fact that $\sum\limits_{\theta}t^{\theta}M(\theta)\vec
{X}_{\theta,1;J}^{\blacktriangleleft}=\sum\limits_{\theta}t^{\theta}%
M(\theta)\vec{X}_{\theta,\infty;K}^{\blacktriangleleft}.$ Now, for each
$\theta\in(0,1),$ and for each $f\in\vec{X}_{\theta,\infty;K}%
^{\blacktriangleleft},$ we have
\begin{align*}
\int_{0}^{\infty}K(\frac{t}{r},f;\vec{X})d\mu(r)  &  =\int_{0}^{\infty
}(K(\frac{t}{r},f;\vec{X})(\frac{t}{r})^{-\theta})(\frac{t}{r})^{\theta}%
d\mu(r)\\
&  \leq\int_{0}^{\infty}\left(  \sup_{x}\{K(x,f;\vec{X})x^{-\theta}\}\right)
(\frac{t}{r})^{\theta}d\mu(r)\\
&  =t^{\theta}\left\Vert f\right\Vert _{\vec{X}_{\theta,\infty;K}%
^{\blacktriangleleft}}\int_{0}^{\infty}r^{-\theta}d\mu(r)\\
&  =N(\theta)t^{\theta}\left\Vert f\right\Vert _{\vec{X}_{\theta,\infty
;K}^{\blacktriangleleft}}\\
&  \preceq M(\theta)t^{\theta}\left\Vert f\right\Vert _{\vec{X}_{\theta
,\infty;K}^{\blacktriangleleft}}.
\end{align*}
Let $f=\sum f_{\theta}$ be a decomposition such that
\[
\sum\left\Vert f_{\theta}\right\Vert _{M(\theta)t^{\theta}\vec{X}%
_{\theta,\infty;K}^{\blacktriangleleft}}\approx\left\Vert f\right\Vert _{\sum
M(\theta)t^{\theta}\vec{X}_{\theta,\infty;K}^{\blacktriangleleft}}%
\approx\left\Vert f\right\Vert _{\sum\limits_{\theta}t^{\theta}M(\theta
)\vec{X}_{\theta,1;J}^{\blacktriangleleft}}.
\]
Then formally (by Fatou's Lemma),%
\begin{align*}
\int_{0}^{\infty}K(\frac{t}{r},f;\vec{X})d\mu(r)  &  =\int_{0}^{\infty}%
K(\frac{t}{r},\sum f_{\theta};\vec{X})d\mu(r)\\
&  \leq\sum\int_{0}^{\infty}K(\frac{t}{r},f_{\theta};\vec{X})d\mu(r)\\
&  \preceq\sum\limits_{\theta}M(\theta)t^{\theta}\left\Vert f_{\theta
}\right\Vert _{\vec{X}_{\theta,\infty;K}^{\blacktriangleleft}}\\
&  \approx\left\Vert f\right\Vert _{\sum\limits_{\theta}t^{\theta}%
M(\theta)\vec{X}_{\theta,\infty;K}^{\blacktriangleleft}}\\
&  \approx\left\Vert f\right\Vert _{\sum\limits_{\theta}t^{\theta}%
M(\theta)\vec{X}_{\theta,1;J}^{\blacktriangleleft}},
\end{align*}
concluding the proof of the equivalence between (\ref{char2}) and
(\ref{char2'}).

\begin{problem}
More generally carry out explicit calculations of the $K-$functionals for
interpolation scales of interpolation functors $\{F_{\rho_{\theta}}%
\}_{\theta\in I}.$
\end{problem}

\begin{remark}
In connection with the previous Problem it is important to recall the concept
of \textquotedblleft complete families of interpolation
functors\textquotedblright\ introduced in \cite{jm}. We say that a family
$\{F_{\rho_{\theta}}\}_{\theta\in I}$ is \textit{complete} if there exists a
constant $C>0$ such that%
\[
\inf_{\theta\in I}\{\frac{\rho_{\theta}(t)}{\rho_{\theta}(s)}\}\leq
C\min\{1,\frac{t}{s}\},s,t>0.
\]
In this context the import of this notion is explained in the next Example.
\end{remark}

\begin{example}
Suppose that $\{F_{\rho_{\theta}}\}_{\theta\in I}$ is complete. Then, if
$\vec{X}$ is mutually closed,
\[
K(t,f;\{F_{\rho_{\theta}}(\vec{X})\}):=\left\Vert f\right\Vert _{\sum
\limits_{\theta}\rho_{\theta}(t)F_{\rho_{\theta}}(\vec{X})}\approx
K(t,f;\vec{X}^{\circ})
\]
and%
\[
J(t,f;\{F_{\rho_{\theta}}(\vec{X})\}):=\left\Vert f\right\Vert
_{{\large \Delta}_{\theta}(\rho_{\theta}(t)F_{\rho_{\theta}}(\vec{X})})\approx
J(t,f;\vec{X}^{\circ})
\]
(cf. \cite[page 15: formulae in line 9 and (2.15)]{jm}, \cite[Theorem
2.2]{AL2017}.)
\end{example}

\begin{problem}
Let $\{\rho_{\theta}\}_{\theta\in I}$ be a family of quasi-concave functions,
let $\{F_{\rho_{\theta}}\}_{\theta\in I}$ be a complete family of
interpolation functors and let $M(\theta)$ be a tempered weight. Compute
$K(t,f;\{M(\theta)F_{\rho_{\theta}}(\vec{X})\})$ and $J(t,f;\{M(\theta
)F_{\rho_{\theta}}(\vec{X})\}).$
\end{problem}

\begin{remark}
When dealing with ordered pairs we can replace the usual real interpolation
spaces by the modified ones $\langle\vec{X}\rangle_{\theta,1;J}%
^{\blacktriangleleft};\langle\vec{X}\rangle_{\theta,q(\theta);K}%
^{\blacktriangleleft},$ etc. (cf. \cite{jm}, \cite{ALM})
\end{remark}

\begin{example}
\label{ex:yano} Let us consider the weak type Yano condition\footnote{The same
argument works, with appropriate adjustment, to deal with general $\alpha>0.$%
}: $T:L(p,1)\subset L^{p}\rightarrow L^{p}\subset L(p,\infty)$ with
$\left\Vert T\right\Vert _{L(p,1)\rightarrow L(p,\infty)}\leq cp,$
$p>p_{0}>1.$ Here all spaces are based on $[0,1].$ Then,
\[
\langle L^{1},L^{\infty}\rangle_{1/p^{\prime},\infty;K}^{\blacktriangleleft
}=L(p,\infty):=\{f:\left\Vert f\right\Vert _{L(p,\infty)}:=\sup_{0<t<1}%
\{f^{\ast\ast}(t)t^{1/p}\}<\infty\},
\]
$L^{p}=\langle L^{1},L^{\infty}\rangle_{1/p^{\prime},p;K}^{\blacktriangleleft
},$ and since $\langle L^{1},L^{\infty}\rangle_{1/p^{\prime},1;J}%
^{\blacktriangleleft}\overset{1}{\subset}\langle L^{1},L^{\infty}%
\rangle_{1/p^{\prime},p;K}^{\blacktriangleleft},$ we have (cf. (\ref{char2})
and (\ref{char2'}))%
\begin{align*}
K(t,Tf;\{\langle L^{1},L^{\infty}\rangle_{1/p^{\prime},\infty;K}%
\}_{1<p<\infty})  &  \preceq K(t,f;\{p\langle L^{1},L^{\infty}\rangle
_{1/p^{\prime},1;J}\}_{1<p<\infty})\\
&  \approx\int_{0}^{\infty}K(\frac{t}{r},f;L^{1},L^{\infty})w(r)dr,
\end{align*}
where $w(r)dr$ is the measure representing the function
\[
\tau(t)=\inf_{p>1}\{pt^{1-1/p}\}=t\inf_{p>1}\{pt^{-1/p}\}=t\inf_{0<u<1}%
\{\frac{e^{u\log\frac{1}{t}}}{u}\}:=t\inf_{0<u<1}g(u).
\]
We can compute directly $\frac{dg(u)}{du} =\frac{g(u)}{u}(u\log\frac{1}%
{t}-1),$ and
\[
\frac{d^{2} g(u)}{du^{2}}=g^{\prime}(u)\frac{(u\log\frac{1}{t}-1)}{u}%
+\frac{g(u)}{u^{2}}=\frac{g(u)}{u^{2}}\left(  (u\log\frac{1}{t}-1)^{2}%
+1\right) .
\]
Then, for $t<1/e$, $u_{\ast}=\frac{1}{\log\frac{1}{t}}\in(0,1),g^{\prime
}(u_{\ast})=0,g^{^{\prime\prime}}(u_{\ast})>0$ and we see that for $t<1/e$,
$\tau(t)=et\log\frac{1}{t}.$ If $t\le1/e,$ then $g^{\prime}(u)\le0,$ and
$g(u)$ decreases. Hence, $\tau(t)=1$ in this range. So all in all,
\[
\tau(t)=et\log\frac{1}{t}\chi_{(0,1/e)}(t)+\chi_{(1/e,\infty)}(t),
\]%
\[
w(t)\approx-t\tau^{\prime\prime}(t)=\chi_{(0,1/e)}(t).
\]
Consequently,%
\[
\int_{0}^{\infty}K(\frac{t}{r},f;L^{1},L^{\infty})w(r)dr=\int_{0}^{1/e}%
K(\frac{t}{r},f;L^{1},L^{\infty})dr=t\int_{et}^{\infty}K(u,f;L^{1},L^{\infty
})\frac{du}{u^{2}}.
\]
Thus, for $0<t<1,$ we can write
\begin{align*}
(Tf)^{\ast\ast}(t)  &  \preceq\int_{et}^{e}\frac{K(u,f;L^{1},L^{\infty})}%
{u}\frac{du}{u}+\int_{e}^{\infty}K(u,f;L^{1},L^{\infty})\frac{du}{u^{2}}\\
&  \preceq f^{\ast\ast}(et)\log\frac{1}{t}+\left\Vert f\right\Vert _{L^{1}%
}e^{-1}.
\end{align*}
Therefore,%
\begin{equation}
\frac{(Tf)^{\ast\ast}(t)}{\log\frac{2}{t}}\preceq f^{\ast\ast}(et)+\frac
{\left\Vert f\right\Vert _{L^{1}}e^{-1}}{\log\frac{2}{t}} \label{anunzia}%
\end{equation}
and we obtain%
\begin{align*}
\sup_{t\in(0,1)}\frac{(Tf)^{\ast\ast}(t)}{\log\frac{2}{t}}  &  \preceq
\left\Vert f\right\Vert _{L^{\infty}}+\frac{\left\Vert f\right\Vert _{L^{1}}%
}{\log2}\\
&  \preceq\left\Vert f\right\Vert _{L^{\infty}}.
\end{align*}
But (\ref{anunzia}) implies more general results. Indeed, let $X$ be an
rearrangement invariant (briefly, r.i.) space\footnote{For the definition see
e.g. \cite{bs}.} such that the Hardy operator $Pf(t)=\frac{1}{t}\int_{0}%
^{t}f(s)ds,$ is bounded on $X,$ and let $X(\log^{-1})$ be defined by the norm
(cf. \cite{AL2009})%
\[
\left\Vert f\right\Vert _{X(\log^{-1})}=\left\Vert \frac{f^{\ast\ast}(t)}%
{\log\frac{2}{t}}\right\Vert _{X}.
\]
Then applying the $X$-norm to (\ref{anunzia}) we obtain the extrapolation
theorem (cf. \cite[Theorem~4.3]{AL2009}):%
\[
T:X\rightarrow X(\log^{-1}).
\]

\end{example}

\begin{example}
A similar result can be obtained if we consider operators such that
\[
\left\Vert T\right\Vert _{L(p,1)\rightarrow L(p,\infty)}\leq cp^{\alpha
}\;\;\mbox{as}\;p\rightarrow\infty,
\]
where $\alpha>0.$ Indeed, in this case we have $\tau(t)=t\inf_{p>1}%
\{p^{\alpha}t^{-1/p}\},0<t<1.$ A computation shows that $\tau(t)\approx
t(\log\frac{1}{t})^{\alpha}$ near zero, and
\[
w(t)\approx-t\tau^{\prime\prime}(t)\approx(\log\frac{1}{t})^{\alpha-1}\text{
as }t\longrightarrow0.
\]
which leads to estimates of the form (cf. \cite{jm})
\[
(Tf)^{\ast\ast}(t)\preceq\int_{et}^{e}\frac{K(u,f;L^{1},L^{\infty})}{u}%
(\log\frac{u}{t})^{\alpha-1}\frac{du}{u},\;\;0<t<1.
\]

\end{example}

\begin{example}
\label{ex:enuno}(Yano type extrapolation theorem) The result of Example
\ref{ex:yano} can be extended to scales. We illustrate this considering
deteriorating norms when $p\rightarrow1.$ Let $\vec{A},\vec{B}$ be mutually
closed ordered Banach pairs. Let $X$ be a r.i. space on $(0,1)$. Suppose that
$T$ is a bounded operator, $T:\vec{A}_{\theta,1;J}^{\blacktriangleleft
}\rightarrow\vec{B}_{\theta,\infty;K}^{\blacktriangleleft},$ with $\left\Vert
T\right\Vert _{\vec{A}_{\theta,1;J}^{\blacktriangleleft}\rightarrow\vec
{B}_{\theta,\infty;K}^{\blacktriangleleft}}\preceq\theta^{-1},\theta\in(0,1).$
Then,%
\[
\left\Vert \frac{d}{dt}(K(t,Tf;\vec{B}))\right\Vert _{X}\leq c\left\Vert
\frac{K(t,f;\vec{A})}{t}\right\Vert _{X}.
\]
In particular, if $\vec{A}=\vec{B}=(L^{1}(0,1),L^{\infty}(0,1))$ and
$T:L(p,1)\rightarrow L(p,\infty),$ with norm less or equal than $cp^{\prime
}=\frac{cp}{p-1},1<p<\infty,$ then%
\[
\left\Vert (Tf)^{\ast}\right\Vert _{X}\leq c\left\Vert f^{\ast\ast}\right\Vert
_{X}.
\]
Yano's theorem (i) corresponds to the case $X=L^{1}.$
\end{example}

\begin{proof}
By extrapolation, via the $K-$functional for scales, we have%
\[
K(t,Tf;\{\vec{B}_{\theta,\infty;K}^{\blacktriangleleft}\})\leq cK(t,f;\{\theta
^{-1}\vec{A}_{\theta,1;J}^{\blacktriangleleft}\}).
\]
Therefore, by calculation similar to that of Example \ref{ex:yano} (cf.
\cite{jm}), there exists an absolute constant $c>0$, such that
\[
K(t,Tf;\vec{B})\leq c\int_{0}^{t}K(s,f;\vec{A})\frac{ds}{s}.
\]
Since $\lim_{t\rightarrow0}K(t,Tf;\vec{B})=0,$ we can rewrite the last
inequality as%
\[
\int_{0}^{t}\frac{d}{ds}(K(s,Tf;\vec{B}))ds\leq c\int_{0}^{t}K(s,f;\vec
{A})\frac{ds}{s}.
\]

Therefore, since $\frac{d}{ds}(K(s,Tf;\vec{B}))$ and $\frac{K(s,f\vec{A})}{s}$
are decreasing then, by the Calder\'{o}n-Mityagin principle for any r.i. space
we have%

\[
\left\Vert \frac{d}{dt}(K(t,Tf;\vec{B})\right\Vert _{X}\leq c\left\Vert
\frac{K(t,f;\vec{A})}{t}\right\Vert _{X},
\]
as desired.
\end{proof}

\begin{example}
\label{ex:lag}(Rearrangement inequalities associated to Yano type
extrapolation) Let $\vec{A},\vec{B}$ be mutually closed Banach pairs. Suppose
that $T$ is a bounded operator, $T:\vec{A}_{\theta,1;J}^{\blacktriangleleft
}\rightarrow\vec{B}_{\theta,\infty;K}^{\blacktriangleleft},$ with $\left\Vert
T\right\Vert _{\vec{A}_{\theta,1;J}^{\blacktriangleleft}\rightarrow\vec
{B}_{\theta,\infty;K}^{\blacktriangleleft}}\preceq\theta^{-1}(1-\theta
)^{-1},\theta\in(0,1).$ Then, there exists a constant $c>0$ such that%
\begin{equation}
K(t,Tf;\vec{B})\leq c\left(  \int_{0}^{t}K(s,f;\vec{A})\frac{ds}{s}+t\int%
_{t}^{\infty}K(s,f;\vec{A})\frac{ds}{s^{2}}\right)  . \label{becomes}%
\end{equation}
In particular, if $\vec{A}=\vec{B}=(L^{1},L^{\infty})$, we have $\vec
{A}_{1/p^{\prime},1;J}^{\blacktriangleleft}=L(p,1),\vec{B}_{1/p^{\prime
},\infty;K}^{\blacktriangleleft}=L(p,\infty),$ $1<p<\infty$. Then
(\ref{becomes}) can be written as%
\begin{equation}
t(Tf)^{\ast\ast}(t)\leq c\left(  \int_{0}^{t}f^{\ast\ast}(s)ds+t\int%
_{t}^{\infty}f^{\ast\ast}(s)\frac{ds}{s}\right)  . \label{becomes1}%
\end{equation}
For the Hilbert transform (\ref{becomes1}) is a classical inequality due to
O'Neil-Weiss \cite{ON}. In the same paper O'Neil-Weiss also prove
(\ref{becomes1}) for Calder\'{o}n-Zygmund operators, a result which they
credit to A. Calder\'{o}n and E. Stein. In \cite[Appendix]{cal}, Calder\'{o}n
extended these results to operators $T$ that, together with their adjoints,
are of weak type $(1,1)$ and strong type $(2,2).$ It has been pointed out by
several authors, for example, Bennett-Rudnick (cf. \cite[Theorem 4.7, page
134]{bs}), Jawerth-Milman \cite[Proposition 5.2.2, page 50]{jm}, Semenov
\cite{semen}), that for the Hilbert transform or C-Z operators, one can
improve (\ref{becomes1}) by means of replacing ** by * throughout, using the
fact that C-Z singular integral operators are of weak type (1,1). The weak
type $(1,1)$ assumption apparently cannot be dispensed with, and in general
does not follow from the assumptions $\left\Vert T\right\Vert
_{L(p,1)\rightarrow L(p,\infty)}\preceq\frac{p^{2}}{p-1},p\in(1,\infty)$ (cf.
\cite[First paragraph, page 2]{jm}, \cite[Remark~4.5]{AL2009}, \cite{hen}). In
\cite[Theorem 2.2, page 603]{bedesa}, the authors define operators of weak
type $(\infty,\infty)$ as those that map $L^{\infty}$ into $L(\infty,\infty),$
and using this definition show that operators of weak type $(1,1)$ and
$(\infty,\infty)$ satisfy (\ref{becomes1}). Earlier, in \cite{devr},
DeVore-Riemenschneider-Sharpley define an abstract notion of generalized weak
type $(1,1),(\infty,\infty)$ by asking that condition (\ref{becomes}) be
satisfied. Therefore, the notion of weak type $(1,1),(\infty,\infty)$ of
DeVore-Riemenschneider-Sharpley is equivalent to the assumption $T:\vec
{A}_{\theta,1;J}^{\blacktriangleleft}\rightarrow\vec{B}_{\theta,\infty
;K}^{\blacktriangleleft},$ with $\left\Vert T\right\Vert _{\vec{A}%
_{\theta,1;J}^{\blacktriangleleft}\rightarrow\vec{B}_{\theta,\infty
;K}^{\blacktriangleleft}}\preceq\theta^{-1}(1-\theta)^{-1},\theta\in(0,1).$

One can prove (\ref{becomes}) via the $K-$functional for scales. Equivalently,
we note (cf. Example \ref{exa:becomes3} below) that for%
\[
\tau(t)=\inf_{0<\theta<1}\{\theta^{-1}(1-\theta)^{-1}t^{\theta}\}
\]
we have%
\[
\tau(t)\approx\int_{0}^{\infty}\min(1,\frac{t}{n})\min(1,n)\frac{dn}{n}.
\]
Then, since under current assumptions $T$ satisfies the $K/J$ inequality (cf.
the next section)%
\[
K(t,Tf;\vec{B})\leq c\tau(t/s)J(s,f;\vec{A}),\;\;s,t>0,
\]
it follows that selecting a decomposition of $f=\int_{0}^{\infty}u(s)\frac
{ds}{s},$ such that
\[
\int_{0}^{\infty}\min(1,\frac{t}{s})J(s,u(s);\vec{A})\frac{ds}{s}\leq
cK(t,f;\vec{A})
\]
leads to%
\begin{align*}
K(t,Tf;\vec{B})  &  \leq c\int_{0}^{\infty}\min(1,n)\int_{0}^{\infty}%
\min(1,\frac{t}{sn})J(s,u(s);\vec{A})\frac{ds}{s}\frac{dn}{n}\\
&  =c\int_{0}^{1}K(\frac{t}{n},f;\vec{A})\min(1,n)\frac{dn}{n}\\
&  =c\left(  \int_{0}^{1}K(\frac{t}{n},f;\vec{A})dn+\int_{1}^{\infty}%
K(\frac{t}{n},f;\vec{A})\frac{dn}{n}\right) \\
&  =c\left(  t\int_{t}^{\infty}K(s,f;\vec{A})\frac{ds}{s^{2}}+\int_{0}%
^{t}K(s,f;\vec{A})\frac{ds}{s}\right)  ,
\end{align*}
as we wished to show.
\end{example}

\section{K/J Inequalities and extrapolation}

In extrapolation we start from a family of inequalities for a given operator,
but we do not have \textit{apriori} end point spaces where the estimates are
valid. It is then natural to collect estimates on elements that belong to the
intersection of all the domain spaces on which the operators act and use this
information to derive a basic inequality. This is the idea of the $K/J$
inequalities of \cite{jm}, a particular case of which we now review.

Indeed, it turns out to be useful for further developments (e.g. Bilinear
Extrapolation treated in Section \ref{bilinear}) to bring to the forefront
some of the key arguments underlying the proof of the equivalence between
(\ref{k/j1}) and (\ref{k/j4}). This leads in two steps to the $K/J$
inequalities: $(i)$ Associated to the extrapolation information there is a
natural concave function that allows one to establish an inequality between
the $K-$ and $J-$functionals; and $(ii)$ Extend this inequality to a
$K-$functional inequality using the strong form of the fundamental Lemma of
Interpolation. As it turns out this approach is essentially equivalent to the
use of the $K-$functional of scales discussed in the previous section, as we
shall now show.

It is instructive to see how $K/J$ inequalities arise in classical setting of
interpolation theory. Let $\vec{A}$, $\vec{B},$ be Banach pairs, and suppose
that $T:\vec{A}\rightarrow$ $\vec{B}$ is a bounded operator. Then,
\begin{equation}
K(t,Tf;\vec{B})\preceq K(t,f;\vec{A}),\;\;t>0, \label{k/j0}%
\end{equation}
which combined with the elementary inequality (cf. \cite[Lemma~3.2.1]{BL})%
\begin{equation}
K(t,f;\vec{A})\leq\min(1,\frac{t}{s})J(s,f;\vec{A}),\;\;t,s>0
\label{sustituir}%
\end{equation}
leads to the \textquotedblleft mother\textquotedblright\ of all $K/J$
inequalities%
\begin{equation}
K(t,Tf;\vec{B})\preceq\min(1,\frac{t}{s})J(s,f;\vec{A}),t,s>0. \label{k/j3}%
\end{equation}
Conversely, if (\ref{k/j3}) holds then we can return to (\ref{k/j0}) via the
strong form of the fundamental Lemma, as we now explain. Select a
decomposition $f=\int_{0}^{\infty}u_{f}(s)\frac{ds}{s},$ such that $\int%
_{0}^{\infty}\min(1,\frac{t}{s})J(s,u_{f}(s);\vec{A})\frac{ds}{s}\leq
cK(t,f;\vec{A}).$ Then,%
\begin{align}
K(t,Tf;\vec{B})  &  \leq c\int_{0}^{\infty}K(t,Tu_{f}(s);\vec{B})\frac{ds}%
{s}\nonumber\\
&  \leq c\int_{0}^{\infty}\min(1,\frac{t}{s})J(s,u_{f}(s);\vec{A})\frac{ds}%
{s}\label{usual}\\
&  \leq cK(t,f;\vec{A}).\nonumber
\end{align}

Now we turn to the context of extrapolation. In this context (\ref{k/j3}) is
not available but there is natural concave function that will replace
$\phi(u)=$ $\min\{1,u\}$ (cf. Example \ref{ex:marcao} below). We develop this
point in detail.

Let $\vec{A}$, $\vec{B},$ be mutually closed Banach pairs, and let $F_{\theta
},G_{\theta},$ be exact interpolation functors of type $\theta.$ Consider a
bounded operator $T:F_{\theta}(\vec{A})\rightarrow G_{\theta}(\vec{B})$, with
norm $M(\theta),\theta\in(0,1).$ Now (\ref{sustituir}) is not available to us
anymore and we search for a substitute. For this purpose we note that for all
$\theta\in(0,1),$ $\vec{A}_{\theta,1;J}^{\blacktriangleleft}%
\overset{1}{\subset}F_{\theta}(\vec{A}),$ and $G_{\theta}(\vec{B}%
)\overset{1}{\subset}\vec{B}_{\theta,\infty;K}^{\blacktriangleleft}$,
consequently, for all $s,t>0,$%
\begin{equation}
K(t,Tf;\vec{B})t^{-\theta}\leq\left\Vert Tf\right\Vert _{\vec{B}%
_{\theta,\infty;K}^{\blacktriangleleft}}\leq M(\theta)\left\Vert f\right\Vert
_{\vec{A}_{\theta,1;J}^{\blacktriangleleft}}\leq M(\theta)s^{-\theta
}J(s,f;\vec{A}), \label{first}%
\end{equation}
where the rightmost inequality is elementary (cf. \cite[Theorem 3.11.2 (4),
page 64]{BL}).

Hence, we arrive to%
\begin{align}
K(t,Tf;\vec{B})  &  \leq\inf_{\theta}\{M(\theta)\left(  \frac{t}{s}\right)
^{\theta}\}J(s,f;\vec{A})\nonumber\\
&  =\phi\left(  \frac{t}{s}\right)  J(s,f;\vec{A}),s,t>0, \label{second}%
\end{align}
where $\phi(u)=\inf_{\theta\in(0,1)}\{M(\theta)u^{\theta}\}.$ Inequalities
(\ref{second}) and (\ref{k/j3}) are examples of $K/J$ inequalities. Let $\Psi$
be a concave function, such that $\lim_{t\rightarrow0}\Psi(t)=\lim
_{t\rightarrow\infty}\frac{\Psi(t)}{t}=0,$ e.g. $\Psi(u)=\min(1,u),$ then we
say that $T$ satisfies a $\Psi-K/J$ inequality if there exists a constant
$c>0$ such that\footnote{These inequalities can be formulated without the use
of the $J-$functional. Indeed, note that since (\ref{second}) is valid for all
$s>0,$ then choosing $s=\frac{\left\Vert f\right\Vert _{A_{0}}}{\left\Vert
f\right\Vert _{A_{1}}},$ gives $J(s,f;\vec{A})=\left\Vert f\right\Vert
_{A_{1}}\max\{\frac{\left\Vert f\right\Vert _{A_{0}}}{\left\Vert f\right\Vert
_{A_{1}}},s\}=\left\Vert f\right\Vert _{A_{0}}$ and (\ref{second}) implies%
\begin{equation}
K(t,Tf;\vec{B})\leq\left\Vert f\right\Vert _{A_{0}}\phi\left(  t\frac
{\left\Vert f\right\Vert _{A_{1}}}{\left\Vert f\right\Vert _{A_{0}}}\right)  .
\label{second1}%
\end{equation}
Conversely, suppose that (\ref{second1}) holds. Then, since for any $s>0,$
$\left\Vert f\right\Vert _{A_{0}}\leq J(s,f;\vec{A}),$ and $\frac{\phi(u)}{u}$
decreases, we have%
\[
\frac{\phi\left(  t\frac{\left\Vert f\right\Vert _{A_{1}}}{\left\Vert
f\right\Vert _{A_{0}}}\right)  }{t\frac{\left\Vert f\right\Vert _{A_{1}}%
}{\left\Vert f\right\Vert _{A_{0}}}}\leq\frac{\phi\left(  t\frac{\left\Vert
f\right\Vert _{A_{1}}}{J(s,f;\vec{A})}\right)  }{t\frac{\left\Vert
f\right\Vert _{A_{1}}}{J(s,f;\vec{A})}}.
\]
Consequently,%
\[
\left\Vert f\right\Vert _{A_{0}}\phi\left(  t\frac{\left\Vert f\right\Vert
_{A_{1}}}{\left\Vert f\right\Vert _{A_{0}}}\right)  \leq J(s,f;\vec{A}%
)\phi\left(  t\frac{\left\Vert f\right\Vert _{A_{1}}}{J(s,f;\vec{A})}\right)
.
\]
But $\left\Vert f\right\Vert _{A_{1}}\leq\frac{J(s,f;\vec{A})}{s}$ and $\phi$
increases. Hence, the right hand side is smaller than%
\[
J(s,f;\vec{A})\phi\left(  \frac{t}{s}\right)
\]
as we wished to show.}%
\[
K(t,Tf;\vec{B})\leq c\Psi\left(  \frac{t}{s}\right)  J(s,f;\vec{A}),s,t>0.
\]
When $\Psi$ is understood we simply drop it and talk about $K/J$ inequalities.

\begin{example}
\label{ex:marcao}Note that if $M(\theta)\equiv1,$ then $\inf_{\theta\in
(0,1)}\{u^{\theta}\}=\inf_{\theta\in(0,1)}\{e^{\theta\log u}\}=\min\{1,u\},$
and we are \textquotedblleft back to interpolation\textquotedblright.
\end{example}

Then, when properly interpreted the equivalence between $K/J$ inequalities and
$K-$functional inequalities persists in the setting of extrapolation theory.

\begin{theorem}
\label{teo:ca1}Let $\vec{A}$, $\vec{B},$ be mutually closed Banach pairs, and
let $F_{\theta},G_{\theta},$ be exact interpolation functors of type $\theta.$
Consider a bounded operator $T:F_{\theta}(\vec{A})\rightarrow G_{\theta}%
(\vec{B})$, with norm $M(\theta),\theta\in(0,1).$ Then, the following are equivalent:

(i) There exists $c>0$ such that%
\begin{equation}
K(t,Tf;\{\vec{A}_{\theta,\infty;K}^{\blacktriangleleft}\}_{\theta\in
(0,1)})\leq cK(t,f;\{M(\theta)\vec{B}_{\theta,1;J}^{\blacktriangleleft
}\}_{\theta\in(0,1)}). \label{k/j1}%
\end{equation}
(ii) There exists a constant $c>0$ such that%
\begin{equation}
K(t,Tf;\vec{B}^{\circ})\leq c\int_{0}^{\infty}K(\frac{t}{r},f;\vec{A})d\mu(r),
\label{k/j4}%
\end{equation}
where $\mu$ is the representing measure of the concave function $\phi
(u)=\inf_{\theta}\{u^{\theta}M(\theta)\},$%
\[
\phi(t)=\int_{0}^{\infty}\min\{1,\frac{t}{r}\}d\mu(r).
\]

(iii) $T$ satisfies the $\phi-K/J$ inequality, that is there exists a constant
$c>0$ such that%
\begin{equation}
K(t,Tf;\vec{B})\leq c\phi(\frac{t}{s})J(s,f;\vec{A}). \label{k/j5}%
\end{equation}

\end{theorem}

\begin{proof}
The equivalence between (\ref{k/j1}) and (\ref{k/j4}) was shown in Section
\ref{sec:J&M} (cf. \ref{char2'}). Furthermore, by \ Theorem \ref{teo:ca2},
(\ref{k/j1}) is equivalent to
\begin{equation}
T:\vec{A}_{\theta,1;J}^{\blacktriangleleft}\rightarrow\vec{B}_{\theta
,\infty;K}^{\blacktriangleleft},\text{ with norm }M(\theta),\theta\in(0,1).
\label{k/j6}%
\end{equation}
The argument provided before the statement of this theorem shows that
(\ref{k/j6}) implies (\ref{k/j5}). Finally, suppose that (\ref{k/j5}) holds.
Using the strong form of the fundamental Lemma, we can write%
\[
f=\int_{0}^{\infty}u_{f}(s)\frac{ds}{s}%
\]
such that%
\[
\int_{0}^{\infty}\min(1,\frac{t}{s})J(s,u_{f}(s);\vec{A})\frac{ds}{s}\preceq
K(t,f;\vec{A}).
\]
Therefore%
\begin{align*}
K(t,Tf;\vec{B})  &  \leq\int_{0}^{\infty}K(t,Tu_{f}(s);\vec{B})\frac{ds}{s}\\
&  \leq c\int_{0}^{\infty}\phi(\frac{t}{s})J(s,u_{f}(s);\vec{A})\frac{ds}{s}\\
&  =c\int_{0}^{\infty}\int_{0}^{\infty}\min(1,\frac{t}{sr})J(s,u_{f}%
(s);\vec{A})\frac{ds}{s}d\mu(r)\\
&  \leq c\int_{0}^{\infty}K(\frac{t}{r},f;\vec{A})d\mu(r),
\end{align*}
as we wished to show.
\end{proof}

Let us comment that one key point of the argument is the fact that the concave
functions $\min(1,\frac{\circ}{s})$ are *extremal* (in a suitable
\textquotedblleft Krein-Milman sense\textquotedblright) and that in fact we
can build \textquotedblleft all\textquotedblright\ concave functions from
them. Indeed, by the representation theorem (cf. \cite[Lemma 5.4.3, page
117]{BL}), to each concave function $\phi$ such that $\lim_{s\rightarrow0}%
\phi(s)=\lim_{s\rightarrow\infty}\phi(s)/s=0,$ there corresponds a measure
(representing measure) such that\footnote{In case $d\mu(r)=w(r)dr$, there is a
simple algorithm to find $w,$ namely $d\mu(t)=-td\phi^{\prime}(t).$}%
\begin{equation}
\phi(t)=\int_{0}^{\infty}\min(1,\frac{t}{r})d\mu(r). \label{analog}%
\end{equation}
More generally,%
\[
\phi(t)=\alpha+\beta t+\int_{0}^{\infty}\min(1,\frac{t}{r})d\mu(r)
\]
where $\alpha=\lim_{s\rightarrow0}\phi(s),$ and $\beta=\lim_{s\rightarrow
\infty}\phi(s)/s.$

One of the difficulties in the treatment of bilinear extrapolation is the lack
of such formulae for concave functions of two variables. This leads to

\begin{problem}
Find an analogue of (\ref{analog}) for concave functions of two variables (cf.
Section \ref{bilinear} below).
\end{problem}

\begin{example}
(cf. \cite{jm1}) We consider an elementary approach to the underlying $K/J$
inequalities associated with Yano's theorem\footnote{For simplicity we let
$\alpha=1.$}. Suppose that $T$ is a bounded operator $T:$ $L^{p}%
(0,1)\rightarrow L^{p}(0,1),$ with%
\begin{equation}
\left\Vert Tf\right\Vert _{L^{p}(0,1)}\leq c\frac{p}{p-1}\left\Vert
f\right\Vert _{L^{p}(0,1)},1<p<\infty. \label{ex1}%
\end{equation}
Let us first remark that since $L(LogL)(0,1)$ is a Lorentz space, by
\cite[Lemma~II,5,2]{KPS}, to prove Yano's theorem we only need to establish
\[
\left\Vert Tf\right\Vert _{L^{1}(0,1)}\leq c\left\Vert f\right\Vert
_{LLogL(0,1)},
\]
for functions of the form $f(x)=\gamma\chi_{A},$ where $\gamma>0$ and
$\chi_{A}$ is the characteristic function of a measurable set $A\subset[0,1]$.
Now, taking limits when $p\rightarrow\infty$ in (\ref{ex1}) we see that%
\begin{equation}
\left\Vert Tf\right\Vert _{L^{\infty}(0,1)}\leq c\left\Vert f\right\Vert
_{L^{\infty}(0,1)}. \label{ex2}%
\end{equation}
Let $t>s=m(A)$ ($m(A)$ is Lebesgue measure of $A$), then for any $p>1,$%
\begin{align*}
\int_{s}^{t}(Tf)^{\ast}(u)du  &  \leq\left\{  \int_{s}^{t}[(Tf)^{\ast}%
(u)]^{p}du\right\}  ^{1/p}(t-s)^{1/p^{\prime}}\\
&  \leq cp^{\prime}\left\Vert f\right\Vert _{L^{p}(0,1)}s^{1/p^{\prime}}%
(\frac{t}{s}-1)^{1/p^{\prime}}\\
&  =cp^{\prime}\gamma s^{1/p}s^{1/p^{\prime}}(\frac{t}{s}-1)^{1/p^{\prime}}\\
&  =cp^{\prime}\gamma s(\frac{t}{s}-1)^{1/p^{\prime}}.
\end{align*}
Suppose that $\frac{t}{s}>e,$ then we can select $p^{\prime}=\log\frac{t}{s},$
and we get%
\begin{equation}
t(Tf)^{\ast\ast}(t)-s(Tf)^{\ast\ast}(s)=\int_{s}^{t}(Tf)^{\ast}(u)du\leq
ce\gamma s\log\frac{t}{s}. \label{ex3}%
\end{equation}
Moreover, by (\ref{ex2}) we have for any $u>0$%
\begin{equation}
u(Tf)^{\ast\ast}(u)\leq u\left\Vert Tf\right\Vert _{L^{\infty}}\leq
uc\left\Vert f\right\Vert _{L^{\infty}}=cu\gamma. \label{ex4}%
\end{equation}
Therefore, if we let $u=s,$ we get%
\[
s(Tf)^{\ast\ast}(s)\leq cs\gamma=c\left\Vert f\right\Vert _{L^{1}}.
\]
Inserting this last inequality in (\ref{ex3}), we find that for $t>es,$
\[
t(Tf)^{\ast\ast}(t)\leq ce\left\Vert f\right\Vert _{L^{1}}(1+\log\frac{t}%
{s}),
\]
while if $t\leq es,$ then we can apply (\ref{ex4}) to find%
\[
t(Tf)^{\ast\ast}(t)\leq c\gamma t=ct\left\Vert f\right\Vert _{L^{\infty}}.
\]
Note that since $t(Tf)^{\ast\ast}(t)=K(t,Tf;L^{1},L^{\infty}),$ and
\[
ces\gamma(1+\log\frac{t}{s})=ce\left\Vert f\right\Vert _{L^{1}}(1+\log
\frac{t\left\Vert f\right\Vert _{L^{\infty}}}{\left\Vert f\right\Vert _{L^{1}%
}})
\]
we have the $K/J$ inequality%
\[
K(t,Tf;L^{1},L^{\infty})\leq ce\left\Vert f\right\Vert _{L^{1}}\phi
(\frac{t\left\Vert f\right\Vert _{L^{\infty}}}{\left\Vert f\right\Vert
_{L^{1}}}),
\]
where%
\[
\phi(u)=\left\{
\begin{array}
[c]{cc}%
e(1+\log u), & u\geq e\\
eu, & u\leq e
\end{array}
\right.  .
\]
The representing measure for $\phi$ is given by
\[
d\mu(r)\simeq w(r)dr,\text{ where }w(r)=\left\{
\begin{array}
[c]{cc}%
\frac{1}{r}, & u\geq e\\
0, & u\leq e
\end{array}
\right.
\]
so that%
\begin{align*}
K(t,Tf;L^{1},L^{\infty})  &  \leq C\int_{e}^{\infty}K(\frac{t}{r}%
,f;L^{1},L^{\infty})\frac{dr}{r}\\
&  \leq C\int_{0}^{t}K(u,f;L^{1},L^{\infty})\frac{du}{u}.
\end{align*}
Therefore,%
\begin{equation}
(Tf)^{\ast\ast}(t)\leq\frac{C}{t}\int_{0}^{t}f^{\ast\ast}(u)du. \label{ex5}%
\end{equation}
Yano's theorem then follows letting $t=1,$ which yields%
\begin{align*}
\left\Vert Tf\right\Vert _{L^{1}(0,1)}  &  \leq C\int_{0}^{1}f^{\ast\ast
}(u)du\\
&  \approx C\left\Vert f\right\Vert _{LLogL(0,1)}.
\end{align*}
The result also holds when dealing with infinite measure spaces, in which case%
\[
(Tf)^{\ast\ast}(1)=\left\Vert Tf\right\Vert _{L^{1}+L^{\infty}},\text{ and
}\int_{0}^{1}f^{\ast\ast}(u)du\approx\left\Vert f\right\Vert _{LLogL+L^{\infty
}},
\]
yielding%
\[
\left\Vert Tf\right\Vert _{L^{1}+L^{\infty}}\leq C\left\Vert f\right\Vert
_{LLogL+L^{\infty}}.
\]
Note that (\ref{ex5}) gives us back (\ref{ex1}). Indeed, since (\ref{ex5}) can
be rewritten as%
\[
\int_{0}^{t}(Tf)^{\ast}(s)ds\leq{C}\int_{0}^{t}f^{\ast\ast}(u)du,\text{ for
all }t>0,
\]
by the Calder\'{o}n-Mityagin principle we have that for all $p\geq1$%
\[
\left\Vert Tf\right\Vert _{L^{p}}\leq C\left\Vert f^{\ast\ast}\right\Vert
_{L^{p}}.
\]
In particular, if $p>1,$ we can continue the estimate of the right hand side
using Hardy's inequality to obtain%
\[
\left\Vert Tf\right\Vert _{L^{p}}\leq C\frac{p}{p-1}\left\Vert f\right\Vert
_{L^{p}}.
\]

\end{example}

\begin{problem}
We ask for a systematic \textquotedblleft elementary\textquotedblright%
\ treatment of extrapolation for general weights $M(p)$.
\end{problem}

\section{$\mathbf{F}$-functors and extrapolation r.i. spaces}

The $e^{L^{1/\alpha}}$ spaces and the $L(LogL)^{\alpha}$ spaces are prototypes
of "extrapolation" or limiting spaces for the scale of $L^{p}$-spaces. Let us
also remark that the $e^{L^{1/\alpha}}$ spaces belong to the class of
Marcinkiewicz spaces, while the $L(LogL)^{\alpha}$ spaces can be seen to be
Lorentz spaces. More generally, it is easy to see that the spaces obtained by
applying the $\Delta-$method of extrapolation to $\{L^{p}\}$ scales can be
described as Marcinkiewicz spaces (cf. Theorem \ref{intersectformula}), on the
other hand, the corresponding $\sum-$spaces can be described as Lorentz spaces
(cf. Theorem \ref{sumformula}). It is then natural to ask for a
characterization of all the Lorentz or Marcinkiewicz spaces that can be
obtained by extrapolation methods applied to scales of $L^{p}$-spaces. More
generally, one would like to describe all the r.i. spaces that can be obtained
by extrapolation of $L^{p}$ spaces. For definiteness, we shall only consider
here r.i. spaces on $[0,1]$. In this setting, $L^{\infty}$ and $L^{1}$ are the
smallest and the largest r.i. spaces, respectively. The prototype scale
associated with the pair $(L^{1},L^{\infty})$ is, of course,
\[
L^{p}=\langle L^{1},L^{\infty}\rangle_{1/p^{\prime},p;K}^{\blacktriangleleft
},\;\;1<p<\infty.
\]
with norm equivalence independent of $p$ \cite[Example~7]{Mil-03}. So our
prototype problem in this section is to characterize certain subclasses of
r.i. spaces $X$ that are in a suitable sense "close" to either $L^{\infty}$ or
to $L^{1}$, and whose norms can be obtained by *extrapolation*, that is, by
using the $\mathbf{F}$-functors of extrapolation that are described briefly in
Appendix~\ref{sec:extrapol}.

Let us consider, for example, the case of spaces $X$ \ "close" to $L^{\infty}%
$, in the sense that $X\subset L^{p},$ for all $p<\infty.$ Our aim is to
describe all r.i. spaces $X$ such that%
\[
X=\mathbf{F}(\{L^{p}\}_{1<p<\infty}),
\]
for some extrapolation functor $\mathbf{F}$, in which case we shall say that
$X$ is an \textit{extrapolation space (at $\infty$).} This can be reformulated
as follows. For each Banach function lattice $F$ on $[1,\infty),$ we let
\[
\mathcal{L}_{F}=\mathbf{F}(\{L^{p}\}_{1<p<\infty})=\{f:[0,1]\rightarrow
\mathbb{R},\text{ such that }\xi_{f}(p):=\Vert f\Vert_{p}\in F\},
\]%
\[
\left\Vert f\right\Vert _{{\mathcal{L}}_{F}}:=\left\Vert \xi_{f}\right\Vert
_{F}.
\]
Since $L^{p}$ is r.i., $\xi_{f}=\xi_{f^{\ast}}$ and consequently
$\mathcal{L}_{F}$ is a r.i. space$.$ Our aim then is to characterize the class
of r.i. spaces $X$, which we denote by $\mathcal{E}_{\infty}$, such that there
exists $F$ so that%
\[
X=\mathcal{L}_{F}\;\;\mbox{(with equivalence of norms)}.
\]
Clearly, this construction is a natural generalization of the $\Delta
$-functor, which corresponds to choosing $F$ to be a weighted $L^{\infty}-$space.

\subsection{Extrapolation characterization of Marcinkiewicz, Orlicz and
Lorentz spaces}

Let $\varphi$ be a quasi-concave function on $[0,1]$. The Marcinkiewicz space
$M(\varphi)$ consists of all measurable functions $f(t)$ on $[0,1]$, such
that
\begin{equation}
\Vert f\Vert_{M(\varphi)}=\sup_{0<s\leq1}{\frac{\varphi(s)}{s}\cdot
\int\limits_{0}^{s}}f{{^{\ast}(t)\,dt=}}\sup_{0<s\leq1}\varphi(s)f{{^{\ast
\ast}(s)\,}}<\infty. \label{mar1}%
\end{equation}

We shall now consider the problem of identifying the Marcinkiewicz spaces
$M(\varphi)$ that belong to $\mathcal{E}_{\infty}$. Suppose then that
$M(\varphi)\subset L^{p}$ for all $1\leq p<\infty$. It follows readily that
$\lim_{t\rightarrow0+}\tilde{\varphi}(t)=0$, where $\tilde{\varphi
}(t):=t/\varphi(t)$. Consequently, the function $\tilde{\varphi}^{\prime}$ is
absolutely continuous on $[0,1],$ $(\tilde{\varphi}^{\prime}){{^{\ast\ast
}(s)=}}\tilde{\varphi}(s)/s$ and therefore we have $\Vert\tilde{\varphi
}^{\prime}\Vert_{M(\varphi)}=1.$ The assumption that $M(\varphi)\subset L^{p}$
for all $p<\infty,$ therefore implies that $\tilde{\varphi}^{\prime}\in L^{p}$
for all $p<\infty.$ Moreover, from the definition (\ref{mar1}) we see that for
all $0<s\leq1,$%
\[
\int\limits_{0}^{s}f{^{\ast}(t)dt} {\leq}\Vert f\Vert_{M(\varphi)}%
\tilde{\varphi}(s) =\Vert f\Vert_{M(\varphi)}\int\limits_{0}^{s}%
{\tilde{\varphi}^{\prime}(t)dt.}%
\]
Therefore, by the Calder\'{o}n-Mityagin theorem (see e.g.
\cite[Theorem~II.4.3]{KPS}), we conclude that for all $f\in M(\varphi)$ and
$1\leq p<\infty$%
\[
\Vert f\Vert_{p}\leq\Vert f\Vert_{M(\varphi)}\cdot\Vert\tilde{\varphi}%
^{\prime}\Vert_{p}.
\]
In other words,
\begin{equation}
M(\varphi)\overset{1}{\subset}{\mathcal{L}}_{F^{\varphi}}, \label{Marc1}%
\end{equation}
where $F^{\varphi}$ is the weighted Banach lattice $L^{\infty}(1/{\Vert
\tilde{\varphi}^{\prime}\Vert_{p}})$.

\begin{remark}
Recall that the fundamental function of a r.i. space $X$ is defined by
$\phi_{X}(t):=\Vert\chi_{\lbrack0,t]}\Vert_{X}$, $0<t\leq1$. In particular,
$\phi_{M(\varphi)}(t)=\varphi(t)$. It follows readily that $M(\varphi)$ is the
largest among all r.i. spaces with the fundamental function $\varphi(t)$ (cf.
\cite[Theorem~II.5.7]{KPS}), the fact which we will need to prove the next proposition.
\end{remark}

\begin{proposition}
\label{prop:Marc} Let $\varphi$ be a quasi-concave function on $[0,1]$. The
following conditions are equivalent:

(i) $M(\varphi)\in\mathcal{E}_{\infty}$;

(ii) $M(\varphi)={\mathcal{L}}_{F^{\varphi}};$

(iii) there exists $C>0$ such that
\begin{equation}
\varphi(t)\leq C\cdot\sup_{p\geq1}\frac{t^{\frac{1}{p}}}{\Vert\tilde{\varphi
}^{\prime}\Vert_{p}},\quad0<t\leq1. \label{mar2}%
\end{equation}

\end{proposition}

\begin{proof}
First of all, each of the conditions (i), (ii), (iii) implies that
$M(\varphi)\subset L^{p}$ for all $p<\infty$. Therefore, in all three cases we
have embedding \eqref{Marc1}.

[(ii)$\leftrightarrow$(iii)]. Suppose that (iii) holds. Then,
\begin{align*}
\varphi(t)  &  \leq C\cdot\sup_{p\geq1}\frac{t^{\frac{1}{p}}}{\Vert
\tilde{\varphi}^{\prime}\Vert_{p}}\\
&  =C\cdot\sup_{p\geq1}\frac{1}{\Vert\tilde{\varphi}^{\prime}\Vert_{p}}%
\Vert\chi_{\lbrack0,t]}\Vert_{p}\\
&  =C\cdot\left\Vert \chi_{\lbrack0,t]}\right\Vert _{{\mathcal{L}}%
_{F^{\varphi}}}\\
&  =C\cdot\varphi_{{\mathcal{L}}_{F^{\varphi}}}(t).
\end{align*}
It follows (cf. the Remark preceding the statement of this Proposition) that
${\mathcal{L}}_{F^{\varphi}}\subset M(\varphi_{{\mathcal{L}}_{F^{\varphi}}%
})\subset M(\varphi),$ which combined with (\ref{Marc1}) yields that
${\mathcal{L}}_{F^{\varphi}}=M(\varphi).$ Conversely, if ${\mathcal{L}%
}_{F^{\varphi}}=M(\varphi),$ then $\varphi\approx\varphi_{{\mathcal{L}%
}_{F^{\varphi}}}$ and (iii) trivially holds.

[(i)$\leftrightarrow$(ii)]. Since the implication $(ii)\Rightarrow(i)$ is
obvious, we only need to prove the converse. Suppose then that $M(\varphi
)={\mathcal{L}}_{F_{1}}$ for some Banach function lattice $F_{1}$. We will now
show that ${\mathcal{L}}_{F_{1}}={\mathcal{L}}_{F^{\varphi}}.$ Indeed, by
(\ref{Marc1}), ${\mathcal{L}}_{F_{1}}=M(\varphi)\overset{1}{\subset
}{\mathcal{L}}_{F^{\varphi}}.$On the other hand, suppose that $g\in
{\mathcal{L}}_{F^{\varphi}},$ then for all $p>1$
\[
\xi_{g}(p)\leq\Vert g\Vert_{{\mathcal{L}}_{F^{\varphi}}}\left\Vert
\tilde{\varphi}^{\prime}\right\Vert _{p}.
\]
Therefore, applying the lattice norm $F_{1}$ to the previous inequality, we
get%
\begin{align*}
\left\Vert \xi_{g}\right\Vert _{F_{1}}  &  \leq\Vert g\Vert_{{\mathcal{L}%
}_{F^{\varphi}}}\left\Vert \left\Vert \tilde{\varphi}^{\prime}\right\Vert
_{p}\right\Vert _{F_{1}}\\
&  =\Vert g\Vert_{{\mathcal{L}}_{F^{\varphi}}}\left\Vert \tilde{\varphi
}^{\prime}\right\Vert _{{\mathcal{L}}_{F_{1}}}\\
&  \approx\Vert g\Vert_{{\mathcal{L}}_{F^{\varphi}}}\left\Vert \tilde{\varphi
}^{\prime}\right\Vert _{M(\varphi)}\\
&  =\Vert g\Vert_{{\mathcal{L}}_{F^{\varphi}}}.
\end{align*}
Consequently,%
\[
\left\Vert g\right\Vert _{{\mathcal{L}}_{F_{1}}}=\left\Vert \xi_{g}\right\Vert
_{F_{1}}\leq C\Vert g\Vert_{{\mathcal{L}}_{F^{\varphi}}},
\]
and therefore $g\in{\mathcal{L}}_{F_{1}},$ concluding the proof.
\end{proof}

\begin{problem}
\label{prob: r.i. spaces0} Let $w(p)$ be a bounded positive function on
$[1,\infty)$. We set $X_{w}:=\Delta_{1\leq p<\infty}(w(p)L^{p})$. Then,
$X_{w}$ is a r.i. space with the fundamental function $\phi_{X_{w}}%
(t)=\sup_{1\leq p<\infty}(w(p)t^{1/p})$. By Proposition~\ref{prop:Marc}, from
$M(\varphi)\in\mathcal{E}_{\infty}$ it follows that $M(\varphi)=X_{w}$ with
$w(p)=1/{\Vert\tilde{\varphi}^{\prime}\Vert_{p}}$. We ask: What other r.i.
spaces may be represented as spaces of the $X_{w}$-type? It is known \cite[Theorem~4.7]{AL} that if
an Orlicz space $L_{M}$ coincides with some Marcinkiewicz space, then $%
L_{M}\in\mathcal{E}_{\infty}$ and therefore it coincides with the space $%
X_{w}$ for some $w$. In contrast to that, in \cite[Proposition 3.4]{ALMA}, one can find examples of the Orlicz spaces of the $X_w$-type that do not coincide with Marcinkiewicz spaces. 
Thus, it is natural to ask, which Orlicz spaces are spaces of the $X_{w}$-type? 
\end{problem}

\begin{problem}
\label{prob: D-concave spaces} In connection with
Problem~\ref{prob: r.i. spaces0}, maybe it could be useful to take into
account that every space $X_{w}$ is $D$-convex (for the definitions, we refer
to \cite{M-S-S} or \cite{ASW}). Moreover, it is known (see \cite[Theorem~23]%
{M-S-S} or \cite[Corollary~4.10]{ASW}) that a r.i. space $X$ with the Fatou
property coincides with an Orlicz space if and only if $X$ is $D$-convex and
$D$-concave (some authors refer to the latter property as $D^{\ast}%
$-convexity). Therefore, we ask: Under what conditions on a weight $w$, is the
space $X_{w}$ $D$-concave?
\end{problem}

\begin{problem}
\label{prob:grand}(Open ended) Let us remark that, with minor modifications,
the framework we are discussing here could be used to derive a generalized
theory of \textquotedblleft Grand Lebesgue spaces\textquotedblright\ (cf.
Section \ref{sec:grand} below for definitions and background). In fact, a
natural setting for generalized \textquotedblleft Grand Lebesgue
spaces\textquotedblright\ could be have by means of replacing $\Delta_{1\leq
p<\infty}(w(p)L^{p})$ by $\Delta_{\theta\in I}(w(\theta)L^{p(\theta)}).$
\end{problem}

\begin{problem}
\label{prob: r.i. spaces1} Give a characterization of Lorentz spaces from the
class $\mathcal{E}_{\infty}$. Recall that the norm in the Lorentz space
$\Lambda_{p}(\varphi)$, where $\varphi$ is an increasing concave function on
$[0,1]$, $\varphi(0)=0$, and $1\leq p<\infty$, is defined as follows:
\[
\Vert f\Vert_{\Lambda_{p}(\varphi)}:=\left(  \int\limits_{0}^{1}(f^{\ast
}(t))^{p}d\varphi(t)\right)  ^{1/p}.
\]
Moreover, we ask to introduce in a similar way a notion of extrapolation r.i.
spaces at $1$, i.e., as $p\rightarrow1+$, which would therefore generalize the
$\Sigma$-functor, and then using this notion to give a description of
Marcinkiewicz, Lorentz, Orlicz spaces that are \textquotedblleft extrapolation
at 1\textquotedblright.
\end{problem}

\begin{remark}
Concerning Problem \ref{prob: r.i. spaces1} we note that some partial results
related to a description of Lorentz spaces from the class $\mathcal{E}%
_{\infty}$ were obtained in \cite[Theorem~3]{AL2006}.
\end{remark}

\begin{problem}
\label{prob:last}The same type of questions can be formulated in the
non-commutative setting. In this context instead of $L^{p}$-spaces, we deal
with the scale of Schatten ideals ${{\mathfrak{S}}}^{p}$, $1<p<\infty$, of
compact operators acting in a separable complex Hilbert space (see
Example~\ref{non-commutative}). It is natural to ask similar questions in
connection with an extrapolation description of Schatten ideals. Some partial
results can be found in \cite{ALM}.
\end{problem}

\begin{problem}
\label{prob: r.i. spaces2} We are asking whether any r.i. space $X$ on $[0,1]$
such that the $X$-norm of every function is determined by the family of its
$L^{p}$-norms coincides with a space of the form ${\mathcal{L}}_{F}$ for a
suitable Banach function lattice $F$ on $[1,\infty)$? More formally, let $X$
be a r.i. space on $[0,1]$ such that $X\subset L_{p}$ for all $p<\infty$.
Suppose that there is $p_{0}>0$ such that from the inequality $\Vert
x\Vert_{p}\leq C\Vert y\Vert_{p}$, for some $C>0$ and all $p\geq p_{0}$, it
follows that $\Vert x\Vert_{X}\leq\Vert y\Vert_{X}$. Does this imply that
$X={\mathcal{L}}_{F}$ for some parameter $F$?
\end{problem}

\begin{problem}
(Open ended) In connection with Problems \ref{prob:grand} and \ref{prob:last}
we are led to ask for the corresponding theory of Non-Commutative Grand
$L^{p}$ spaces. We believe that abstract Extrapolation theory provides the
right tools to develop this project.
\end{problem}

\subsection{Tempered $\mathbf{{F}}$-parameters and strong extrapolation r.i.
spaces}

The following definition, introduced in \cite{AL}, could be considered as a
natural generalization of the notion of a tempered weight.

\begin{definition}
We shall say that a parameter space $F$ (i.e. $F$ is a Banach function lattice
on $[1,\infty)$) of an extrapolation $\mathbf{{F}}$-method is
\textit{tempered} if the operator $Df(p):=f(2p)$ is bounded on $F$.
\end{definition}

It turns out that the spaces $\mathcal{L}_{F}$ with tempered parameters $F$
form a very special subclass of the extrapolation at $\infty$ r.i. spaces.

Let $X$ be a r.i. space on $[0,1]$, we denote by $\tilde{X}$ the Banach
lattice of all the measurable functions $f$ on $(1,\infty)$ such that
\[
{\Vert f\Vert}_{\tilde{X}}:={\Vert f\big(\log(e/t)\big)\Vert}_{X}<\infty.
\]

\begin{definition}
\label{StrongKvSdvig} We shall say that a r.i. space $X$ is a \textit{strong
extrapolation space with respect to the $L^{p}$-scale} (in which case we shall
write $X\in\mathcal{SE_{\infty}})$ if $X={\mathcal{L}}_{\tilde{X}}$ (with
equivalence of norms).
\end{definition}

By definition, if $X\in\mathcal{SE_{\infty}}$, then the corresponding
extrapolation parameter $F$ is explicitly determined by $X$. More precisely,
\[
{\Vert f\Vert}_{X}\approx\big\|\,\Vert f\Vert_{\log(e/t)}\,\big\|_{X},
\]
where, consistently with our notation throughout this paper, for each
$t\in(0,1),$ we let $\Vert f\Vert_{\log(e/t)}:=\Vert f\Vert_{L^{\log(e/t)}}.$

The class $\mathcal{SE_{\infty}}$ admits a simple characterization (see
\cite[Theorem~4.3]{AL2017}).

\begin{theorem}
\label{KrStrExt2} Let $X$ be a r.i. space on $[0,1].$ The following conditions
are equivalent:

(1) $X={\mathcal{L}}_{F}$ $\ $for some tempered extrapolation parameter $F$;

(2) $X\in\mathcal{SE}_{\infty}$;

(3) the operator $Sf(t)=f(t^{2})$ is bounded on $X$.
\end{theorem}

The class $\mathcal{SE_{\infty}}$ is rather wide. In particular, the Zygmund
$\mathrm{Exp}\,L^{\alpha}$ spaces, with $\alpha>0$, lie in this class.
Moreover, a Marcinkiewicz space $\mathcal{M}(\varphi)$ (resp. a Lorentz space
$\Lambda(\varphi)$) belongs to the class $\mathcal{SE_{\infty}}$ if and only
if $\varphi(t)\approx\varphi(t^{2})$, $0<t\leq1$ (cf. \cite[Theorem
2.10]{AL2009}). In this connection it is worth to note that, in the definition
of the space $X(\log^{-1})$ (see Example~\ref{ex:yano}), the $^{\ast\ast}$ may
be replaced by $^{\ast}$, whenever $X\in\mathcal{SE_{\infty}}$ (cf.
\cite[Proposition~4.1]{AL2009}).

One can verify that, for every $1<p<\infty,$ we have
\[
L(2p,\infty)\overset{2}{\subset}L(p,1)\overset{1}{\subset}L(p,\infty).
\]
Hence, if $F$ is a tempered parameter, we have the following generalized
version of the second relation from \eqref{validity} (as it applies to the
pair $\vec{A}=(L^{1},L^{\infty}))$:
\begin{equation}
\mathbf{F}(\{L(p,1)\}_{1<p<\infty})=\mathbf{F}(\{L(p,\infty)\}_{1<p<\infty}),
\label{multiplier for F}%
\end{equation}
where $\mathbf{F}(\{L(p,q)\}_{1<p<\infty})$, $1\leq q\leq\infty$, is defined
exactly as the space $\mathcal{L}_{F}=\mathbf{F}(\{L^{p}\}_{1<p<\infty}),$
replacing $L^{p}$ by $L(p,q)$.

\begin{problem}
\label{prob:multiplier for F}(\textbf{The multiplier problem for
\textbf{F}-functor}) Characterize the parameters for the $\mathbf{{F}}$-method
of extrapolation that have the property \eqref{multiplier for F}.
\end{problem}


\section{Operators with a quasi-Banach target space}

The classical extrapolation theorems deal mainly with linear or sublinear
operators taking values on quasi-Banach spaces. For example we may want to
extrapolate estimates for maximal operators, e.g. the maximal operator $M$ of
Hardy-Littlewood, or the Carleson maximal operator $\mathfrak{C}$ defined by
\begin{equation}
\mathfrak{C}f(e^{i\theta}):=\sup_{N=1,2,\dots}|S_{N}f(e^{i\theta})|,
\label{Carleson}%
\end{equation}
where $S_{N}f(e^{i\theta}):=\sum_{n=-N}^{N}\hat{f}(n)e^{in\theta}$, and
$\hat{f}(n)$ is the $n$-th Fourier coefficient of $f$. Note that for these
examples the "natural target space\textquotedblright\ is the quasi-normed
space $L(1,\infty)$ (see e.g. \cite{AdR}).

By construction, the $\sum-$methods crucially use linearity and the triangle
inequality. But it is possible to modify the construction of $\sum-$methods in
order to be able to deal with some of these difficulties, although we shall
not discuss the issues in detail here (cf. \cite[Section 4, pages 35-44]{jm}).

Here an operator $T$ defined on a Banach space $X$ and taking values in the
set of all measurable functions $f:\,[0,1]\rightarrow\mathbb{R}\cup\{\pm
\infty\},$ is \textit{sublinear} if for some $B>0$ and an arbitrary expansion
$x=\sum_{j=1}^{\infty}x_{j}$, as a convergent series in $X$, we have
\[
|Tx(t)|\leqslant B\sum_{j=1}^{\infty}|Tx_{j}(t)|\quad\mbox{a.e. on}\;[0,1].
\]
\ \ \ \ \ 

Another important class of non-linear operators acting on function spaces are
those for which we have $T(f+g)=Tf+Tg,$ whenever functions $f$ and $g$ have
disjoint supports. It turns out that in the context of lattices one can find
suitable versions of the strong form of the fundamental Lemma (cf.
Appendix~\ref{sec:sfl}, including Remark \ref{re:delotro}) that guarantee the
existence of good representations $f=\sum f_{n}$ such that the $f_{n}$'s are
disjointly supported (cf. \cite{cwni}).

The weak $L^{1}-$space, $L(1,\infty),$ is quasi-normed but it barely misses to
be normable. It belongs to the class of logconvex quasi-Banach lattices. We
shall say that a quasi-Banach space $Y$ is called \textit{logconvex} if there
is a constant $C>0$ such that for all $y_{j}\in Y$ we have
\[
\left\Vert \sum_{j=1}^{\infty}y_{j}\right\Vert _{Y}\leqslant C\sum
_{j=1}^{\infty}(1+\log j){\Vert y_{j}\Vert}_{Y},
\]
(see e.g. \cite{Kalton}). $L(1,\infty)$ is \textit{logconvex} (cf. \cite[Lemma
2.3]{SW} or \cite[Theorem~3.4]{Kalton}).

The following result presents a version\footnote{We refer to
\cite{LykovMathSb} for more general results.} of Yano's extrapolation theorem
for operators that take values in a logconvex space$.$

\begin{theorem}
\cite[Theorem 9]{LykovMathSb}\label{ExtrLogConvexSub} Let $\mathcal{M}$ be the
set of all measurable functions $f:\,[0,1]\rightarrow\mathbb{R}$, and let
$Y\subset\mathcal{M}$ be a logconvex quasi-Banach lattice. Let $\alpha>0,$ and
let $T$ be a sublinear operator defined on the Lorentz space $\Lambda
(\psi_{\alpha})$, where
\[
\psi_{\alpha}(t){\approx}t\log^{\alpha}(b/t)\;\log\log\log(b/t),\quad
0<t\leq1,\quad\mbox{where}\;b>e^{e}.
\]
Furthermore, suppose that $T$ is bounded, $T:$ $L(p,1)\rightarrow Y$, $p>1$,
and for some $C>0$ and all $p>1$ we have
\begin{equation}
{\Vert T\Vert}_{L(p,1)\rightarrow Y}\leqslant C\left(  \frac{p}{p-1}\right)
^{\alpha}. \label{Quasi1}%
\end{equation}

Then, $T$ is a bounded operator
\[
T:\Lambda(\psi_{\alpha})\rightarrow Y.
\]

\end{theorem}

\begin{problem}
To understand Theorem \ref{ExtrLogConvexSub} note that, if $Y$ had been a
Banach space, by the classical extrapolation theorem (cf. Section
\ref{sec:discutido}) we would get that $T$ is bounded, $T:L(LogL)^{\alpha
}\rightarrow Y.$ In other words, the penalty we pay for having only
logconvexity is the extra "triple logarithm" factor that appears in the norm
of $\Lambda(\psi_{\alpha}).$ In this connection we therefore ask if this
result is sharp or if it is possible to enlarge the domain space?
\end{problem}

\subsection{A.e. convergence of Fourier series and extrapolation}

A classical extrapolation due to Carleson--Sj\"{o}lin theorem \cite{Sj-68}
states that if $T$ is a continuous sublinear operator on $L^{p}$ such that for
every measurable set $A\subset\lbrack0,1]$ and all $1<p\leq2,\;t>0$
\[
t\cdot m\{x\in\lbrack0,1]:\,|T\chi_{A}(x)|>t\}^{1/p}\leq C(p-1)^{-1}%
m(A)^{1/p},
\]
with some constant $C>0$ independent of $A$, $p$, and $t$ ($m$ is the Lebesgue
measure), then $T$ maps the space $L(LogL)LogLogL$ into $L(1,\infty)$.
Applying this result to the Carleson maximal operator $M$ (see
\eqref{Carleson}), we immediately get that the Fourier series of each function
from the space $L(LogL)LogLogL$ converges a.e. If $\alpha>0$, and $(p-1)^{-1}$
is replaced with $(p-1)^{-\alpha}$ then an analogous result holds replacing
$L(LogL)LogLogL$ by $L(LogL)^{\alpha}LogLogL$, (cf. \cite[Theorem~5.7.1]{jm}),
but the domain space provided by Theorem \ref%
{ExtrLogConvexSub} is wider than the space $L(LogL)LogLogL$.

\begin{problem}
Do the conditions of the Carleson--Sj\"{o}lin extrapolation theorem imply that
a sublinear operator $T$ maps the larger space $L(LogL)LogLogLogL$ into
$L(1,\infty),$ thus implying a well-known result for Fourier series due to
Antonov \cite{Antonov} ?
\end{problem}

\section{Grand Lebesgue spaces and their versions via
extrapolation\label{sec:grand}}

{Let $1<p<\infty$.} The \textit{Grand Lebesgue $L^{p)}$ space} introduced by
Iwaniec and Sbordone \cite{IwSb}, consists of all measurable functions $f$ on
$[0,1]$ such that
\begin{equation}
{\Vert f\Vert}_{L^{p)}}:=\sup_{0<\varepsilon<p-1}\varepsilon^{\frac
{1}{p-\varepsilon}}{\Vert f\Vert}_{L^{p-\varepsilon}}<\infty. \label{spre}%
\end{equation}
These spaces have found many applications in analysis, including the study of
maximal operators, PDEs, interpolation theory, etc (see \cite{FFG,JSS} and the
references therein). On the other hand, the expression (\ref{spre}) is
somewhat difficult to work with. In this context, Fiorenza-Karadzhov
\cite[Theorem 4.2]{FiorKar} gave the following more explicit description of
the Grand Lebesgue spaces $L^{p)}$ in the terms of the decreasing
rearrangement of the function $f$:
\[
{\Vert f\Vert}_{L^{p)}}\approx\sup_{0<t<1}(\log(e/t))^{-\frac{1}{p}}\left(
\int\limits_{t}^{1}f^{\ast}(s)^{p}\,ds\right)  ^{\frac{1}{p}},
\]
with universal constants of equivalence. The proof given in \cite[Theorem
4.2]{FiorKar} is based on the extrapolation methods of \cite{karami}. A
simpler proof of this result was obtained recently in \cite[Theorem~12]{ALM}
exploiting an extrapolation description of suitable limiting interpolation
spaces. More generally the same method yields that for every $\alpha>0$,%
\[
{\Vert f\Vert}_{L^{p),\alpha}}:=\sup_{0<\varepsilon<p-1}\varepsilon
^{\frac{\alpha}{p-\varepsilon}}{\Vert f\Vert}_{L^{p-\varepsilon}}\approx
\sup_{0<t<1}\log^{-\alpha/p}(e/t)\left(  \int_{t}^{1}(f^{\ast}(s))^{p}%
\,ds\right)  ^{\frac{1}{p}}%
\]
(see \cite[Remark~11]{ALM}), a result obtained for the first time in
\cite[Theorem 1.1 and Theorem 3.1]{FFGKR}.

Let $p>1$ and let $\psi:\,(0,p-1)\rightarrow\lbrack0,\infty)$ be a
nondecreasing function. We shall say that $\psi\in\Delta_{2}$ if $\psi
(2t)\leq\psi(t),$ for small $t$. In \cite[Theorem~1]{FFG} it is shown that the
equivalence
\begin{equation}
\sup_{0<\varepsilon<p-1}\psi(\varepsilon){\Vert f\Vert}_{L^{p-\varepsilon}%
}\approx\sup_{0<t<1}\psi\left(  \frac{p-1}{1-log\,t}\right)  \left(  \int%
_{t}^{1}(f^{\ast}(s))^{p}\,ds\right)  ^{\frac{1}{p}} \label{spre2}%
\end{equation}
holds if and only if $\psi\in\Delta_{2}\cap L^{\infty}$.

\begin{problem}
\label{grand Lebesgue spaces1} Prove equivalence \eqref{spre2} in the case of
non-power functions $\psi\in\Delta_{2}\cap L^{\infty}$, using extrapolation
methods. What is the corresponding equivalence formula when the $\Delta_{2}%
$-assumption does not hold?
\end{problem}

\section{Bilinear Extrapolation: Calder\'{o}n's operator revisited}

\label{bilinear}

The $K/J$ inequalities can be extended to bilinear operators and used to prove
extrapolation theorems following the scheme used to treat the linear case (cf.
\cite{jm2}). We briefly review the story here and present a number of open problems.

We start by recalling that, in the classical setting, the interpolation of
bilinear operators can be effected using $K/J$ inequalities (cf. \cite{jm2}).
In this case starting with weak type inequalities for a bilinear operator $T$
we control the $K-$functional of $K(t,T(f,g))$ via an analog of the
Calder\'{o}n operator that this time is expressed as the multiplicative
convolution of the $K$-functionals of $f$ and $g.$ We recall the details. Let
$\vec{A},\vec{B},\vec{C}$ be pairs of mutually closed spaces and let $T$ be a
bounded bilinear operator $T:$ $\vec{A}\times\vec{B}\rightarrow\vec{C}.$ Then
it is easy to see that there exists $c>0$ such that (cf. \cite{jm2})%
\[
J(rs,T(f,g);\vec{C})\leq cJ(r,f;\vec{A})J(s,g;\vec{B}),\;\;r,s>0.
\]
If we combine this fact with the basic elementary $K/J$ inequality: For any
pair $\vec{X}$, and $f\in\Delta(\vec{X}),$%
\[
K(t,h;\vec{X})\leq\min\{1,\frac{t}{s}\}J(s,h;\vec{X}),
\]
we obtain for $f\in\Delta(\vec{A})$ and $g\in\Delta(\vec{B}),$%
\begin{align}
K(t,T(f,g);\vec{C})  &  \leq\min\{1,\frac{t}{rs}\}J(rs,T(f,g);\vec
{C})\nonumber\\
&  \leq c\min\{1,\frac{t}{rs}\}J(r,f;\vec{A})J(s,g;\vec{B}). \label{bil1}%
\end{align}
Suppose now that $\gamma_{a},\gamma_{b}$ and $\gamma_{c}$ are quasi-concave
functions such that%
\[
\frac{1}{\gamma_{c}(uv)}\leq\frac{1}{\gamma_{a}(u)}\frac{1}{\gamma_{b}%
(v)},\;\;u,v>0
\]
and furthermore suppose that%
\[
1\leq q_{i}\leq\infty,i=1,2,3,\text{ and }\frac{1}{q_{3}}=\frac{1}{q_{1}%
}+\frac{1}{q_{2}}-1.
\]
Then,%
\[
T:\vec{A}_{\gamma_{a}},_{q_{1};K}\times\vec{B}_{\gamma_{b}},_{q_{2}%
;K}\rightarrow\vec{C}_{\gamma_{c}},_{q_{3};K}.
\]
We go over the proof. Represent $f=\int_{0}^{\infty}u_{f}(s)\frac{ds}{s},$ and
$g=\int_{0}^{\infty}u_{g}(s)\frac{ds}{s},$ so that $J(r,u_{f}(r);\vec{A})\leq
K(r,f;\vec{A}),$ and $\int_{0}^{\infty}\min\{1,\frac{t}{s}\}J(s,u_{g}%
(s);\vec{B})\frac{ds}{s}\preceq K(t,g;\vec{B}).$ Now applying the
$K$-functional to the representation
\[
T(f,g)=\int_{0}^{\infty}\int_{0}^{\infty}T(u_{f}(r),u_{g}(s))\frac{dr}{r}%
\frac{ds}{s},
\]
by \eqref{bil1}, yields
\begin{align*}
\frac{K(t,T(f,g);\vec{C})}{\gamma_{c}(t)}  &  \leq\int_{0}^{\infty}\int%
_{0}^{\infty}\frac{K(t,T(u_{f}(r),u_{g}(s));\vec{C})}{\gamma_{c}(t)}\frac
{dr}{r}\frac{ds}{s}\\
&  \leq c\int_{0}^{\infty}\int_{0}^{\infty}\frac{\min\{1,\frac{t}{rs}%
\}}{\gamma_{c}(\frac{t}{r}r)}J(r,u_{f}(r);\vec{A})J(s,u_{g}(s);\vec{B}%
)\frac{dr}{r}\frac{ds}{s}\\
&  \leq c\int_{0}^{\infty}\frac{K(r,f;\vec{A})}{\gamma_{a}(r)}\frac{1}%
{\gamma_{b}(\frac{t}{r})}\int_{0}^{\infty}\min\{1,\frac{t}{rs}\}J(s,u_{g}%
(s);\vec{B})\frac{ds}{s}\frac{dr}{r}\\
&  \preceq\int_{0}^{\infty}\frac{K(r,f;\vec{A})}{\gamma_{a}(r)}\frac
{K(\frac{t}{r},g;\vec{B})}{\gamma_{b}(\frac{t}{r})}\frac{dr}{r}\\
&  =\frac{K(\circ,f;\vec{A})}{\gamma_{a}(\circ)}\blacklozenge\frac
{K(\circ,g;\vec{B})}{\gamma_{b}(\circ)},\text{ where }\blacklozenge\text{ is
convolution w.r. to }(\mathbb{R}^{+},\frac{dt}{t}).
\end{align*}
Consequently, by Young's convolution inequality for the multiplicative group,
there exists a constant $c>0$ such that%
\[
\left\Vert T(f,g)\right\Vert _{\vec{C}_{\gamma_{c}},q_{3};K}\leq c\left\Vert
f\right\Vert _{\vec{A}_{\gamma_{a}},q_{1};K}\left\Vert g\right\Vert _{\vec
{B}_{\gamma_{b}},q_{2};K}.
\]

We now consider the extrapolation case. Let $T$ be a bilinear operator, let
$M\left(  \theta\right)  ,N(\theta)$ be weights such that for all $\theta
\in(0,1),$%
\begin{equation}
T:M(\theta)\vec{A}_{_{\theta,1:J}}^{\blacktriangleleft}\times N(\theta)\vec
{B}_{_{\theta},_{1;J}}^{\blacktriangleleft}\rightarrow\vec{C}_{_{\theta
,\infty;K}}^{\blacktriangleleft},\text{ with norm }1,\text{ for all }\theta
\in(0,1). \label{bil0}%
\end{equation}
Then we have the following bilinear Calder\'{o}n type extrapolation version of
Theorem \ref{teo:ca2} (cf. the discussion that follows it), and Theorem
\ref{teo:ca1}.

\begin{theorem}
\label{teo:local}The following are equivalent:

(i) (\ref{bil0}) holds.

(ii) (bilinear $K/J$ inequality) There exists a constant $c>0,$ such that $T$
satisfies the following (compare with (\ref{bil1}) above): for all $f\in
\Delta(\vec{A}),g\in\Delta(\vec{B}),t,s,h>0,$%
\begin{equation}
K(t,T(f,g);\vec{C})\leq c\tau(\frac{t}{sh})J(s,f;\vec{A})J(h,g;\vec{B})
\label{bil2}%
\end{equation}
where%
\[
\tau(x)=\inf_{0<\theta<1}\{x^{\theta}M(\theta)N(\theta)\}.
\]

(iii) There exists a constant $c>0$ such that%
\[
K(t,T(f,g);\vec{C})\leq c\inf\{\int_{0}^{\infty}\int_{0}^{\infty}\tau(\frac
{t}{sh})J(s,u_{f}(s);\vec{A})J(h,u_{g}(h);\vec{B})\frac{ds}{s}\frac{dh}{h}\}
\]
where the infimum is taken over all the $J-$decompositions of $f=\int%
_{0}^{\infty}u_{f}(s)\frac{ds}{s},$ and $g=\int_{0}^{\infty}u_{g}(h)\frac
{dh}{h}.$
\end{theorem}

\begin{proof}
If (i) holds then for all $f\in\Delta(\vec{A}),g\in\Delta(\vec{B}),\theta
\in(0,1),$%
\begin{align*}
K(t,T(f,g),\vec{C})  &  \leq t^{\theta}\left\Vert T(f,g)\right\Vert _{\vec
{C}_{\theta,\infty:K}^{\blacktriangleleft}}\\
&  \leq t^{\theta}M(\theta)\left\Vert f\right\Vert _{\vec{A}_{\theta
,1:J}^{\blacktriangleleft}}N(\theta)\left\Vert g\right\Vert _{\vec{B}%
_{\theta,1;J}^{\blacktriangleleft}}\\
&  \leq t^{\theta}h^{-\theta}s^{-\theta}M(\theta)N(\theta)J(s,f;\vec
{A})J(h,g;\vec{B}).
\end{align*}
Consequently,
\begin{align*}
K(t,T(f,g),\vec{C})  &  \leq\inf_{\theta}\left(  \left(  \frac{t}{sh}\right)
^{\theta}M(\theta)N(\theta))\right)  J(s,f;\vec{A})J(h,g;\vec{B})\\
&  =\tau(\frac{t}{sh})J(s,f;\vec{A})J(h,g;\vec{B}).
\end{align*}

This takes care of the implication (i)$\rightarrow$(ii). The implication
(ii)$\rightarrow$(iii) follows from the triangle inequality. Indeed, suppose
that $f=\int_{0}^{\infty}u_{f}(s)\frac{ds}{s},g=\int_{0}^{\infty}u_{g}%
(h)\frac{dh}{h}$, then%
\begin{align*}
T(f,g)  &  =T(\int_{0}^{\infty}u_{f}(s)\frac{ds}{s},\int_{0}^{\infty}%
u_{g}(h)\frac{dh}{h})\\
&  =\int_{0}^{\infty}\int_{0}^{\infty}T(u_{f}(s),u_{g}(h))\frac{ds}{s}%
\frac{dh}{h}%
\end{align*}
and therefore, applying (\ref{bil2}), we find that for all $t>0,$%
\begin{align}
K(t,T(f,g),\vec{C})  &  \leq\int_{0}^{\infty}\int_{0}^{\infty}K(t,T(u_{f}%
(s),u_{g}(h));\vec{C})\frac{ds}{s}\frac{dh}{h}\nonumber\\
&  \leq\int_{0}^{\infty}\int_{0}^{\infty}\tau(\frac{t}{sh})J(s,u_{f}%
(s);\vec{A})J(h,u_{g}(h);\vec{B})\frac{ds}{s}\frac{dh}{h}. \label{bil8}%
\end{align}
Finally, if (iii) holds then for each $\theta\in(0,1)$ and $\varepsilon>0$ we
can select decompositions $f=\int_{0}^{\infty}u_{f}(s)\frac{ds}{s},g=\int%
_{0}^{\infty}u_{g}(h)\frac{dh}{h},$ such that%
\[
\int_{0}^{\infty}J(s,u_{f}(s);\vec{A})s^{-\theta}\frac{ds}{s}\leq
(1+\varepsilon)\left\Vert f\right\Vert _{\vec{A}_{\theta,1:J}%
^{\blacktriangleleft}},\text{ \ }\int_{0}^{\infty}J(h,u_{g}(h);\vec
{A})h^{-\theta}\frac{ds}{s}\leq(1+\varepsilon)\left\Vert g\right\Vert
_{\vec{B}_{\theta,1;J}^{\blacktriangleleft}}.
\]
Moreover, since by definition, for each $\theta\in(0,1),t,s,h>0$
\[
\tau(\frac{t}{sh})\leq M(\theta)N(\theta)\frac{t^{\theta}}{s^{\theta}%
h^{\theta}},
\]
it follows that%
\begin{align*}
&  \int_{0}^{\infty}\int_{0}^{\infty}\tau(\frac{t}{sh})J(s,u_{f}(s);\vec
{A})J(h,u_{g}(h);\vec{B})\frac{ds}{s}\frac{dh}{h}\\
&  \leq t^{\theta}M(\theta)\int_{0}^{\infty}J(s,u_{f}(s);\vec{A})s^{-\theta
}\frac{ds}{s}N(\theta)\int_{0}^{\infty}J(h,u_{g}(h);\vec{B})h^{-\theta}%
\frac{dh}{h}\\
&  \leq(1+\varepsilon)^{2}t^{\theta}M(\theta)\left\Vert f\right\Vert _{\vec
{A}_{\theta,1:J}^{\blacktriangleleft}}N(\theta)\left\Vert g\right\Vert
_{\vec{B}_{\theta,1;J}^{\blacktriangleleft}}%
\end{align*}
Combining this inequality with \eqref{bil8}, and letting $\varepsilon
\rightarrow0,$ we get%
\[
\left\Vert T(f,g)\right\Vert _{\vec{C}_{\theta,\infty:K}^{\blacktriangleleft}%
}=\sup_{t>0}K(t,T(f,g),\vec{C})t^{-\theta}\leq M(\theta)\left\Vert
f\right\Vert _{\vec{A}_{\theta,1:J}^{\blacktriangleleft}}N(\theta)\left\Vert
g\right\Vert _{\vec{B}_{\theta,1;J}^{\blacktriangleleft}}.
\]

\end{proof}

Extrapolation theorems follow from assumptions about $\tau.$ In particular, to
the assumptions of the previous theorem we add the following property that is
particularly suitable for our method.

\begin{definition}
Let us say that a concave function $\tau:(0,\infty)\rightarrow R_{+},$ is
adequate if there exists a measure $\nu$ on $(0,\infty)$ \ such that $\tau$
can be represented \ by%
\begin{equation}
\tau(t)=\int_{0}^{\infty}\int_{0}^{\infty}\min(1,\frac{t}{n})\min(1,\frac
{n}{r})d\nu(r)\frac{dn}{n},t>0. \label{bil9}%
\end{equation}

\end{definition}

\begin{theorem}
\label{teo:bilg}Let $\vec{A},\vec{B},\vec{C},$ be mutually closed Banach pairs
and let $T$ be a bilinear operator such that (\ref{bil0}) holds where,
moreover, $M,$ $N$ are weights such that the function%
\[
\tau(t)=\inf_{0<\theta<1}\{t^{\theta}M(\theta)N(\theta)\}.
\]
is adequate. Then,%
\[
K(t,T(f,g);\vec{C})\leq c\int_{0}^{\infty}\int_{0}^{\infty}K(\frac{t}%
{u},f;\vec{A})K(\frac{u}{r},g;\vec{B})d\nu(r)\frac{du}{u}.
\]

\end{theorem}

\begin{proof}
Using the strong form of fundamental Lemma we select $J-$decompositions of
$f=\int_{0}^{\infty}$ $u_{f}(s)\frac{ds}{s}$ and $g=\int_{0}^{\infty}$
$u_{g}(h)\frac{dh}{h}$ such that%
\[
\int_{0}^{\infty}\min(1,\frac{t}{s})J(s,u_{f}(s);\vec{A})\frac{ds}{s}\preceq
K(t,f;\vec{A});\text{ }\int_{0}^{\infty}\min(1,\frac{t}{h})J(h,u_{g}%
(h);\vec{B})\frac{dh}{h}\preceq K(t,g;\vec{B}).
\]
In view of Theorem \ref{teo:local} we know that (\ref{bil8}) above holds for
the above decompositions. To estimate the resulting right hand side of
(\ref{bil8}) we use the fact that $\tau$ is adequate. Then, we can estimate
$K(t,T(f,g),\vec{C})$ by%
\begin{align*}
&  \int_{0}^{\infty}\int_{0}^{\infty}\tau(\frac{t}{sh})J(s,u_{f}(s);\vec
{A})J(h,u_{g}(h);\vec{B})\frac{ds}{s}\frac{dh}{h}\\
&  =\int_{0}^{\infty}\int_{0}^{\infty}\int_{0}^{\infty}\int_{0}^{\infty}%
\min(1,\frac{t}{nsh})\min(1,\frac{n}{r})J(s,u_{f}(s);\vec{A})J(h,u_{g}%
(h);\vec{B})\frac{ds}{s}\frac{dh}{h}d\nu(r)\frac{dn}{n}\\
&  =\int_{0}^{\infty}\int_{0}^{\infty}\int_{0}^{\infty}\min(1,\frac{n}%
{r})J(h,u_{g}(h);\vec{B})\left(  \int_{0}^{\infty}\min(1,\frac{t}%
{nsh})J(s,u_{f}(s);\vec{A})\frac{ds}{s}\right)  \frac{dh}{h}d\nu(r)\frac
{dn}{n}\\
&  \preceq\int_{0}^{\infty}\int_{0}^{\infty}J(h,u_{g}(h);\vec{B})\int%
_{0}^{\infty}\min(1,\frac{n}{r})K(\frac{t}{nh},f;\vec{A})\frac{dn}{n}\frac
{dh}{h}d\nu(r)\\
&  =\int_{0}^{\infty}\int_{0}^{\infty}\int_{0}^{\infty}J(h,u_{g}(h);\vec
{B})\min(1,\frac{u}{rh})K(\frac{t}{u},f;\vec{A})\frac{du}{u}\frac{dh}{h}%
d\nu(r)\text{ \ \ \ \ (let\ }u=nh)\\
&  =\int_{0}^{\infty}\int_{0}^{\infty}\left(  \int_{0}^{\infty}J(h,u_{g}%
(h);\vec{B})\min(1,\frac{u}{rh})\frac{dh}{h}\right)  K(\frac{t}{u},f;\vec
{A})\frac{du}{u}d\nu(r)\\
&  \preceq\int_{0}^{\infty}\int_{0}^{\infty}K(\frac{t}{u},f;\vec{A})K(\frac
{u}{r},g;\vec{B})\frac{du}{u}d\nu(r),
\end{align*}
Thus,%
\[
K(t,T(f,g),\vec{C})\preceq\int_{0}^{\infty}\int_{0}^{\infty}K(\frac{t}%
{u},f;\vec{A})K(\frac{u}{r},g;\vec{B})\frac{du}{u}d\nu(r),
\]
as we wished to show.
\end{proof}

\begin{example}
\label{exa:becomes3} Suppose that
\[
M(\theta)N(\theta)\approx\theta^{-1}(1-\theta)^{-1}%
\]
and let
\[
\tau(t)=\inf_{\theta}\{t^{\theta}M(\theta)N(\theta)\}.
\]
Then, by direct computation we see that (cf. \cite[page 50]{jm}, Example
\ref{ex:yano} for similar computations)
\begin{align*}
\int_{0}^{\infty}\min(1,\frac{t}{n})\min(1,n)\frac{dn}{n}  &  \approx\left\{
\begin{array}
[c]{cc}%
2t+t\log\frac{1}{t} & 0<t<1\\
2+\log t & t>1
\end{array}
\right. \\
&  \approx\inf_{\theta}\{t^{\theta}\theta^{-1}(1-\theta)^{-1}\}\\
&  =\tau(t).
\end{align*}
Consequently, (\ref{bil9}) holds if we select $\nu$ to be the delta-function
at $1,$ $\nu=\delta_{1}.$ Indeed,%
\begin{align*}
\int_{0}^{\infty}\int_{0}^{\infty}\min(1,\frac{t}{n})\min(1,\frac{n}{r}%
)\frac{dn}{n}d\nu(r)  &  =<\int_{0}^{\infty}\int_{0}^{\infty}\min(1,\frac
{t}{n})\min(1,\frac{n}{r})\frac{dn}{n}d\delta_{r=1}(r)>\\
&  =\int_{0}^{\infty}\min(1,\frac{t}{n})\min(1,n)\frac{dn}{n}.
\end{align*}

\end{example}

\begin{example}
\label{ex:lagbi}(Yano's bilinear extrapolation) Let $\vec{A},\vec{B},\vec{C},$
be mutually closed ordered (with constant $1$) Banach pairs, moreover, let
$M,$ $N$ and $\tau$ be as in the previous Example. Suppose that $T$ is a
bilinear operator such that (\ref{bil0}) holds. Then

(i)%
\[
T:\langle\vec{A}\rangle_{0,1;K}\times\langle\vec{B}\rangle_{0,1;K}\rightarrow
C_{0}.
\]

(ii)%
\[
T:A_{1}\times B_{1}\rightarrow Exp(\vec{C})=\{f:\left\Vert f\right\Vert
_{Exp(\vec{C})}=\sup_{0<t<1}\frac{K(t,f;\vec{C})}{t(1+\log\frac{1}{t})}%
<\infty\}.
\]

\end{example}

\begin{proof}
(i) By Theorem \ref{teo:bilg} and the previous example%
\begin{align*}
K(t,T(f,g),\vec{C})  &  \preceq\int_{0}^{\infty}\int_{0}^{\infty}K(\frac{t}%
{u},f;\vec{A})K(\frac{u}{r},g;\vec{B})d\nu(r)\frac{du}{u}\\
&  =\int_{0}^{\infty}K(\frac{t}{u},f;\vec{A})K(u,g;\vec{B})\frac{du}{u}.
\end{align*}
Letting $t=1,$we get
\begin{align*}
\left\Vert T(f,g)\right\Vert _{C_{0}}  &  =K(1,T(f,g),\vec{C})\\
&  \preceq\int_{0}^{\infty}K(\frac{1}{u},f;\vec{A})K(u,g;\vec{B})\frac{du}%
{u}\\
&  =\int_{0}^{1}K(\frac{1}{u},f;\vec{A})K(u,g;\vec{B})\frac{du}{u}+\int%
_{1}^{\infty}K(\frac{1}{u},f;\vec{A})K(u,g;\vec{B})\frac{du}{u}.
\end{align*}
Now,%
\[
K(\frac{1}{u},f;\vec{A})\chi_{(0,1)}(u)\leq\left\Vert f\right\Vert _{A_{0}%
},\text{ while \ }K(u,g;\vec{B})\chi_{(1,\infty)}(u)\leq\left\Vert
g\right\Vert _{B_{0}}.
\]
Consequently,%
\begin{align*}
\left\Vert T(f,g)\right\Vert _{C_{0}}  &  \preceq\left\Vert f\right\Vert
_{A_{0}}\int_{0}^{1}K(u,g;\vec{B})\frac{du}{u}+\left\Vert g\right\Vert
_{B_{0}}\int_{1}^{\infty}K(\frac{1}{u},f;\vec{A})\frac{du}{u}\\
&  \preceq\left\Vert f\right\Vert _{A_{0}}\left\Vert g\right\Vert
_{\langle\vec{B}\rangle_{0,1;K}}+\left\Vert g\right\Vert _{B_{0}}\left\Vert
f\right\Vert _{\langle\vec{A}\rangle_{0,1;K}}\\
&  \preceq\left\Vert f\right\Vert _{\langle\vec{A}\rangle_{0,1;K}}\left\Vert
g\right\Vert _{\langle\vec{B}\rangle_{0,1;K}},
\end{align*}
as we wished to show.

(ii) For $t\in(0,1)$ we can write%
\begin{align*}
K(t,T(f,g),\vec{C})  &  \preceq\int_{0}^{\infty}K(\frac{t}{u},f;\vec
{A})K(u,g;\vec{B})\frac{du}{u}\\
&  \leq\int_{0}^{t}....\frac{du}{u}+\int_{t}^{1}...\frac{du}{u}+\int%
_{1}^{\infty}...\frac{du}{u}\\
&  =(I)+(II)+(III).
\end{align*}
We estimate the terms on the right hand side as follows,%

\begin{align*}
(I)  &  \leq\int_{0}^{t}K(\frac{t}{u},f;\vec{A})K(u,g;\vec{B})\frac{du}{u}\\
&  \leq\left\Vert f\right\Vert _{A_{0}}\int_{0}^{t}\left\Vert g\right\Vert
_{B_{1}}du\\
&  \leq t\left\Vert f\right\Vert _{A_{0}}\left\Vert g\right\Vert _{B_{1}}.
\end{align*}
Likewise,%
\begin{align*}
(II)  &  =\int_{t}^{1}K(\frac{t}{u},f;\vec{A})K(u,g;\vec{B})\frac{du}{u}\\
&  \leq\left\Vert f\right\Vert _{A_{1}}\int_{t}^{1}\frac{t}{u}K(u,g;\vec
{B})\frac{du}{u}\\
&  \leq\left\Vert f\right\Vert _{A_{1}}\left\Vert g\right\Vert _{A_{1}}%
t\int_{t}^{1}\frac{du}{u}\\
&  \leq\left\Vert f\right\Vert _{A_{1}}\left\Vert g\right\Vert _{A_{1}}%
t\log\frac{1}{t},
\end{align*}

\begin{align*}
(III)  &  =\int_{1}^{\infty}K(\frac{t}{u},f;\vec{A})K(u,g;\vec{B})\frac{du}%
{u}\\
&  \leq\left\Vert f\right\Vert _{A_{1}}\left\Vert g\right\Vert _{A_{0}}%
t\int_{1}^{\infty}\frac{du}{u^{2}}\\
&  =\left\Vert f\right\Vert _{A_{1}}\left\Vert g\right\Vert _{A_{0}}t.
\end{align*}
Consequently,%
\begin{align*}
\frac{K(t,T(f,g),\vec{C})}{t(1+\log\frac{1}{t})}  &  \preceq\frac{t\left\Vert
f\right\Vert _{A_{0}}\left\Vert g\right\Vert _{B_{0}}}{t(1+\log\frac{1}{t}%
)}+\frac{\left\Vert f\right\Vert _{A_{1}}\left\Vert g\right\Vert _{A_{1}}%
t\log\frac{1}{t}}{t(1+\log\frac{1}{t})}+\frac{\left\Vert f\right\Vert _{A_{1}%
}\left\Vert g\right\Vert _{A_{0}}t}{t(1+\log\frac{1}{t})}\\
&  \preceq\left\Vert f\right\Vert _{A_{1}}\left\Vert g\right\Vert _{B_{1}}.
\end{align*}

\end{proof}

\begin{example}
For a finite measure space we let $\vec{A}=\vec{B}=\vec{C}=(L^{1},L^{\infty
}).$ Then, if $T:L(p,1)\times$ $L(p,1)\rightarrow L^{1},$ with norm $\sim
\frac{p}{p-1},1<p<\infty,$ then

(i)%
\[
T:L(\log L)\times L(\log L)\rightarrow L^{1}.
\]

(ii)%
\[
T:L^{\infty}\times L^{\infty}\rightarrow e^{L}.
\]

\end{example}

\begin{proof}
(i) Follows from the previous Example if we recall that $\langle
L^{1},L^{\infty}\rangle_{0,1;K}=L(\log L).$

(ii) Since $K(t,f;L^{1},L^{\infty})=tf^{\ast\ast}(t),$ we have%
\[
\left\Vert f\right\Vert _{Exp(L^{1},L^{\infty})}=\sup_{0<t<1}\frac
{K(t,f;\vec{C})}{t(1+\log\frac{1}{t})}=\sup_{0<t<1}\frac{f^{\ast\ast}%
(t)}{(1+\log\frac{1}{t})}%
\]
and therefore (cf. \cite{jm})%
\[
Exp(L^{1},L^{\infty})=e^{L}.
\]

\end{proof}

\begin{remark}
Let us note that part (ii) of Example \ref{ex:lagbi} can be obtained directly
by linear extrapolation. Indeed, freezing one variable, say by letting $f\in
A_{1}$ be fixed, then we have a linear operator $T_{f}:=T(f,\circ):$%
\[
T_{f}:M(\theta)N(\theta)\vec{B}_{_{\theta},_{1;J}}^{\blacktriangleleft
}\rightarrow\vec{C}_{_{\theta,\infty;K}}^{\blacktriangleleft}%
\]
yielding%
\[
\left\Vert T_{f}(g)\right\Vert _{Exp(\vec{C})}\preceq\left\Vert f\right\Vert
_{A_{1}}\left\Vert g\right\Vert _{B_{1}}.
\]

\end{remark}

\begin{example}
The case where the norm decays as $\sim\left(  \frac{p}{p-1}\right)  ^{\alpha
},\alpha>1,$ was treated in \cite[Theorem 4.10]{jm2}.
\end{example}

\begin{problem}
Develop the corresponding extrapolation results for the general case. That is
suppose now that $\{\rho_{\theta,a}\}_{\theta\in(0,1)},\{\rho_{\theta
,b}\}_{\theta\in(0,1)},\{\rho_{\theta,c}\}_{\theta\in(0,1)}$ are families of
quasi-concave functions, let $M(\theta),N(\theta)$ be weights, and let $T$ be
a bilinear operator such that for all $\theta\in(0,1),$%
\[
T:M(\theta)\vec{A}_{\rho_{\theta,a}}^{\blacktriangleleft},_{1;J}\times
N(\theta)\vec{B}_{\rho_{\theta,b}}^{\blacktriangleleft},_{1;J}\rightarrow
\vec{C}_{\rho_{\theta,c}}^{\blacktriangleleft},_{\infty;K},\text{ with norm
}1,\theta\in(0,1).
\]
Then the following $K/J$ inequality holds (compare with (\ref{bil2})): for all
$t,u,s>0,f\in\Delta(\vec{A}),g\in\Delta(\vec{B}),$%
\begin{equation}
K(t,T(f,g);\vec{C})\leq J(s,f;\vec{A})J(u,g;\vec{B})\inf_{\theta}\{\frac
{\rho_{\theta,c}(t)}{\rho_{\theta,a}(s)\rho_{\theta,b}(u)}M(\theta
)N(\theta)\},. \label{desi1}%
\end{equation}
Develop the corresponding $K/J$ inequalities and extrapolation results (cf.
\cite{jm2}).
\end{problem}

\begin{problem}
(Open ended) Using Theorem \ref{teo:local} and the results of the previous
Problem develop a complete theory of bilinear extrapolation.
\end{problem}

\begin{problem}
Part of the difficulties of dealing with extrapolation of bilinear operators
lies with the theory of representation of concave functions of two variables.
We ask for a representation formula for concave functions of two variables.
\end{problem}

\begin{problem}
Is it possible to eliminate the restriction for $\tau$ to be adequate (cf.
(\ref{bil9}))?
\end{problem}

\begin{problem}
Arrange conditions on the function $\tau$ so that to be able to control the
decay in each component separately and in this fashion prove bilinear results
where the resulting spaces are different in each component.
\end{problem}

\begin{problem}
In view of Example \ref{ex:lag}, Theorem \ref{teo:local} and Example
\ref{ex:lagbi} one may extend the definition of weak type $(1,1),(\infty
,\infty)$ of \cite{devr} to the bilinear case by demanding that a bilinear
operator satisfies%
\begin{equation}
K(t,T(f,g),\vec{C})\preceq\int_{0}^{\infty}K(\frac{t}{u},f;\vec{A}%
)K(u,g;\vec{B})\frac{du}{u}. \label{devr1}%
\end{equation}
For example, one can then ask for an extension to bilinear operators of the
result in \cite{bedesa}: Do bilinear mappings such that $T:L^{1}\times
L^{1}\rightarrow L(1,1),$ and $L^{\infty}\times L^{\infty}\rightarrow
L(\infty,\infty)$ satisfy (\ref{devr1})?
\end{problem}

\section{Converse to Yano's theorem: Tao's Theorem\label{sec:Tao}}

It is natural to ask if one can prove a converse to Yano's theorem. In other
words, suppose that an operator $T$ is bounded, $T:L(LogL)^{\alpha
}(0,1)\rightarrow L^{1}(0,1),$ then we ask: does there exist a constant $c$
such that $\left\Vert T\right\Vert _{L^{p}\rightarrow L^{p}}\leq\frac
{c}{(p-1)^{\alpha}},$ as $p\rightarrow1?$

It is well known and easy to see that, for general operators, the answer is
negative. and therefore one needs to impose more assumptions on the operators.
A positive result in this direction was obtained by Tao \cite{tao}, who
considers translation invariant operators. More precisely, Tao considered
translation invariant operators $T$ defined on compact symmetric spaces $X$
that are provided with a compact symmetry group $G.$ In this context Tao
\cite{tao} shows that%
\begin{equation}
\left\Vert T\right\Vert _{L^{p}(X)\rightarrow L^{p}(X)}\leq\frac
{c}{(p-1)^{\alpha}},1<p<2\Leftrightarrow T:L(LogL)^{\alpha}(X)\rightarrow
L^{1}(X). \label{tao}%
\end{equation}
So we ask

\begin{problem}
What other types of conditions guarantee the validity of a Yano-Tao theorem?
In particular, we ask this question in the context of non-commutative $L^{p}$ spaces.
\end{problem}

\subsection{Multiplier Problem II: Equivalence of $K-$functional inequalities}

In extrapolation, constants can create what, at first, may appear to be
unexpected effects. The situation that we shall describe now already appears
in some form or other in previous discussions but here we shall consider only
a typical concrete example connected with the extrapolation of $L^{p}%
:=L^{p}(0,1)$ spaces. Let us consider the \textquotedblleft Yano
situation\textquotedblright:%
\begin{equation}
T:\{\frac{p}{p-1}L^{p}\}_{p>1}\rightarrow\{L^{p}\}_{p>1}, \label{mul1}%
\end{equation}
which, as we know, yields, via the $\sum-$method:%
\begin{equation}
T:L\log L=\sum\limits_{p>1}\{\frac{p}{p-1}L^{p}\}_{p>1}\rightarrow\sum
_{p>1}\{L^{p}\}_{p>1}=L^{1}. \label{mul1'}%
\end{equation}
Now, multiplying the underlying inequalities of $($\ref{mul1}$)$ by
$\{\frac{p}{p-1}\}$ yields%
\begin{equation}
T:\{\left(  \frac{p}{p-1}\right)  ^{2}L^{p}\}_{p>1}\rightarrow\{\left(
\frac{p}{p-1}\right)  L^{p}\}_{p>1} \label{mul2}%
\end{equation}
which, once again by the $\sum-$method, gives%
\begin{equation}
T:L(\log L)^{2}=\sum\limits_{p>1}\{\left(  \frac{p}{p-1}\right)  ^{2}%
L^{p}\}_{p>1}\rightarrow\sum_{p>1}\{\left(  \frac{p}{p-1}\right)
L^{p}\}_{p>1}=LLogL. \label{mul3}%
\end{equation}
Conversely, starting from (\ref{mul2}) we can, by multiplication by
$\{\frac{p-1}{p}\}$ return (\ref{mul1}). On the other hand, once we have
applied the $\sum$-functor we cannot claim the equivalence between the pair of
estimates $\{$(\ref{mul1}) and (\ref{mul1'})$\}$ or the equivalence between
$\{$(\ref{mul2}) and (\ref{mul3})$\},$ unless we have extra assumptions (e.g.
Tao's theorem \ (cf. Section \ref{sec:Tao}))$.$ On the other hand, if we apply
the corresponding $K-$functionals for scales, we know that (\ref{mul1}) is
equivalent to (cf. (\ref{validity}) and (\ref{char2}))%
\begin{equation}
K(t,Tf;\{L^{p}\}_{p>1})\leq cK(t,f;\{\frac{p}{p-1}L^{p}\}_{p>1}), \label{lk1}%
\end{equation}
and likewise (\ref{mul2}) is equivalent to%
\begin{equation}
K(t,Tf;\{\frac{p}{p-1}L^{p}\}_{p>1})\leq cK(t,f;\{\left(  \frac{p}%
{p-1}\right)  ^{2}L^{p}\}_{p>1}). \label{lk2}%
\end{equation}
Therefore, since (\ref{mul1}) and (\ref{mul2}) are equivalent we see that
(\ref{lk1}) and (\ref{lk2}) are equivalent. These $K-$functional estimates can
be made explicit, (cf. Examples \ref{ex:yano}, \ref{ex:enuno},\ref{ex:lag})%
\begin{align*}
K(t,Tf;\{L^{p}\}_{p>1})  &  \approx K(t,Tf;\{(L^{1},L^{\infty})_{1/p^{\prime
},\infty;K}^{\blacktriangleleft}\}_{p>1})\\
&  \approx K(t,Tf;L^{1},L^{\infty})=\int_{0}^{t}(Tf)^{\ast}(s)ds,
\end{align*}%
\begin{align*}
K(t,f;\{\frac{p}{p-1}L^{p}\}_{p>1})  &  \approx K(t,f;\{\frac{p}{p-1}%
(L^{1},L^{\infty})_{1/p^{\prime},1;J}^{\blacktriangleleft}\}_{p>1})\\
&  \approx\int_{0}^{t}f^{\ast}(s)\log\frac{t}{s}ds.
\end{align*}
Likewise,%
\begin{align*}
K(t,f;\{\left(  \frac{p}{p-1}\right)  ^{2}L^{p}\}_{p>1})  &  \approx
K(t,f;\{\left(  \frac{p}{p-1}\right)  ^{2}(L^{1},L^{\infty})_{1/p^{\prime
},1;J}^{\blacktriangleleft}\}_{p>1})\\
&  \approx\int_{0}^{t}f^{\ast}(s)(\log\frac{t}{s})^{2}ds.
\end{align*}
So, we get the equivalence of the rearrangement inequalities
\begin{equation}
\int_{0}^{t}(Tf)^{\ast}(s)ds\preceq\int_{0}^{t}f^{\ast}(s)\log\frac{t}%
{s}ds\;\text{ and }\;\int_{0}^{t}(Tf)^{\ast}(s)\log\frac{t}{s}ds\preceq
\int_{0}^{t}f^{\ast}(s)(\log\frac{t}{s})^{2}ds. \label{lk3}%
\end{equation}
Usually in the classical papers such results are described as a *gain* or
*loss* of logarithms.

Here is another example that comes from classical interpolation. Consider
informally the ultra classical situation (Calder\'{o}n's Theorem): Let $T$ be
a linear operator such that $T:L^{1}\rightarrow L^{1},$ and $T:$ $L^{\infty
}\rightarrow L^{\infty}$. These conditions are equivalent to%
\[
K(t,Tf;L^{1},L^{\infty})\preceq K(t,f;L^{1},L^{\infty})
\]
which yields%
\begin{equation}
\int_{0}^{t}(Tf)^{\ast}(s)ds\preceq\int_{0}^{t}f^{\ast}(s)ds. \label{lk4}%
\end{equation}
Comparing with (\ref{lk3}) shows the lack of logarithms on the right hand
side, reflecting that the corresponding deterioration *weight* is a constant:
$M(p)\approx1.$ Now from (\ref{lk4}) we see that, by the Calder\'{o}n-Mityagin
theorem, with constants independent of $p$%
\[
T:L^{p}\rightarrow L^{p},\;p\geq1.
\]
In particular, multiplication by $\{\frac{p}{p-1}\}$ yields%
\[
T:\{\frac{p}{p-1}L^{p}\}_{p>1}\rightarrow\{\frac{p}{p-1}L^{p}\}_{p>1},
\]
which as we have seen is equivalent\footnote{This equivalence in principle is
valid for functions in $L^{1}\cap L^{\infty}.$} now to%
\[
\int_{0}^{t}(Tf)^{\ast}(s)\log\frac{t}{s}ds\preceq\int_{0}^{t}f^{\ast}%
(s)\log\frac{t}{s}ds.
\]
So in this case there is no *gain* of logarithm, as indeed it should be, since
the assumption that $T:L^{1}\rightarrow L^{1}$ and $T:$ $L^{\infty}\rightarrow
L^{\infty}$ is essentially stronger. We also note that once explicit
inequalities are written down they can be proved by more direct methods. This
is certainly the case of rearrangement inequalities.

\begin{problem}
Give a direct proof of the equivalence of rearrangement inequalities
(\ref{lk3}).
\end{problem}

\begin{problem}
Let $T$ be a bounded operator, $T:\,L^{p}[0,1]\rightarrow L^{p}[0,1]$, $%
1<p<\infty ,$ and let $\Phi (p)=\left\Vert T\right\Vert _{p\rightarrow p}.$
Characterize in terms of $\Phi $ the functions $\varphi :\,[1,\infty
)\rightarrow \lbrack 0,\infty )$, $\varphi (1)=0,$ that make the
inequalities 
\begin{equation*}
\int_{0}^{t}(Tf)^{\ast }(s)ds\preceq \int_{0}^{t}f^{\ast }(s)\varphi ({t}/{s}%
)ds,\;\;0<t\leq 1,
\end{equation*}
\begin{equation*}
\int_{0}^{t}(Tf)^{\ast }(s)\varphi ({t}/{s})ds\preceq \int_{0}^{t}f^{\ast
}(s)(\varphi ({t}/{s}))^{2}ds,\;\;0<t\leq 1,
\end{equation*}%
equivalent for all $f\in L^{\infty }[0,1]$.  
\end{problem}

\section{Non-Commutative Calder\'{o}n Operator and Extrapolation
\label{sec: non-comm Yano}}

Let $\mathcal{N}$ be a semifinite von Neumann algebra on a Hilbert space
$\mathcal{H}$ equipped with a faithful normal semifinite trace $\tau$,
${{\mathfrak{S}}}^{p}(\mathcal{N})$ be the corresponding Schatten-von Neumann
classes, $\mathcal{M}^{1}(\mathcal{N})$ be the Matsaev ideal, $\mu=\mu(t,A)$
be the $^{\ast}$-operation in the non-commutative setting (see for the
definitions \cite{LSZ-13} and Example~\ref{non-commutative}). Moreover, let
$S$ be the Calder\'{o}n operator, defined by
\begin{equation}
Sf(t):=\frac{1}{t}\int_{0}^{t}f(s)\,ds+\int_{t}^{\infty}\frac{f(s)}%
{s}\,ds,\;\;t>0. \label{extrapola0}%
\end{equation}

Let $\mathcal{T}:\,{{\mathfrak{S}}}^{2}(\mathcal{N})\rightarrow{{\mathfrak{S}%
}}^{2}(\mathcal{N})$ be a selfadjoint contraction. Suppose that $\mathcal{T}$
admits a bounded linear extension on ${{\mathfrak{S}}}^{p}(\mathcal{N}),$ for
all $1<p\leq2$. If
\begin{equation}
\Vert\mathcal{T}\Vert_{{{\mathfrak{S}}}^{p}(\mathcal{N})\rightarrow
{{\mathfrak{S}}}^{p}(\mathcal{N})}\leq\frac{Cp}{p-1},\;\;1<p\leq
2,\label{extrapola1}%
\end{equation}
then it is shown in \cite{STZ} that with some absolute constant $C$
\begin{equation}
\frac{1}{t}\int_{0}^{t}\mu(s,\mathcal{T}(A))ds\leq C\frac{1}{t}\int_{0}%
^{t}S(\mu(\cdot,A))(s)ds,\;\;t>0,\;A\in\mathcal{M}^{1}(\mathcal{N}%
).\label{extrapola}%
\end{equation}
Let us show how this result can be obtained by extrapolation. Indeed, we will
show that (\ref{extrapola}) is the exact non-commutative analogue of
Calder\'{o}n's result, whose abstract extrapolation extension was formulated
in \cite{jm} and discussed at length in Example \ref{ex:lag}. Let us present
the details. For this purpose it will be convenient to let $Pf(t):=\frac{1}%
{t}\int_{0}^{t}f(s)ds,Qf(t)=\int_{t}^{\infty}f(s)\frac{ds}{s}.$ Then, we can
rewrite (\ref{extrapola}) as%
\begin{equation}
P(\mu(\cdot,\mathcal{T}(A))(t)\leq CP(S(\mu(\cdot,A)))(t).\label{extrapola2}%
\end{equation}
Moreover, from the definitions of $P$ and $Q$, (\ref{extrapola0}), and a
simple computation, we have that%
\begin{align*}
S &  =P+Q\\
&  =PQ\\
&  =QP.
\end{align*}
Therefore%
\[
PS=P(QP)=PQP=QPP=SP
\]
and we can rewrite (\ref{extrapola2}) as%
\begin{equation}
P(\mu(\cdot,\mathcal{T}(A))(t)\leq C(SP(\mu(\cdot,A)))(t).\label{frum3}%
\end{equation}
On the other hand, if we combine the assumption (\ref{extrapola1}) with the
fact that $\mathcal{T}$ is selfadjoint, and the duality formula
${{\mathfrak{S}}}^{p}(\mathcal{N})^{\ast}={{\mathfrak{S}}}^{p^{\prime}%
}(\mathcal{N}),$ we see that%
\begin{equation}
\Vert\mathcal{T}\Vert_{{{\mathfrak{S}}}^{p}(\mathcal{N})\rightarrow
{{\mathfrak{S}}}^{p}(\mathcal{N})}\leq\frac{Cp^{2}}{p-1},\;\;1<p<\infty
.\label{frum}%
\end{equation}
Now taking into account the classical computation of $K-$functionals for
non-commutative $L^{p}$ spaces (cf. \cite{pee}, and
Example~\ref{non-commutative})%
\begin{equation}
K(t,A,{{\mathfrak{S}}}^{1}(\mathcal{N}),{{\mathfrak{S}}}^{\infty}%
(\mathcal{N}))=\int_{0}^{t}\mu(s,A)ds\label{frum2}%
\end{equation}
yields, just like in the commutative case (cf. \cite{Mil-03}), that%
\begin{equation}
{{\mathfrak{S}}}^{p}(\mathcal{N})=({{\mathfrak{S}}}^{1}(\mathcal{N}%
),{{\mathfrak{S}}}^{\infty}(\mathcal{N}))_{1/p^{\prime},p;K}%
^{\blacktriangleleft}.\label{frum1}%
\end{equation}
Given (\ref{frum}) and (\ref{frum1}) we can apply the extrapolation theorem of
\cite{jm} (discussed extensively in Example \ref{ex:lag} above) to obtain an
absolute constant $C$ such that
\[
\frac{K(t,\mathcal{T}(A),{{\mathfrak{S}}}^{1}(\mathcal{N}),{{\mathfrak{S}}%
}^{\infty}(\mathcal{N}))}{t}\leq C\,S(\frac{K(s,A,{{\mathfrak{S}}}%
^{1}(\mathcal{N}),{{\mathfrak{S}}}^{\infty}(\mathcal{N}))}{s})(t).
\]
Finally, combining the last inequality with the formula for the $K-$functional
provided by (\ref{frum2}) yields (\ref{frum3}), as we wished to show.

Likewise, it is shown in \cite[Theorem 14 (ii)]{STZ} that if $\mathcal{T}$ is
assumed to be of weak type $(1,1),$ then one can replace the non-commutative
averaging operator $P$ by the corresponding non-commutative $^{\ast}%
$-operation. Again from the previous discussion and Example
\ref{non-commutative} we see that this last result follows from Jawerth-Milman
\cite[Proposition 5.2.2, page 50]{jm}.

In particular, our extrapolation method allows to treat more general type of
norm decays, e.g. $\sim\left(  \frac{p}{p-1}\right)  ^{\alpha},\alpha>1,$ weak
type versions, etc.

\begin{problem}
\label{prob: non-comm} In view of the previous discussion and Example
\ref{non-commutative} we are asking for the non-commutative $(\infty,\infty)$
version of the Bennett-DeVore-Sharpley theorem. In this connection we ask:
Formulate interpolation/extrapolation theorems for operators assuming
properties of their adjoints. For example, what can be said about operators
$T$ such that $T$ and $T^{\ast}$ are weak type $(1,1)?$
\end{problem}

\section{More Open Ended Problems}

\subsection{Gagliardo coordinate spaces and Extrapolation}

\begin{problem}
This project asks to incorporate the \textquotedblleft Gagliardo coordinate
spaces" (cf. \cite{Mbmo}) to extrapolation theory. Let $\theta\in
\lbrack0,1],q\in(0,\infty].$ We define the spaces
\[
(X_{1},X_{2})_{\theta,q}^{(1)}=\left\{  f\in X_{1}+X_{2}:\left\Vert
f\right\Vert _{(X_{1},X_{2})_{\theta,q}^{(1)}}<\infty\right\}  ,
\]
where%
\[
\left\Vert f\right\Vert _{(X_{1},X_{2})_{\theta,q}^{(1)}}=\left\{  \int%
_{0}^{\infty}\left(  t^{1-\theta}\left[  \frac{K(t,f;X_{1},X_{2})}%
{t}-K^{\prime}(t,f;X_{1},X_{2})\right]  \right)  ^{q}\frac{dt}{t}\right\}
^{1/q},
\]
and%
\[
(X_{1},X_{2})_{\theta,q}^{(2)}=\left\{  f\in X_{1}+X_{2}:\left\Vert
f\right\Vert _{(X_{1},X_{2})_{\theta,q}^{(2)}}<\infty\right\}  ,
\]
where%
\[
\left\Vert f\right\Vert _{(X_{1},X_{2})_{\theta,q}^{(2)}}=\left\{  \int%
_{0}^{\infty}(t^{-\theta}tK^{\prime}(t,f;X_{1},X_{2}))^{q}\frac{dt}%
{t}\right\}  ^{1/q}.
\]
The Gagliardo coordinate spaces in principle are not linear, and the
corresponding functionals, $\left\Vert f\right\Vert _{(X_{1},X_{2})_{\theta
,q}^{(i)}},i=1,2,$ are not norms. However, it turns out that, when $\theta
\in(0,1),q\in(0,\infty],$ we have, with *norm* equivalence (cf. \cite{holm},
\cite{jm1}, \cite{Mbmo}),
\begin{equation}
(X_{1},X_{2})_{\theta,q}^{(1)}=(X_{1},X_{2})_{\theta,q}^{(2)}=(X_{1}%
,X_{2})_{\theta,q}. \label{equiva}%
\end{equation}
More precisely, the \textquotedblleft norm\textquotedblright\ equivalence
depends only on $\theta,$ and $q.$ On the other hand, at the end points,
$\theta=0$ or $\theta=1,$ the resulting spaces can be very different.
\end{problem}

Let $(X_{1},X_{2})=(L^{1},L^{\infty}).$ Then, if $\theta=1,q=\infty,$ we
obtain the space introduced by Bennett-DeVore-Sharpley \cite{bedesa}%
\begin{equation}
\left\Vert f\right\Vert _{(L^{1},L^{\infty})_{1,\infty}^{(1)}}=\left\Vert
f\right\Vert _{L(\infty,\infty)}=\sup\{f^{\ast\ast}(t)-f^{\ast}(t)\},\text{
}(L^{1},L^{\infty})_{0,\infty}^{(2)}\text{\ }=L(1,\infty). \label{wehaveseen}%
\end{equation}
It was shown in \cite{bedesa} that in the case of finite measure
$L(\infty,\infty)$ is the rearrangement invariant hull of BMO. The
corresponding space that one obtains when $q<\infty,$ $L(\infty,q),$ also
makes sense, and was first introduced by Bastero-Milman-Ruiz \cite{bmr} who
showed a sharp end point for the Sobolev embedding theorem
\[
\left\Vert f\right\Vert _{L(\infty,n)}\leq c\left\Vert \nabla f\right\Vert
_{W_{n}^{1}(\mathbb{R}^{n})},\text{ }f\in C_{0}^{\infty}(\mathbb{R}^{n}).
\]
More generally these spaces play an important r\^{o}le in the theory of
Sobolev inequalities (cf. \cite{milpu}). This justifies the interest in the following

\begin{problem}
(Open Ended) We ask to incorporate the Gagliardo coordinate spaces to
Extrapolation Theory. In particular, find the constants of basic interpolation
inequalities connected with Gagliardo coordinate spaces. In this direction
some results were obtained in \cite{Mbmo}.
\end{problem}

\begin{problem}
$L(\infty,\infty)$ vs $e^{L}?$ Here is a concrete extrapolation problem
concerning the space $L(\infty,\infty).$ What extra conditions are needed to
be able to extrapolate weak type $(\infty,\infty)$ from the usual
extrapolation assumptions? More concretely, suppose that $\left\Vert
Tf\right\Vert _{L^{p}}\leq c_{p}\left\Vert f\right\Vert _{L^{p}}$ with
$c_{p}=c\frac{p^{2}}{p-1},$ for large $p,$ under what extra conditions can we
conclude that restricted to simple functions, $T:L^{\infty}\rightarrow
L(\infty,\infty)?$
\end{problem}

\subsection{Calder\'{o}n-Mityagin Scales}

The characterization of Calder\'{o}n-Mityagin pairs has been extensively
studied in the context of interpolation theory (cf. \cite{BK} and the
references therein). The concept can be extended to scales of spaces. Here we
consider only one of the simplest possible definitions (cf. \cite[page
71-72]{jm} for a more general formulation). Let $\{F_{\theta}\}_{\theta\in I}%
$, $\{G_{\theta}\}_{\theta\in I}$ be two families of interpolation functors of
exact type $\theta,$ and let $M(\theta)$ be a tempered weight. Let $\vec
{A},\vec{B}$ be pairs of mutually closed Banach spaces. We shall say that the
pair of scales $(\{F_{\theta}(\vec{A})\}_{\theta\in I},\{G_{\theta}(\vec
{B})\}_{\theta\in I})$ is a \textit{relative Calder\'{o}n-Mityagin $M-$pair of
scales}, if given $a\in\sum A_{\theta}$ and $b\in\sum B_{\theta},$ such that
\[
K(t,b;\{G_{\theta}(\vec{B})\}_{\theta\in I})\leq cK(t,a;\{M(\theta)F_{\theta
}(\vec{A})\}_{\theta\in I}),\;\;t>0,
\]
where $c>0$ is a constant independent of $t>0$, it follows that there exist an
operator $T$ and constant $C>0$ such that%
\[
T:\{M(\theta)F_{\theta}(\vec{A})\}_{\theta\in I}\overset{C}{\rightarrow
}\{G_{\theta}(\vec{B})\}_{\theta\in I},\text{ with }Ta=b.
\]

\begin{example}
(cf. \cite[Theorem 6.4]{jm}) Let $M(\theta)$ be such that $\tau(t)=\inf
_{0<\theta<1} (M(\theta)t^{\theta})$ is a $C^{2}$-function with $-t^{2}%
\tau^{\prime\prime}(t)$ quasi-concave. Let $\vec{A},\vec{B}$ be mutually
closed pairs, and let $F_{\theta}(\vec{A})=\vec{A}_{\theta,1;J}%
^{\blacktriangleleft},$ and $G_{\theta}(\vec{B})=\vec{B}_{\theta,\infty
;K}^{\blacktriangleleft}.$ Then $(\{\vec{A}_{\theta,1;J}^{\blacktriangleleft
}\}_{\theta\in I},\{\vec{B}_{\theta,\infty;K}^{\blacktriangleleft}%
\}_{\theta\in I})$ is a Calder\'{o}n-Mityagin $M-$pair of scales.
\end{example}

\begin{example}
(cf. \cite[Corollary 6.6]{jm}) Under the same assumptions as in the previous
example, suppose that $1\leq q(\theta)\leq\infty.$ Let $F_{\theta}(\vec
{A})=\vec{A}_{\theta,q(\theta);J}^{\blacktriangleleft},$ and $G_{\theta}%
(\vec{B})=\vec{B}_{\theta,q(\theta);K}^{\blacktriangleleft}.$ Then,
$(\{\vec{A}_{\theta,q(\theta);J}^{\blacktriangleleft}\}_{\theta\in I}%
,\{\vec{B}_{\theta,q(\theta);K}^{\blacktriangleleft}\}_{\theta\in I})$ is a
Calder\'{o}n-Mityagin $M-$pair of scales.
\end{example}

\begin{example}
If $F_{1/p^{\prime}}(\vec{A})=G_{1/p^{\prime}}(\vec{A})=\vec{A}_{1/p^{\prime
},p;K},$ then the previous Example implies that $(\{L^{p}\}_{p>1}%
,\{L^{p}\}_{p>1})$ is a Calder\'{o}n-Mityagin $1-$pair of scales (this is
essentially a reformulation of the classical result of Calder\'{o}n-Mityagin).
\end{example}

\begin{problem}
Let $\{F_{\theta}\}_{\theta\in I}$, $\{G_{\theta}\}_{\theta\in I}$ be two
families of interpolation functors of exact type $\theta,$ and let $M(\theta)$
be a weight. Find sharp conditions on $M\,$so that for all $\vec{A},\vec{B}$
be mutually closed pairs $(\{F_{\theta}(\vec{A})\}_{\theta\in I},\{G_{\theta
}(\vec{B})\}_{\theta\in I})$ is a Calder\'{o}n-Mityagin $M-$pair of scales.
\end{problem}

\subsection{Complex Extrapolation (open ended project)}

Given the central r\^{o}le of complex methods in Interpolation theory it is
somewhat surprising that so far there has been little progress in the
direction of developing connections between complex methods and Extrapolation
theory. The general interpolation methods introduced in \cite{cw} provide a
unification of real and complex interpolation. We now give a brief summary of
the basic definitions.

The spaces introduced in \cite{cw} are based on an extension of the concept of
\textquotedblleft lattice\textquotedblright. Let $\mathbf{Ban}$ be the class
of all Banach spaces over the complex numbers. Then we say that a mapping
$\mathcal{X}:\mathbf{Ban}\rightarrow\mathbf{Ban}$ is a pseudolattice\textit{
(}or a pseudo-$Z$-lattice)\textit{, }if it satisfies the following conditions

(i) for each $B\in\mathbf{Ban}$ the space $\mathcal{X}(B)$ consists of
$B-$valued sequences $\{b_{n}\}_{n\in\mathbb{Z}}$;

(ii) whenever $A$ is a closed subspace of $B$ it follows that $\mathcal{X}(A)$
is a closed subspace of $\mathcal{X}(B)$;

(iii) there exists a positive constant $C=C(\mathcal{X})$ such that, for all
$A,B\in\mathbf{Ban}$ and all bounded linear operators $T:A\rightarrow B$ and
every sequence $\{a_{n}\}_{n\in\mathbb{Z}}\in\mathcal{X}(A)$, the sequence
$\{Ta_{n}\}_{n\in\mathbb{Z}}\in\mathcal{X}(B)$ and satisfies the estimate
\[
\Vert\{Ta_{n}\}_{n\in\mathbb{Z}}\Vert_{\mathcal{X}(B)}\leq C(\mathcal{X})\Vert
T\Vert_{A\rightarrow B}\Vert\{a_{n}\}\Vert_{\mathcal{X}(A)};
\]

(iv)%
\[
\Vert b_{m}\Vert_{B}\leq\Vert\{b_{n}\}\Vert_{\mathcal{X}(B)}%
\]
for all $m\in\mathbb{Z}$, all $\{b_{n}\}_{n\in\mathbb{Z}}\in\mathcal{X}(B)$
and all Banach spaces $B$.

\textbf{Example of pseudo-lattices:} Lattices; the Fourier spaces $FL^{1},FC$;
$UC;$ the space of unconditionally convergent series; $WUC,$ the space of
weakly unconditionally convergent sequences. We refer to \cite{cw} for
complete details.

For each Banach pair $\vec{B}$\textbf{\ }and each pair $\mathbf{X=(\mathcal{X}%
_{0},\mathcal{X}_{1})}$ of pseudolattices, let $\mathcal{J}(\mathbf{X},\vec
{B})$ to be the space of all $B_{0}\cap B_{1}$ valued sequences $\{b_{n}%
\}_{n\in\mathbb{Z}}$ for which the sequence $\{e^{jn}b_{n}\}_{n\in\mathbb{Z}}$
is in $\mathcal{X}_{j}(B_{j})$ for $j=0,1$. This space is normed by%

\[
\Vert\{b_{n}\}_{n\in\mathbb{Z}}\Vert_{\mathcal{J}(\mathbf{X},\vec{B})}%
:=\max_{j=0,1}\Vert\{e^{jn}b_{n}\}_{n\in\mathbb{Z}}\Vert_{\mathcal{X}%
_{j}(B_{j})}.
\]

For each Banach pair $\vec{B}$ , each pair of pseudolattices $\mathbf{X}$ as
above, and each fixed $s\in\mathbb{A=}$ $\{z\in\mathbb{C}:1<\left\vert
z\right\vert <e\}$, the spaces $\vec{B}_{\mathbf{X},s}$ consist of all
elements of the form $b=\sum_{n\in\mathbb{Z}}s^{n}b_{n}$ where $\{b_{n}%
\}_{n\in\mathbb{Z}}\in\mathcal{J}(\mathbf{X},\vec{B})$, with the natural
quotient norm
\begin{equation}
\left\Vert b\right\Vert _{\vec{B}_{\mathbf{X},s}}:=\inf\left\{  \left\Vert
\{b_{n}\}_{n\in\mathbb{Z}}\right\Vert _{\mathcal{J}(\mathbf{X},\vec{B}%
)}:b=\sum_{n\in\mathbb{Z}}s^{n}b_{n}\right\}  . \label{djms}%
\end{equation}
\bigskip

\textbf{Examples: }Let $s=e^{\theta}$ for some $\theta\in(0,1).$ (i) If
$\mathbf{X=(\mathcal{X}_{0},\mathcal{X}_{1}),}$ with $\mathcal{X}%
_{0}=\mathcal{X}_{1}=\ell^{p}$, then space $\vec{B}_{\mathbf{X},s}$ coincides
with the Lions-Peetre real method space $\vec{B}_{\theta,p}=(B_{0}%
,B_{1})_{\theta,p}$; (ii) If $\mathbf{X=(\mathcal{X}_{0},\mathcal{X}_{1}),}$
with $\mathcal{X}_{0}=\mathcal{X}_{1}=FC$, then $\vec{B}_{\mathbf{X},s}$
coincides, to within equivalence of norm, with the Calder\'{o}n complex method
space $\vec{B}_{[\theta]}=[B_{0},B_{1}]_{\theta}=[\vec{B}]_{\theta}$; (iii)
Likewise, if $\mathcal{X}_{0}=\mathcal{X}_{1}=UC$, then $\vec{B}%
_{\mathbf{X},s}$ is the Peetre $\pm$ method space $\vec{B}_{<\theta>}=\langle
B_{0},B_{1}\rangle_{\theta}$; If we replace $UC$ by $WUC$, in (iii) we obtain
the Gustavsson-Peetre variant of $\langle B_{0},B_{1}\rangle_{\theta}$ which
is denoted by $\langle\vec{B},\rho_{\theta}\rangle$ (cf. \cite{jan}).

\begin{problem}
We ask to incorporate the interpolation spaces $\vec{B}_{\mathbf{X},s}$ of
\cite{cw} to Extrapolation theory. In particular, we ask for an extrapolation
version of the classical interpolation theorem for analytic families of
operators (cf. \cite{cwja1}). For a different possible connection between
complex methods and extrapolation we refer to \cite{lemp} and the references therein.
\end{problem}

\section{Appendix \label{sec:app}}

Let us start recalling some definitions from interpolation theory (cf.
\cite{BL}). The classical theory of interpolation deals with pairs of
compatible Banach spaces (\textquotedblleft Banach pairs" or simply
\textquotedblleft pairs"), $\vec{A}=(A_{0},A_{1}),$ which are contained in a
suitable larger Hausdorff topological space. We say that a Banach space $H$ is
intermediate with respect to $\vec{A},$ if $A_{0}\cap A_{1}\subset H\subset
A_{0}+A_{1}.$ Given two Banach pairs $\vec{A}$ and $\vec{B},$ let $H$ be
intermediate for the pair $\vec{A}$ and let $G$ be intermediate for the pair
$\vec{B},$ we then say that the Banach spaces $H,G,$ are interpolation spaces
with respect to the pairs $\vec{A}$ and $\vec{B}$ if any operator $T$ that is
bounded from $\vec{A}$ to $\vec{B}$ \footnote{That is $T:A_{i}\rightarrow
B_{i},i=0,1.$} also defines a bounded operator $T:H\rightarrow G.$ An
\textquotedblleft interpolation method\textquotedblright\ is a functor $F$
that assigns to each pair $\vec{A}$ an interpolation space $F(\vec{A})$ so
that for all linear operators $T$ that are bounded from $\vec{A}$ to $\vec{B}$
we have that $T:F(\vec{A})\rightarrow F(\vec{B})$ is bounded. An interpolation
method is exact if $\left\Vert T\right\Vert _{F(\vec{A})\rightarrow F(\vec
{B})}\leq\max\{\left\Vert T\right\Vert _{A_{0}\rightarrow B_{0}},\left\Vert
T\right\Vert _{A_{1}\rightarrow B_{1}}\}:=\left\Vert T\right\Vert _{\vec
{A}\rightarrow\vec{B}}.$

\subsection{The $K$ and $J$ methods of interpolation\label{apendixA:1}}

It can be argued that the most successful method of real interpolation is the
one based on using the $K-$functional of Peetre\footnote{Calder\'{o}n and his
student Oklander (cf. \cite{oklan}, \cite{Oklander}) independently also
defined $K-$functionals for Banach pairs and implemented some of the early
applications of $K-$functionals to interpolation theory. In particular, to
weak type interpolation (cf. Section \ref{sec:J&M}).}. The method explicitly
provides a penalty problem on the splitting of elements that underlies the
Lions-Peetre method of interpolation.

Recall that given a compatible pair of Banach spaces $\vec{X}=(X_{0},X_{1})$
we let, for $f\in X_{0}+X_{1},t>0,$%
\[
K(t,f;\vec{X}):=\inf_{f=f_{0}+f_{1},f_{i}\in X_{i}}\{\left\Vert f_{0}%
\right\Vert _{X_{0}}+t\left\Vert f_{0}\right\Vert _{X_{1}}\}.
\]
It follows immediately that if $T$ is a bounded operator, $T:\vec
{X}\rightarrow\vec{Y},$ then%
\[
K(t,Tf;\vec{Y})\preceq K(t,f;\vec{X}),\;\;t>0.
\]
If $\rho$ is any function norm on measurable functions on $(0,\infty)$
then\footnote{It is often more convenient to write the inequalities in terms
of decreasing functions and thus use the expression $\frac{K(t,f;\vec{X})}%
{t}.$}%
\[
\rho(K(\cdot,Tf;\vec{Y}))\preceq\rho(K(\cdot,f;\vec{X})).
\]
In particular, let $0<\theta<1,$ \ and $1\leq q\leq\infty$, and consider the
function norms
\[
\Phi_{\theta,q}(f):=\left\{
\begin{array}
[c]{cc}%
\left\{  \int_{0}^{\infty}\left(  s^{-\theta}|f(s)|\right)  ^{q}\frac{ds}%
{s}\right\}  ^{1/q} & if\text{ }q<\infty\\
\sup_{s>0}\left\{  s^{-\theta}|f(s)|\right\}  & if\text{ }q=\infty.
\end{array}
\right.
\]
The Lions-Peetre interpolation spaces $\vec{X}_{\theta,q;K}$ consist of the
elements $f\in X_{0}+X_{1},$ such that $\Vert f\Vert_{\vec{X}_{\theta,q;K}%
}<\infty,$ where
\[
\Vert f\Vert_{\vec{X}_{\theta,q;K}}:=\Phi_{\theta,q}(K(\cdot,f;\vec{X})).
\]
We normalize the norms so that the interpolation functor $\vec{X}%
\rightarrow\vec{X}_{\theta,q;K}$ is of exact type $\theta,$ and for each
Banach pair $\vec{X}$ we denote the corresponding normalized spaces by
$\vec{X}_{\theta,q;K}^{\blacktriangleleft}:$
\begin{equation}
\Vert f\Vert_{\vec{X}_{\theta,q;K}^{\blacktriangleleft}}:=(q\theta
(1-\theta))^{\frac{1}{q}}\Vert f\Vert_{\vec{X}_{\theta,q;K}}, \label{ck}%
\end{equation}
with the convention that $(q\theta(1-\theta))^{\frac{1}{q}}=1$ when
$q=\infty.$ Our convention means that%
\[
\Vert\circ\Vert_{\vec{X}_{\theta,\infty;K}^{\blacktriangleleft}}=\Vert
\circ\Vert_{\vec{X}_{\theta,\infty;K}}.
\]

There is a dual construction associated with the $J-$functional which is
defined for $g\in X_{0}\cap X_{1},t>0,$ by%
\[
J(t,g;\vec{X}):=\max\{\left\Vert g\right\Vert _{X_{0}},t\left\Vert
g\right\Vert _{X_{1}}\}.
\]
The corresponding $\vec{X}_{\theta,q;J}$ spaces consist of all $g\in
X_{0}+X_{1}$ that can be represented as
\begin{equation}
g=\int_{0}^{\infty}u(t)\frac{dt}{t}\;\;\text{ (convergence in }X_{0}+X_{1}),
\label{represetation}%
\end{equation}
for some strongly measurable function $u:(0,\infty)\rightarrow X_{0}\cap
X_{1}$ such that $\ \Phi_{\theta,q}(J(s,u(s);\vec{X}))<\infty$. We let%
\[
\Vert g\Vert_{\vec{X}_{\theta,q;J}}:=\inf\{\Phi_{\theta,q}(J(s,u(s);\vec
{X})):g=\int_{0}^{\infty}u(t)\frac{dt}{t}\}.
\]
The interpolation functor $\vec{X}\rightarrow\vec{X}_{\theta,q;J}$ can be
normalized so that it becomes of exact type $\theta.$ This is achieved using
the norms
\begin{equation}
\Vert g\Vert_{_{^{\vec{X}_{\theta,q;J}^{\blacktriangleleft}}}}:=(q^{\prime
}\theta(1-\theta))^{-1/q^{\prime}}\Vert g\Vert_{_{\vec{X}_{\theta,q;J}}},
\label{cj}%
\end{equation}
with the convention that if $q=1$ we set $(q^{\prime}\theta(1-\theta
))^{-1/q^{\prime}}=1$. Thus,%
\[
\Vert\circ\Vert_{\vec{X}_{\theta,1;J}}^{\blacktriangleleft}=\Vert\circ
\Vert_{\vec{X}_{\theta,1;J}}.
\]
We will consider also the modified spaces $\langle\vec{X}\rangle_{\theta
,q;K},$ defined by
\begin{equation}
\langle\vec{X}\rangle_{\theta,q;K}:=\{f\in X_{0}+X_{1}:\Vert f\Vert
_{\langle\vec{X}\rangle_{\theta,q;K}}:=\Phi_{\theta,q}(\chi_{(0,1)}%
K(s,f;\vec{X}))<\infty\} \label{dela}%
\end{equation}
and the similarly constructed $\langle\vec{X}\rangle_{\theta,q;J}$ spaces (cf.
\cite{ALM}).

\begin{example}
From $K(t,f;L^{1},L^{\infty})=tf^{\ast\ast}(t)=\int_{0}^{t}f^{\ast}(s)ds,$ by
Hardy's inequality and reverse Hardy's inequality for decreasing functions
(cf. \cite[Example~7]{Mil-03}), it follows that
\[
(L^{1},L^{\infty})_{1/p^{\prime},p;K}^{\blacktriangleleft}=L^{p},1<p<\infty,
\]
with constants of norm equivalence independent of $p.$
\end{example}

\subsubsection{The strong form of the fundamental Lemma\label{sec:sfl}}

Underlying the equivalence of these methods is the fundamental Lemma of
Interpolation theory (cf. \cite{BL}). The strong form of the fundamental Lemma
can be found in \cite{CJM90} and states that there exists a constant $\gamma$
such that if $\vec{X}$ is a mutually closed pair then, for all $f\in
X_{0}+X_{1},$ such that $\lim K(t,f;\vec{X})\min\{1,\frac{1}{t}\}=0$ when
$t\rightarrow0$ and $t\rightarrow\infty$, and for all $\varepsilon>0,$ there
exists $u:(0,\infty)\rightarrow X_{0}\cap X_{1},$ strongly measurable, such
that $f=\int_{0}^{\infty}u(s)\frac{ds}{s}$ and
\[
\int_{0}^{\infty}\min(1,\frac{t}{s})J(s,u(s);\vec{X})\frac{ds}{s}\leq
(\gamma+\varepsilon)K(t,f;\vec{X}).
\]
It follows that there exists a decomposition $f=\int_{0}^{\infty}u(s)\frac
{ds}{s}$ such that%
\[
K(t,f;\vec{X})\approx\int_{0}^{\infty}\min(1,\frac{t}{s})J(s,u(s);\vec
{X})\frac{ds}{s},\;\;t>0.
\]

\begin{example}
\label{exa:sfl}Let us show that for a mutually closed Banach pair $\vec{X}$ we
have%
\[
\vec{X}_{\theta,1;K}^{\blacktriangleleft}=\vec{X}_{\theta,1;J}.
\]
We shall use the elementary inequality (cf. \cite[Lemma~3.2.1]{BL})%
\[
K(t,f;\vec{X})\leq\min(1,\frac{t}{s})J(s,f;\vec{X}),\;\;t,s>0.
\]
Then, for any decomposition%
\[
f=\int_{0}^{\infty}u(s)\frac{ds}{s},
\]
we have%
\[
K(t,f;\vec{X})\leq\int_{0}^{\infty}\min(1,\frac{t}{s})J(s,u(s);\vec{X}%
)\frac{ds}{s}.
\]
Therefore,%
\begin{align*}
\int_{0}^{\infty}K(t,f;\vec{X})t^{-\theta}\frac{dt}{t}  &  \leq\int%
_{0}^{\infty}J(s,u(s);\vec{X})\int_{0}^{\infty}\min(1,\frac{t}{s})t^{-\theta
}\frac{dt}{t}\frac{ds}{s}\\
&  =\frac{1}{\theta(1-\theta)}\int_{0}^{\infty}J(s,u(s);\vec{X})s^{-\theta
}\frac{ds}{s}.
\end{align*}
Taking infimum over all such decompositions we find%
\[
\left\Vert f\right\Vert _{\vec{X}_{1/p^{\prime},1;K}^{\blacktriangleleft}}%
\leq\left\Vert f\right\Vert _{\vec{X}_{1/p^{\prime},1;J}}.
\]
On the other hand, applying the strong form of the fundamental Lemma, we can
find a special decomposition such that%
\[
\int_{0}^{\infty}\min(1,\frac{t}{s})J(s,u(s);\vec{X})\frac{ds}{s}\leq\gamma
K(t,f;\vec{X}).\label{recol1}%
\]
Consequently,
\begin{align*}
\frac{1}{\theta(1-\theta)}\int_{0}^{\infty}J(s,u(s);\vec{X})s^{-\theta}%
\frac{ds}{s}  &  =\int_{0}^{\infty}J(s,u(s);\vec{X})\int_{0}^{\infty}%
\min(1,\frac{t}{s})t^{-\theta}\frac{dt}{t}\frac{ds}{s}\\
&  \leq\gamma\int_{0}^{\infty}K(t,f;\vec{X})t^{-\theta}\frac{dt}{t}%
\end{align*}
and the desired result follows.
\end{example}

\begin{remark}
\label{re:delotro}For some problems it is useful to replace integrals by
series. We refer to \cite{CJM90} for complete details.
\end{remark}

\subsection{Extreme extrapolation functors\label{apendixA:2}}

In extrapolation the starting point are families of interpolation
spaces, and we are trying to find the end point spaces of them. Here is the basic set up. We are given compatible families $\{A_{\theta}\}_{\theta\in(0,1)}$ and
$\{B_{\theta}\}_{\theta\in(0,1)}$ of Banach spaces (in the sense that there
exist two Banach spaces $\mathbb{A}_{0}$ and $\mathbb{A}_{1}$, such that for
each $\theta\in(0,1)$, we have with continuous inclusions\footnote{In practice
these families consist of interpolation spaces, e.g. $A_{\theta}=[A_{0}%
,A_{1}]_{\theta},$ $A_{\theta}=(A_{0},A_{1})_{\theta,q(\theta)},$ and in this
case we can take $\mathbb{A}_{0}=A_{0}+A_{1},$ and $\mathbb{A}_{1}=A_{0}\cap
A_{1}.$ In particular, if the pair ($A_{0},A_{1})$ is ordered, $A_{1}\subset
A_{0}$, we can take $\mathbb{A}_{0}=A_{0},\mathbb{A}_{1}=A_{1}.$ More
generally, the construction of the $\Delta$-method makes sense if we consider
families of spaces $\{A_{\theta}\}_{\theta\in\gamma},$ where $\gamma$ is a
lattice (cf. \cite{mami}).} $\mathbb{A}_{1}\subset A_{\theta}\subset
\mathbb{A}_{0}).$ We shall say that $A$ and $B$ are extrapolation spaces with
respect to the compatible families $\{A_{\theta}\}_{\theta\in(0,1)}%
,\{B_{\theta}\}_{\theta\in(0,1)}$ if $\mathbb{A}_{1}\subset A\subset
\mathbb{A}_{0},\mathbb{A}_{1}\subset B\subset\mathbb{A}_{0},$ and every
operator $T$ that is bounded, $T:A_{\theta}\overset{1}{\rightarrow}B_{\theta}$
for all $\theta\in(0,1),$ has an extension that is bounded, $T:A\rightarrow
B.$ An extrapolation method $\mathfrak{E}$ assigns to each compatible family
an extrapolation space $\mathfrak{E}(\{A_{\theta}\}_{\theta\in(0,1)})$ with
the following interpolation property. If $T$ is an operator such that
$T:A_{\theta}\overset{1}{\rightarrow}B_{\theta},$ for each $\theta\in(0,1),$
then $T$ can be extended to $T:\mathfrak{E}(\{A_{\theta}\}_{\theta\in
(0,1)})\rightarrow\mathfrak{E}(\{B_{\theta}\}_{\theta\in(0,1)})$. Given a
compatible family $\{A_{\theta}\}_{\theta\in(0,1)},$ we let $M_{\Sigma}%
(\theta)$ denote the norm of the inclusions$\,\ A_{\theta}\subset
\mathbb{A}_{0},$ and let $M_{\Delta}(\theta)$ denote the norm of the
corresponding inclusions $\mathbb{A}_{1}\subset A_{\theta}.$ We shall say that
the family is \textit{strongly compatible} if these inclusions are uniformly
bounded, that is if $\sup_{\theta\in(0,1)}\{M_{\Sigma}(\theta),M_{\Delta
}(\theta)\}<\infty.$

Here we shall restrict ourselves to consider strongly compatible
families\footnote{Which we will also refer to as \textquotedblleft
families\textquotedblright}. There are two natural constructions of strongly
compatible families: the $\Sigma$- and $\Delta$-methods of extrapolation.
Given a family $\{A_{\theta}\}_{\theta\in(0,1)},$ the space $\Sigma
(\{A_{\theta}\}_{\theta\in(0,1)})$ consists of all the elements $x\in
\mathbb{A}_{0}$ that can be represented \ by $x=\sum_{0<\theta<1}a_{\theta}$,
$a_{\theta}\in A_{\theta}$, with $\sum_{0<\theta<1}\left\Vert a_{\theta
}\right\Vert _{A_{\theta}}<\infty$. We endow $\Sigma(\{A_{\theta}\}_{\theta
\in(0,1)})$ with the corresponding quotient norm. It is customary to write
$\Sigma_{\theta\in(0,1)}A_{\theta}$ rather than $\Sigma\mathfrak{(}%
\{A_{\theta}\}_{\theta\in(0,1)}).$

Likewise, we let form the space $\Delta\{A_{\theta}\}_{\theta\in(0,1)}$ of all
elements $x\in\cap_{\theta\in(0,1)}A_{\theta},$ such that
\[
\left\Vert x\right\Vert _{\Delta\{A_{\theta}\}_{\theta\in(0,1)}}:=\sup
_{\theta\in(0,1)}\left\Vert x\right\Vert _{A_{\theta}}<\infty.
\]
It is customary to write $\Delta_{\theta\in(0,1)}A_{\theta}$ rather than
$\Delta\{A_{\theta}\}_{\theta\in(0,1)}$. It is easy to verify that the
$\Sigma$ and $\Delta$ are extrapolation functors, and moreover they are
\textbf{exact} in the sense that, if $\mathfrak{E}$ denotes either the
$\Sigma$- or $\Delta$-method and $T:A_{\theta}\overset{1}{\rightarrow
}B_{\theta},$ $\theta\in(0,1)$, then%
\[
\left\Vert T\right\Vert _{\mathfrak{E(}\{A_{\theta}\}_{\theta\in
(0,1)})\rightarrow\mathfrak{E(}\{B_{\theta}\}_{\theta\in(0,1)})}\leq
\sup_{0<\theta<1}\{\left\Vert T\right\Vert _{A_{\theta}{\rightarrow}B_{\theta
}}\}.
\]

The $\Sigma$-method exhibits a behavior analogous to the $A_{\theta,1,J}$
spaces, while the $\Delta$-method is closely related to the $A_{\theta
,\infty,K}$-construction of classical interpolation theory. In the setting of
rearrangement invariant spaces these constructions are related to the Lorentz
spaces ($\Sigma$-method) and the Marcinkiewicz spaces ($\Delta$-method). Thus,
the $\Sigma$- and $\Delta$-methods are, in a suitable sense, extremal
extrapolation functors on the class of exact extrapolation functors. To see
this let us first show that

\begin{lemma}
An extrapolation method applied to a constant family, i.e. a family where all
its elements are equal to a given Banach space, reproduces this space. In
other words, if given a Banach space $A$ we consider the family $\{A_{\theta
}\}_{\theta\in(0,1)},$ where $A_{\theta}=A,$ for all $\theta\in(0,1),$ then if
$\mathfrak{E}$ is an extrapolation functor we have,
\[
\mathfrak{E(}\{A_{\theta}\}_{\theta\in(0,1)})=A.
\]

\end{lemma}

\begin{proof}
In fact, by definition%
\[
A=\mathbb{A}_{0}\subset\mathfrak{E(}\{A_{\theta}\}_{\theta\in(0,1)}%
)\subset\mathbb{A}_{1}=A,
\]
which forces (with equivalent norms)
\[
\mathfrak{E(}\{A_{\theta}\}_{\theta\in(0,1)})=A,
\]
as we wished to show\footnote{At this point it will be convenient to agree on
the following notation. Let $A$ be a Banach space and let $\mathfrak{E}$ be an
extrapolation functor. By abuse of notation we shall write \ $\mathfrak{E}(A)$
to denote the space $\mathfrak{E}(\{A_{\theta}\}_{\theta\in(0,1)}),$ where
$A_{\theta}=A,\theta\in(0,1).$ Thus, with this notation the previous
discussion shows that we have $\mathfrak{E}(A)=A.$}
\end{proof}

At this point it will be convenient agree that if $\mathfrak{E}$ is an
extrapolation functor and $A$ is a Banach space by abuse of language we shall
let $\mathfrak{E(}A):=\mathfrak{E(\{}A\}_{\theta\in(0,1)}).$

We shall now compare any exact extrapolation functor $\mathfrak{E}$ with the
$\sum$ and $\Delta$ functors.

\begin{lemma}
\label{elementary} Let $\mathfrak{E}$ be an exact extrapolation functor, then
for all strongly compatible families, $\{A_{\theta}\}_{\theta\in(0,1)}$ we
have
\[
\Delta(\{A_{\theta}\}_{\theta\in(0,1)})\overset{1}{\subset}\mathfrak{E(}%
\{A_{\theta}\}_{\theta\in(0,1)})\overset{1}{\subset}\sum(\{A_{\theta
}\}_{\theta\in(0,1)}).
\]

\end{lemma}

\begin{proof}
Let $\{A_{\theta}\}_{\theta\in(0,1)}$ be a strongly compatible family. Let
$\theta_{0}\in(0,1).$ Then since any $a\in A_{\theta_{0}},$ can be represented
by a sum $a=\sum_{\theta\in(0,1)}a_{\theta},$ where all the terms are zero
except for $a_{\theta_{0}}=a,$ we see that%
\begin{equation}
A_{\theta_{0}}\overset{1}{\subset}\sum_{\theta\in(0,1)}A_{\theta},\text{ for
all }\theta_{0}\in(0,1). \label{intro1}%
\end{equation}
Consider the family of strongly compatible spaces that is defined for all
$\theta\in(0,1)$ by $B_{\theta}=\sum_{\theta\in(0,1)}A_{\theta},$ then
applying the extrapolation functor $\mathfrak{E}$ to (\ref{intro1}) and Lemma
\ref{elementary} yield%
\[
\mathfrak{E(}\{A_{\theta}\}_{\theta\in(0,1)})\overset{1}{\subset}%
\mathfrak{E(}\{B_{\theta}\}_{\theta\in(0,1)})=\sum_{\theta\in(0,1)}A_{\theta
}.
\]
Likewise, since
\[
\Delta_{\theta\in(0,1)}A_{\theta}\overset{1}{\subset}A_{\theta},\text{ for all
}\theta\in(0,1),
\]
it follows that for any exact extrapolation functor $\mathfrak{E}$
\[
\Delta_{\theta\in(0,1)}A_{\theta}=\mathfrak{E(}\Delta_{\theta\in
(0,1)}A_{\theta})\subset\mathfrak{E(}\{A_{\theta}\}_{\theta\in(0,1)}),
\]
as we wished to show.
\end{proof}

\subsection{Abstract extrapolation methods\label{sec:extrapol}}

The $\Sigma$- and $\Delta$-methods are part of more general families of extrapolation functors that were{\ introduced in \cite{As2003} (cf. also \cite{karamixi}), and then were studied in \cite{As05,AL2009,AL,AL2017}. Let
$F$ be a Banach function lattice on the interval $[0,1]$ (with respect to the
usual Lebesgue measure). A given family $\{A_{\theta}\}_{\theta\in(0,1)}$ of
compatible Banach spaces, we define the Banach space $\mathbf{F}(\{A_{\theta
}\}_{\theta\in(0,1)})$, consisting of all $a\in$}$\cap_{\theta\in
(0,1)}A_{\theta}${\ such that the function $\xi_{a}(\theta):={\Vert a\Vert
}_{A_{\theta}}$ defined on $(0,1)$ belongs to $F$, endowed with the norm
$\Vert a\Vert:={\left\Vert \,\xi_{a}\,\right\Vert }_{F}$. In particular, if
$F=L^{\infty}[0,1]$, we arrive at the definition of the $\Delta$-functor. In
analogous way one can define a family of extrapolation functors generalizing
the $\Sigma$-functor (see the definition of the $\vec{A}_{\xi,q,G}^{J}$ spaces
in \cite{ALM,asly}).
}

\end{document}